\documentclass[a4paper,oneside,english]{amsart}
\usepackage[T1]{fontenc}
\usepackage[utf8]{inputenc}
\usepackage{babel}
\usepackage{mathtools}
\usepackage{bm}
\usepackage{amstext}
\usepackage{amsthm}
\usepackage{amssymb}
\usepackage{esint}
\usepackage[unicode=true,pdfusetitle,
 bookmarks=true,bookmarksnumbered=false,bookmarksopen=false,
 breaklinks=false,pdfborder={0 0 1},backref=false,colorlinks=false]
 {hyperref}

\makeatletter

\pdfpageheight\paperheight
\pdfpagewidth\paperwidth

\numberwithin{equation}{section}
\theoremstyle{plain}
\newtheorem{thm}{\protect\theoremname}[section]
\theoremstyle{remark}
\newtheorem{rem}[thm]{\protect\remarkname}
\theoremstyle{plain}
\newtheorem{prop}[thm]{\protect\propositionname}
\theoremstyle{plain}
\newtheorem{lem}[thm]{\protect\lemmaname}
\theoremstyle{plain}
\newtheorem{cor}[thm]{\protect\corollaryname}

\makeatother

\providecommand{\corollaryname}{Corollary}
\providecommand{\lemmaname}{Lemma}
\providecommand{\propositionname}{Proposition}
\providecommand{\remarkname}{Remark}
\providecommand{\theoremname}{Theorem}

\begin{document}
\global\long\def\bbC{\mathbb{C}}%
\global\long\def\bbN{\mathbb{N}}%
\global\long\def\bbQ{\mathbb{Q}}%
\global\long\def\bbR{\mathbb{R}}%
\global\long\def\bbS{\mathbb{S}}%
\global\long\def\bbZ{\mathbb{Z}}%

\global\long\def\bfD{{\bf D}}%

\global\long\def\calE{\mathcal{E}}%
\global\long\def\calH{\mathcal{H}}%
\global\long\def\calK{\mathcal{K}}%
\global\long\def\calL{\mathcal{L}}%
\global\long\def\calN{\mathcal{N}}%
\global\long\def\calU{\mathcal{U}}%
\global\long\def\calO{\mathcal{O}}%
\global\long\def\calP{\mathcal{P}}%
\global\long\def\calZ{\mathcal{Z}}%

\global\long\def\frka{\mathfrak{a}}%
\global\long\def\frkb{\mathfrak{b}}%
\global\long\def\frkc{\mathfrak{c}}%
\global\long\def\frkn{\mathfrak{n}}%
\global\long\def\frks{\mathfrak{s}}%
\global\long\def\frku{\mathfrak{u}}%

\global\long\def\eps{\epsilon}%
\global\long\def\gmm{\gamma}%
\global\long\def\tht{\theta}%
\global\long\def\lmb{\lambda}%
\global\long\def\Lmb{\Lambda}%

\global\long\def\rd{\partial}%
\global\long\def\aleq{\lesssim}%
\global\long\def\ageq{\gtrsim}%

\global\long\def\peq{\mathrel{\phantom{=}}}%
\global\long\def\To{\longrightarrow}%
\global\long\def\weakto{\rightharpoonup}%
\global\long\def\embed{\hookrightarrow}%
\global\long\def\Re{\mathrm{Re}}%
\global\long\def\Im{\mathrm{Im}}%
\global\long\def\chf{\mathbf{1}}%
\global\long\def\td#1{\widetilde{#1}}%
\global\long\def\br#1{\overline{#1}}%
\global\long\def\ul#1{\underline{#1}}%
\global\long\def\wh#1{\widehat{#1}}%
\global\long\def\tint#1#2{{\textstyle \int_{#1}^{#2}}}%
\global\long\def\tsum#1#2{{\textstyle \sum_{#1}^{#2}}}%

\global\long\def\RR{\mathrm{RR}}%
\global\long\def\inn{\mathrm{in}}%
\global\long\def\out{\mathrm{out}}%
\global\long\def\rough{\mathrm{rough}}%
\global\long\def\conn{\mathrm{conn}}%
\global\long\def\match{\mathrm{match}}%

\title[Sharp blow-up rate for stable blow-up]{Sharp universal rate for stable blow-up of corotational wave maps}
\author{Kihyun Kim}
\email{khyun@ihes.fr}
\address{IHES, 35 route de Chartres, Bures-sur-Yvette 91440, France}
\begin{abstract}
We consider the energy-critical (corotational) 1-equivariant wave
maps into the two-sphere. By the seminal work \cite{RaphaelRodnianski2012Publ.Math.}
of Raphaël and Rodnianski, there is an open set of initial data whose
forward-in-time development blows up in finite time with the blow-up
rate $\lmb(t)=(T-t)e^{-\sqrt{|\log(T-t)|}+O(1)}$. In this paper,
we show that this $e^{O(1)}$-factor in fact converges to the universal
constant $2e^{-1}$, and hence these solutions contract at the \emph{universal
rate} $\lmb(t)=2e^{-1}(T-t)e^{-\sqrt{|\log(T-t)|}}(1+o_{t\to T}(1))$.
Our proof is inspired by recent works on type-II blow-up dynamics
for parabolic equations. The key improvement is in the construction
of an \emph{explicit} invariant subspace decomposition for the linearized
operator perturbed by the scaling generator in the dispersive case,
from which we obtain a more precise ODE system determining $\lmb(t)$.
\end{abstract}

\maketitle
\tableofcontents{}

\section{Introduction}

\subsection{Equivariant wave maps}

We consider the energy-critical wave maps on $\bbR^{1+2}$ with the
$\bbS^{2}$-target. These are defined by formal critical points to
the action 
\begin{equation}
\calL(\phi)=\frac{1}{2}\int_{\bbR^{1+2}}(-|\rd_{t}\phi|^{2}+|\nabla\phi|^{2})dxdt.\label{eq:WMaction}
\end{equation}
The Euler--Lagrange equation associated with the action \eqref{eq:WMaction},
which is 
\begin{equation}
\rd_{tt}\phi-\Delta\phi=(-|\rd_{t}\phi|^{2}+|\nabla\phi|^{2})\phi\label{eq:WM-nonradial}
\end{equation}
in our case, is called the \emph{wave maps equation} from $\bbR^{1+2}$
into $\bbS^{2}$. The model under consideration is also known as the
$O(3)$ sigma model on the plane in the physics literature. It is
one of the simplest nonlinear scalar field models that admit topological
solitons given by harmonic maps. A particular interest is drawn from
the fact that harmonic maps come from solutions to the first-order
Bogomol'nyi equation \cite{Bogomolnyi1976}, manifesting the self-dual
structure of the equation. Mathematically, the wave maps equation
is a natural geometric generalization of the free wave equation whose
nonlinear structure is tied to the geometry of the target manifold.
We refer to \cite{MantonSutcliffe,GebaGrillakisWMbook,ShatahStruwe1998book}
for a nice introduction to the subject.

The action \eqref{eq:WMaction} is invariant under the symmetries
of both the Minkowski space and the target space $\bbS^{2}$. In particular,
from the time-translation symmetry, the \emph{energy} functional is
conserved under the flow \eqref{eq:WM-nonradial}:
\[
E(\phi,\rd_{t}\phi)=\frac{1}{2}\int_{\bbR^{2}}(|\rd_{t}\phi|^{2}+|\nabla\phi|^{2})dx
\]
Moreover, \eqref{eq:WM-nonradial} is \emph{scaling invariant}: if
$\phi(t,x)$ is a solution to \eqref{eq:WM-nonradial}, then so is
$\phi_{\lmb}(t,x)=\phi(\frac{t}{\lmb},\frac{x}{\lmb})$. \eqref{eq:WM-nonradial}
is called \emph{energy-critical} because this scaling preserves the
energy.

In this paper, we consider \eqref{eq:WM-nonradial} within \emph{1-equivariance}.
This means that we restrict to solutions of the form 
\begin{equation}
\phi(t,r,\theta)=(\sin u(t,r)\cos\theta,\sin u(t,r)\sin\theta,\cos u(t,r))\in\bbS^{2},\label{eq:def-1-equiv}
\end{equation}
where $(r,\theta)$ denote polar coordinates on $\bbR^{2}$ and $u=u(t,r)$
is the new unknown function taking values in $\bbR$. The wave maps
equation \eqref{eq:WM-nonradial} then reads 
\begin{equation}
\rd_{tt}u-\rd_{rr}u-\frac{1}{r}\rd_{r}u+\frac{\sin(2u)}{2r^{2}}=0.\label{eq:WM}
\end{equation}
We note that one can also consider \emph{$k$-equivariant maps}, where
$k\in\bbN$ and the map $\phi$ takes the form \eqref{eq:def-1-equiv}
with the two $\theta$s in the right hand side being replaced by $k\theta$.
The equation \eqref{eq:WM} is then changed by multiplying $k^{2}$
to the last term.

We will use the first-order formulation of \eqref{eq:WM}. In other
words, we consider the $\bbR^{2}$-valued function $\bm{u}=\bm{u}(t,r)$
defined by 
\[
\bm{u}\coloneqq\begin{bmatrix}u\\
\dot{u}
\end{bmatrix}\coloneqq\begin{bmatrix}u\\
\rd_{t}u
\end{bmatrix}
\]
and rewrite the equation \eqref{eq:WM} as 
\begin{equation}
\rd_{t}\bm{u}=\rd_{t}\begin{bmatrix}u\\
\dot{u}
\end{bmatrix}=\begin{bmatrix}\dot{u}\\
\rd_{rr}u+\frac{1}{r}\rd_{r}u-\frac{\sin(2u)}{2r^{2}}
\end{bmatrix}.\label{eq:u-vec-eqn}
\end{equation}
The energy functional reads 
\[
E(\bm{u})=2\pi\int_{0}^{\infty}\frac{1}{2}\Big(|\dot{u}|^{2}+|\rd_{r}u|^{2}+\frac{\sin^{2}u}{r^{2}}\Big)rdr.
\]

The equation is known to be well-posed in the \emph{energy space}
$\calE$, defined by the set of $1$-equivariant maps with finite
energy. For any $\bm{u}\in\calE$ both limits of $u(r)$ as $r\to+\infty$
and $r\to0$ exist and take values in $\pi\bbZ$. Thus the space $\calE$
is a disjoint union of its connected components $\calE_{\ell,m}$
for $\ell,m\in\bbZ$, whose elements $\bm{u}$ additionally satisfy
$\lim_{r\to0}u(r)=\ell\pi$ and $\lim_{r\to+\infty}u(r)=m\pi$. The
dynamics of \eqref{eq:WM} for finite energy solutions are then considered
separately on each component $\calE_{\ell,m}$. Note that we may always
assume $\ell=0$ because \eqref{eq:WM} is invariant under the transforms
$u\mapsto u+\pi$ and $u\mapsto-u$.

In the energy space $\calE$, there exists a nontrivial static solution
(i.e., nonconstant harmonic map)
\[
Q(r)\coloneqq2\arctan(r),
\]
which is unique up to sign and addition of a multiple of $\pi$. $\bm{Q}\coloneqq(Q,0)$
belongs to $\calE_{0,1}$, has energy 
\[
E(\bm{Q})=4\pi,
\]
and is characterized (up to scaling symmetry) as the energy minimizer
in the class $\calE_{0,1}$. Moreover, $Q$ is the \emph{ground state
without symmetry} (the corresponding map $\phi$ being the stereographic
projection) in the sense that $Q$ has the least energy among nontrivial
harmonic maps. Within $k$-equivariance for general $k\in\bbN$, the
ground state is given by 
\[
Q^{(k)}(r)\coloneqq2\arctan(r^{k})\qquad\text{with}\qquad E(\bm{Q}^{(k)})=4\pi k.
\]

The local-in-time Cauchy problem of \eqref{eq:WM-nonradial} is by
now well-understood. The local well-posedness for regular solutions,
being a semilinear problem, is classical. Through the efforts of many
authors, the regularity threshold for local well-posedness has been
pushed down to the scaling critical regularity; see for example \cite{KlainermanMachedon1993CPAM,KlainermanMachedon1995Duke,Tao2001CMP,Krieger2004CMP,Tataru2005AJM}.
We refer to \cite{GebaGrillakisWMbook} and references therein for
more discussions on the local theory of wave maps.

Moreover, there have been spectacular developments in the description
of the long-term dynamics for wave maps. Earlier works include the
global regularity results by Christodoulou--Tahvildar-Zadeh \cite{ChristodoulouTahvildarZadeh1993CPAM},
Shatah--Tahvildar-Zadeh \cite{ShatahTahvildarZadeh1992CPAM,ShatahTahvildarZadeh1994CPAM},
and the bubbling result of Struwe \cite{Struwe2003CPAM}, under symmetry
conditions. These works also say that, if a (symmetric) wave map develops
a singularity in finite time then (i) the energy does not concentrate
at the backward lightcone and (ii) the solution bubbles off a nontrivial
harmonic map as approaching to the blow-up time. The latter fact suggests
the following \emph{threshold conjecture}: if a wave map has energy
less than the ground state energy, then it exists globally (and scatters).
Note that the threshold is taken to be $+\infty$ if there are no
nontrivial harmonic maps, and the threshold for the $\bbS^{2}$-target
is $E(\bm{Q})=4\pi$. Even \emph{without symmetry}, Krieger--Schlag
\cite{KriegerSchlagWMbook}, Sterbenz--Tataru \cite{SterbenzTataru2010CMP2,SterbenzTataru2010CMP1},
and Tao \cite{TaoWM3-7} established the threshold conjecture for
reasonable targets. (See also \cite{CoteKenigMerle2008CMP} for equivariant
data.) For the $\bbS^{2}$-target, $4\pi$ is the real threshold as
demonstrated by the finite-time blow-up constructions with energy
arbitrarily close to $4\pi$, due to Krieger--Schlag--Tataru \cite{KriegerSchlagTataru2008Invent}
and Raphaël--Rodnianski \cite{RaphaelRodnianski2012Publ.Math.}.
We will come back to these blow-up solutions in more details.

A further interesting observation in the threshold conjecture is that,
if one restricts to \emph{degree zero} wave maps (this corresponds
to $\bm{u}\in\calE_{0,0}$ in our equivariance reduction) then the
real threshold is the twice of the energy of the ground state ($2E(\bm{Q})=8\pi$
for the $\bbS^{2}$-target) due to topological reasons; see \cite{CoteKenigLawrieSchlag2015AJM1,LawrieOh2016CMP}.
For the $\bbS^{2}$-target within $k$-equivariance, the complete
classification of degree zero wave maps with the threshold energy\footnote{Within $k$-equivariance, the threshold energy is $8\pi k$ because
$\bm{Q}^{(k)}$ is the ground state with energy $4\pi k$.} $8\pi k$ is provided by Jendrej, Lawrie, and Rodriguez \cite{Jendrej2019AJM,JendrejLawrie2018Invent,JendrejLawrie2020arXiv1,JendrejLawrie2020arXiv2,Rodriguez2021APDE}.
These threshold wave maps can develop a singularity only by forming
a pure two-bubble solution, which exists globally in time for $k\geq2$
and blows up in finite time for $k=1$.

Remarkably, recent works by Jendrej--Lawrie \cite{JendrejLawrie2021arXiv}
(for all $k\geq1$) and Duyckaerts--Kenig--Martel--Merle \cite{DuyckaertsKenigMartelMerle2021arXiv}
(for $k=1$ and 4D radial critical wave equation (NLW)) established
 \emph{soliton resolution for equivariant wave maps into the $\bbS^{2}$-target}:
any finite energy equivariant wave maps asymptotically decompose into
the sum of decoupled (in scales) harmonic maps and a radiation. In
the context of equivariant wave maps, some prior works to the resolution
are \cite{Cote2015CPAM,CoteKenigLawrieSchlag2015AJM1,CoteKenigLawrieSchlag2015AJM2,JiaKenig2017AJM},
where deep insights from the works \cite{DuyckaertsKenigMerle2011JEMS,DuyckaertsKenigMerle2012GAFA,DuyckaertsKenigMerle2013CambJMath}
on the critical NLW play a crucial role. See also \cite{DuyckaertsJiaKenigMerle2017GAFA,DuyckaertsJiaKenigMerle2018IMRN,Grinis2017CMP}
for nonradial results in this direction. We refer to \cite{DuyckaertsKenigMerle2013CambJMath,DuyckaertsKenigMere2019arXiv1,DuyckaertsKenigMartelMerle2021arXiv,CollotDuyckaertsKenigMere2022arXiv}
for soliton resolution for the radial critical wave equations in various
dimensions.

In this paper, we focus on the singularity formation via one bubble.
More precisely, we consider finite energy blow-up solutions $\bm{u}(t,r)$
which decompose as 
\begin{equation}
u(t)-Q_{\lmb(t)}^{(k)}\to z^{\ast}\qquad\text{as }t\to T,\label{eq:intro-decomp}
\end{equation}
where $Q_{\lmb(t)}^{(k)}=Q^{(k)}(\cdot/\lmb(t))$, $k\geq1$, $T\in(0,+\infty]$,
and $z^{\ast}=z^{\ast}(r)$ is the asymptotic profile (or radiation).
The convergence in \eqref{eq:intro-decomp} indeed holds in the energy
topology, but it suffices to consider much weaker convergences at
this moment.

The first rigorous constructions of one bubble blow-up solutions are
due to Krieger--Schlag--Tataru \cite{KriegerSchlagTataru2008Invent}
(for $k=1$), Rodnianski--Sterbenz \cite{RodnianskiSterbenz2010Ann.Math.}
(for $k\geq4$) and Raphaël--Rodnianski \cite{RaphaelRodnianski2012Publ.Math.}
(for all $k\geq1$). However, the blow-up solutions in \cite{KriegerSchlagTataru2008Invent}
have quite different characters from those of \cite{RodnianskiSterbenz2010Ann.Math.,RaphaelRodnianski2012Publ.Math.}.
In \cite{KriegerSchlagTataru2008Invent} (and \cite{GaoKrieger2015CPAA}),
the authors construct $1$-equivariant finite-time blow-up solutions
with the blow-up rates 
\begin{equation}
\lmb(t)=(T-t)^{1+\nu}\label{eq:KST-rate-intro}
\end{equation}
for any $\nu>0$ via the method of backward construction. These solutions
have very limited regularity at the backward light cone $r=|T-t|$
especially for small values of $\nu$. As the method heavily uses
the fact that the linearized operator (see \eqref{eq:def-H} below)
when $k=1$ has the zero resonance 
\begin{equation}
\Lmb Q(r)=\frac{2r}{1+r^{2}},\label{eq:LmbQ-formula}
\end{equation}
it seems difficult to extend the construction for higher $k$ (see
however \cite{KriegerSchlagTataru2009AdvMath}). Recently, for small
values of $\nu$, Krieger, Miao, and Schlag \cite{KriegerMiao2020Duke,KriegerMiaoSchlag2020arXiv}
proved the stability of these solutions under smooth non-equivariant(!)
perturbations supported inside the light cone, so that the perturbation
still preserves the shock at the light cone. Moreover, Pillai \cite{Pillai2019arXiv,Pillai2020arXiv}
extended the approach to global-in-time solutions and constructed
infinite-time blow-up (and also oscillating or relaxing) solutions
for $k=1$.

On the other hand, the authors in \cite{RodnianskiSterbenz2010Ann.Math.,RaphaelRodnianski2012Publ.Math.}
describe a \emph{stable} finite-time blow-up regime via the method
of forward construction. More precisely, there is an \emph{open} set
(in $H^{2}\times H^{1}$ topology) of initial data within $k$-equivariance
for all $k\geq1$, containing \emph{smooth} finite energy initial
data, such that forward-in-time maximal solutions starting from this
set blow up in finite time in a universal regime: 
\begin{equation}
\lmb(t)=\begin{cases}
c_{k}(T-t)|\log(T-t)|^{-\frac{1}{2k-2}}(1+o_{t\to T}(1)) & \text{if }k\geq2,\\
(T-t)e^{-\sqrt{|\log(T-t)|}+O(1)} & \text{if }k=1,
\end{cases}\label{eq:blow-rate-intr}
\end{equation}
where $c_{k}$ is some universal constant depending only on $k$.
Key features of the proof are approximation of the blow-up dynamics
by a finite-dimensional dynamics of well-prepared blow-up profiles
and the forward-in-time control of remainders using monotonicity (more
precisely, \emph{repulsivity}).

Let us finally mention a recent work \cite{JendrejLawrieRodriguez2019arXiv}
of Jendrej, Lawrie, and Rodriguez, who investigated the relation between
the blow-up speed $\lmb(t)$ and the asymptotic profile $z^{\ast}(r)$.

\subsection{\label{subsec:Main-results}Main results}

The goal of this paper is to refine the description of the finite-time
blow-up solutions constructed in \cite{RaphaelRodnianski2012Publ.Math.}
for $k=1$.

Before we recall this result, we first give the definition of the
function space $\bm{\calH}_{Q}^{2}$, where the initial data of these
smooth blow-up solutions belong to. Roughly speaking, $\bm{\calH}_{Q}^{2}$
will look like $(\dot{H}^{2}\times\dot{H}^{1})\cap\calE_{0,1}$. We
define the function spaces $\dot{\calH}_{1}^{2}$ and $\dot{H}_{1}^{1}$
with the norms:
\begin{align*}
\|f\|_{\dot{\calH}_{1}^{2}}^{2} & \coloneqq\|\rd_{yy}f\|_{L^{2}}^{2}+\Big\|\frac{1}{y\langle\log y\rangle}|f|_{-1}\Big\|_{L^{2}}^{2},\\
\|g\|_{\dot{H}_{1}^{1}}^{2} & \coloneqq\|\rd_{y}g\|_{L^{2}}^{2}+\Big\|\frac{1}{y}g\Big\|_{L^{2}}^{2}.
\end{align*}
The affine space $\bm{\calH}_{Q}^{2}$ is then defined by\footnote{We note that the space $\bm{\calH}_{Q}^{2}$ defined here and the
space $\calH_{a}^{2}$ used in \cite[(1.19)]{RaphaelRodnianski2012Publ.Math.}
are the same. Indeed, $\|\rd_{yy}f\|_{L^{2}}^{2}+\|\chf_{(0,1]}\frac{1}{y}(\rd_{y}f-\frac{f}{y})\|_{L^{2}}^{2}$
(plus some subcoercive term $\|\chf_{y\sim1}f\|_{L^{2}}^{2}$) of
$\calH_{a}^{2}$ can control $\|\frac{1}{y\langle\log y\rangle}|f|_{-1}\|_{L^{2}}^{2}$
of $\dot{\calH}_{1}^{2}$. Conversely, for $f\in\dot{\calH}_{1}^{2}$
(note that $f\equiv1$ is not allowed), $\|\rd_{yy}f\|_{L^{2}}^{2}$
can control $\|\chf_{(0,1]}\frac{1}{y}(\rd_{y}f-\frac{f}{y})\|_{L^{2}}^{2}$
of $\calH_{a}^{2}$. See for example the proof of \eqref{eq:loc-subcoer-1}
of this paper.}
\begin{equation}
\bm{\calH}_{Q}^{2}=\bm{Q}+\bm{\calH}^{2},\qquad\text{where}\quad\bm{\calH}^{2}=(\dot{\calH}_{1}^{2}\cap\dot{H}_{1}^{1})\times H_{1}^{1}.\label{eq:def-H2Q}
\end{equation}
We are now ready to recall the result of \cite{RaphaelRodnianski2012Publ.Math.}
for $k=1$.
\begin{thm}[Raphaël--Rodnianski blow-up solutions for $k=1$ \cite{RaphaelRodnianski2012Publ.Math.}]
\label{thm:RR-blow-up-intro}There exists an open set $\bm{\calO}$
in $\bm{\calH}_{Q}^{2}$ such that for any $\bm{u}_{0}=(u_{0},\dot{u}_{0})\in\bm{\calO}$
its forward-in-time solution $\bm{u}=(u,\dot{u})$ to \eqref{eq:WM}
satisfies the following:
\begin{itemize}
\item (Finite-time blow-up) $\bm{u}$ blows up in finite time $T=T(\bm{u}_{0})\in(0,\infty)$;
\item (Description of the blow-up) There exist $\lmb(t)\in C^{1}([0,T),\bbR_{+})$
and $\bm{u}^{\ast}=(u^{\ast},\dot{u}^{\ast})\in\dot{H}_{1}^{1}\times L^{2}$
such that 
\[
\bm{u}(t)-\bm{Q}_{\lmb(t)}\to\bm{u}^{\ast}\text{ in }\dot{H}_{1}^{1}\times L^{2}
\]
as $t\to T^{-}$ with 
\begin{equation}
\lmb(t)=(T-t)e^{-\sqrt{|\log(T-t)|}+O(1)}.\label{eq:RR-blow-up-rate-intro}
\end{equation}
\item (Additional properties) $\bm{u}$ also satisfies the properties in
Proposition~\ref{prop:RR-blow-up-original}.
\end{itemize}
\end{thm}

Note that the blow-up rate \eqref{eq:RR-blow-up-rate-intro} has a
room of $O(1)$-freedom. The question of whether this $O(1)$-freedom
is present or not has been left open, though the numerical simulations
in Ovchinnikov--Sigal \cite{OvchinnikovSigal2011PhysD} suggest the
non-existence of such $O(1)$-freedom and the convergence of $e^{O(1)}$
to some universal constant. Our main result provides a rigorous proof
of this \emph{universality} and \emph{computes} this value precisely.
\begin{thm}[Sharp universal blow-up rates]
\label{thm:MainThm}Let $\bm{u}$ be a finite-time blow-up solution
as in Theorem~\ref{thm:RR-blow-up-intro}. Then, we have 
\begin{equation}
\lmb(t)=2e^{-1}(T-t)e^{-\sqrt{|\log(T-t)|}}(1+o_{t\to T}(1)).\label{eq:sharp-blow-up-rate}
\end{equation}
\end{thm}

\begin{rem}[Sharp universal blow-up rate]
Our result is inspired by recent works \cite{CollotGhoulMasmoudiNguyen2019arXiv1,CollotGhoulMasmoudiNguyen2019arXiv2}
by Collot--Ghoul--Masmoudi--Nguyen on the 2D Keller-Segel system,
who applied the eigenfunction expansion method introduced by Hadžić
and Raphaël \cite{HadzicRaphael2019JEMS} to refine the blow-up construction
in Raphaël--Schweyer \cite{RaphaelSchweyer2014MathAnn}. They in
particular obtained the sharp asymptotics of the blow-up rate, proved
the stability of the blow-up under nonradial perturbations, and also
constructed other blow-up regimes that are conditionally stable (of
finite codimension).

Our universal constant for the blow-up rate, $2e^{-1}\approx0.736$,
does not match with the value $\sqrt{0.146}\approx0.382$ obtained
numerically in \cite{OvchinnikovSigal2011PhysD}. $2e^{-1}$ is close
to the twice of this numerical value.
\end{rem}

\begin{rem}[Comments on the proof]
The key ingredient of the proof of Theorem~\ref{thm:MainThm} is
the extension of (some parts of) the approach of \cite{HadzicRaphael2019JEMS,CollotMerleRaphael2020JAMS,CollotGhoulMasmoudiNguyen2019arXiv1,CollotGhoulMasmoudiNguyen2019arXiv2}
on type-II blow-up problems to the dispersive setting. More precisely,
we will obtain an \emph{explicit} formal invariant subspace decomposition
(or, partial diagonalization) for the linearized operator $\mathbf{M}_{\nu}$
(see \eqref{eq:def-M_nu} below) arising in the self-similar coordinates.
As a consequence, we deduce Theorem~\ref{thm:MainThm} from Theorem~\ref{thm:RR-blow-up-intro}.
The part that could not be extended in this paper is the dissipativity
of $\mathbf{M}_{\nu}$ with respect to our invariant subspace decomposition.
See Remark \ref{rem:on-dissip} for more discussions. This is the
reason why our blow-up analysis is built upon the result of \cite{RaphaelRodnianski2012Publ.Math.}.
However, to put it differently, our work demonstrates that the approach
of \cite{RaphaelRodnianski2012Publ.Math.}  combined with the invariant
subspace decomposition for $\mathbf{M}_{\nu}$ yields (i) the existence
of stable blow-up regime for \eqref{eq:WM} and (ii) the sharp universality
of the blow-up rate for such blow-up solutions.
\end{rem}

\begin{rem}[On Krieger--Schlag--Tataru blow-up solutions]
As mentioned above, the finite-time blow-up solutions considered
here arise from smooth initial data. This is also reflected in our
analysis of the linear operator $\mathbf{M}_{\nu}$, as our blow-up
regime is derived from determining some smooth eigenpairs $(\lmb_{j},\bm{\varphi}_{j})$
to $\mathbf{M}_{\nu}$ and studying the dynamics of the eigenfunctions
$\bm{\varphi}_{j}$ under \eqref{eq:WM} (more precisely, the dynamics
of $\bm{P}$; see \eqref{eq:def-mod-prof}). The same operator also
appears in the blow-up construction of \cite{KriegerSchlagTataru2008Invent}.
However, the latter construction involves approximate formal inversions
of the operator $\mathbf{M}_{\nu}-\lmb$ with various choices of $\lmb<0$
depending on the prescribed blow-up rate \eqref{eq:KST-rate-intro}
and hence results in non-smooth solution ansatz at the light cone
$r=|T-t|$ (See for example the behavior of the fundamental solutions
$h_{1}$ and $h_{2}$ at $z=0$ in Section~\ref{subsec:Outer-eigenfunctions}
of this paper).
\end{rem}

\begin{rem}[Extension to weaker equivariance and rotational instability]
The equivariant symmetry used in this paper is a strong restriction;
$\phi(t,r,\theta)$ with fixed $\theta$ is confined in one great
circle of $\bbS^{2}$. Thus one may also consider the dynamics under
a weaker version of equivariant symmetry, which only requires the
property $\phi(t,r,\theta)=R(\theta)\phi(t,r,0)$ where $R(\theta)$
is the rotation on $\bbS^{2}$ by the angle $\theta$ around the $z$-axis.
It is believed that the blow-up regime considered here is no longer
stable under this symmetry. The blow-up is expected to be stable under
codimension-1 perturbations and rather exhibit the \emph{rotational
instability}: solutions stop concentrating at small but nonzero scale,
take a quick rotation by the angle $\pi$, and then start to spread
out. We refer to the discussions in \cite{MerleRaphaelRodnianski2013InventMath,BergWilliams2013}
for Schrödinger maps and the rigorous construction of solution families
exhibiting the rotational instability \cite{KimKwon2019arXiv} for
the self-dual Chern--Simons--Schrödinger equation (which is also
a self-dual model). It is an interesting open problem to show this
rotational instability rigorously.
\end{rem}

\subsection{\label{subsec:Strategy-of-the}Strategy of the proof}

We use the notation collected in Section~\ref{subsec:Notations}.

We use modulation analysis to prove Theorem~\ref{thm:MainThm}. We
proceed similarly as in recent works \cite{HadzicRaphael2019JEMS,CollotMerleRaphael2020JAMS,CollotGhoulMasmoudiNguyen2019arXiv1,CollotGhoulMasmoudiNguyen2019arXiv2}
on type-II blow-up dynamics for parabolic equations. We extend some
parts of this approach (namely, the invariant subspace decomposition)
to type-II blow-up problems in the \emph{dispersive} setting, and
show by its consequence that our main Theorem~\ref{thm:MainThm}
can be deduced from Theorem~\ref{thm:RR-blow-up-intro}.

Since we will assume the result of Theorem~\ref{thm:RR-blow-up-intro},
we already know the blow-up time $T<+\infty$ of $\bm{u}$ and its
refined description near the blow-up time. As the blow-up speed is
almost self-similar, it is natural to study the dynamics of $\bm{u}$
in the self-similar coordinates $(\tau,\rho)$ defined by 
\begin{equation}
\tau\coloneqq-\log(T-t)\qquad\text{and}\qquad\rho\coloneqq\frac{r}{T-t}.\label{eq:def-self-similar-variables}
\end{equation}
Denoting by $\bm{v}$ the renormalized solution of $\bm{u}$
\begin{equation}
\bm{v}(\tau,\rho)=\begin{bmatrix}v(\tau,\rho)\\
\dot{v}(\tau,\rho)
\end{bmatrix}\coloneqq\begin{bmatrix}u(t,r)\\
(T-t)\dot{u}(t,r)
\end{bmatrix},\label{eq:def-v}
\end{equation}
the equation \eqref{eq:u-vec-eqn} reads 
\begin{equation}
\rd_{\tau}\bm{v}=\rd_{\tau}\begin{bmatrix}v\\
\dot{v}
\end{bmatrix}=\begin{bmatrix}-\Lmb v+\dot{v}\\
\rd_{\rho\rho}v+\frac{1}{\rho}\rd_{\rho}v-\frac{\sin(2v)}{2r^{2}}-\Lmb_{0}\dot{v}
\end{bmatrix}.\label{eq:v-vec-eqn}
\end{equation}
Note that $\bm{v}$ can be written as 
\[
\bm{v}(\tau,\rho)=\bm{Q}_{\nu(\tau)}(\rho)+\td{\bm{v}}(\tau,\rho),
\]
where $\nu(\tau)=\lmb(t)/(T-t)$ is a \emph{slowly varying} scale
going to $0$ as $\tau\to+\infty$ and $\td{\bm{v}}$ is the remainder
term. Linearizing \eqref{eq:v-vec-eqn} around $\bm{Q}_{\nu}$ in
the self-similar coordinates, we have 
\begin{equation}
\rd_{\tau}\td{\bm{v}}=-(\rd_{\tau}\bm{Q}_{\nu}+\bm{\Lmb Q}_{\nu})+\mathbf{M}_{\nu}\td{\bm{v}}-\begin{bmatrix}0\\
R_{\mathrm{NL}}(\td v)
\end{bmatrix},\label{eq:td-v-eqn}
\end{equation}
where $\mathbf{M}_{\nu}$ is the linear operator defined by 
\begin{equation}
{\bf M}_{\nu}\coloneqq\begin{bmatrix}-\Lmb & 1\\
-H_{\nu} & -\Lmb_{0}
\end{bmatrix},\label{eq:def-M_nu}
\end{equation}
$H_{\nu}$ is the usual linearized operator around $Q_{\nu}$ 
\begin{equation}
H_{\nu}\coloneqq-\rd_{\rho\rho}-\frac{1}{\rho}\rd_{\rho}+\frac{V_{\nu}}{\rho^{2}}\coloneqq-\rd_{\rho\rho}-\frac{1}{\rho}\rd_{\rho}+\frac{\cos(2Q_{\nu})}{\rho^{2}},\label{eq:def-H}
\end{equation}
and $R_{\mathrm{NL}}(\td v)$ is a nonlinear term in $\td v$. As
seen in \cite{RaphaelRodnianski2012Publ.Math.} for $k=1$, the nonlinear
term $R_{\mathrm{NL}}(\td v)$ does not affect the modulation equation
(i.e., the evolution law) for $\nu$. Therefore, \emph{any sharper
information on the modulation dynamics lies in the analysis of the
linear operator $\mathbf{M}_{\nu}$ in the regime $0<\nu\ll1$.}\vspace{5bp}

\noindent \emph{1. Construction of the first two eigenpairs for $\mathbf{M}_{\nu}$.}
We construct the first two eigenpairs $(\lmb_{0},\bm{\varphi}_{0})$
and $(\lmb_{1},\bm{\varphi}_{1})$ of $\mathbf{M}_{\nu}$ with $\lmb_{0}\approx1$
and $\lmb_{1}\approx0$. We refer to Remark \ref{rem:comparison-exact-self-sim}
for more discussions on these eigenvalues, and let us focus on the
construction of these eigenpairs here. We find that the matching argument
of \cite{HadzicRaphael2019JEMS,CollotMerleRaphael2020JAMS,CollotGhoulMasmoudiNguyen2019arXiv2}
in the parabolic case can be extended to our operator $\mathbf{M}_{\nu}$.

To explain the argument, we first notice that if $(\lmb,\bm{\varphi})$
is an eigenpair of $\mathbf{M}_{\nu}$, then $\bm{\varphi}$ has the
structure 
\[
\bm{\varphi}=\begin{bmatrix}\varphi\\
(\lmb+\Lmb)\varphi
\end{bmatrix}
\]
and $\varphi$ solves the second-order differential equation 
\begin{equation}
[H_{\nu}+(\lmb+\Lmb_{0})(\lmb+\Lmb)]\varphi=0.\label{eq:eigen-ft-rel-no-j}
\end{equation}
We need to find eigenvalues near $0$ and $1$ such that $\varphi$
is a smooth solution to \eqref{eq:eigen-ft-rel-no-j}, in particular
on $[0,1]$.

The matching argument consists of the following three steps. First,
one observes that for any $\lmb$ (with $\lmb\approx0$ or $\lmb\approx1$),
one can construct a solution $\varphi_{\inn}$ in the region $\rho\leq2\delta_{0}\ll1$
(which we call the \emph{inner eigenfunction}) to \eqref{eq:eigen-ft-rel-no-j}
which is smooth at $\rho=0$. When $\lmb$ and $\nu$ are fixed, such
a solution is unique up to multiplication by scalars. Next, one observes
that for any $\lmb$, one can construct a solution $\varphi_{\out}$
in the region $\rho\geq\frac{1}{2}\delta_{0}$ (which we call the
\emph{outer eigenfunction}) to \eqref{eq:eigen-ft-rel-no-j} which
is smooth at the light cone $\rho=1$. Again, such a solution is unique
up to multiplication by scalars. The final step is to glue the functions
$\varphi_{\inn}$ and $\varphi_{\out}$ at $\rho=\delta_{0}$. Since
we have a freedom of multiplying $\varphi_{\inn}$ or $\varphi_{\out}$
by any scalars, one can have $\varphi_{\inn}|_{\rho=\delta_{0}}=\varphi_{\out}|_{\rho=\delta_{0}}$
for any $\lmb$. However, the first derivatives of $\varphi_{\inn}$
and $\varphi_{\out}$ at $\rho=\delta_{0}$ are in general different,
and matching them is possible only for \emph{non-generic} values of
$\lmb$. These non-generic values are the eigenvalues, because the
glued function becomes a smooth solution to \eqref{eq:eigen-ft-rel-no-j}.

For the construction of the inner eigenfunctions, since $\rho\ll1$,
we can treat $(\lmb+\Lmb_{0})(\lmb+\Lmb)\varphi$ as a perturbative
term. Starting from $\frac{1}{\nu}\Lmb Q_{\nu}$, which is a smooth
kernel element of $H_{\nu}$, we perform a fixed-point argument to
construct $\varphi_{\inn}$. However, due to the \emph{critical nature}
of our spectral problem similarly as in \cite{CollotGhoulMasmoudiNguyen2019arXiv2},
we will need to study more refined structure of $\varphi_{\inn}$;
see Section~\ref{subsec:Refined-inner-eigenfunctions}. For the construction
of the outer eigenfunctions, the term $(\lmb+\Lmb_{0})(\lmb+\Lmb)\varphi$
should be taken into account, but we can approximate $H_{\nu}$ by
the $1$-equivariant Laplacian $-\Delta_{1}$, thanks to $\nu\ll1$.
The operator $-\Delta_{1}+(\lmb+\Lmb_{0})(\lmb+\Lmb)$ can be transformed
into the hypergeometric differential operator and we can construct
the outer eigenfunction $\varphi_{\out}$ by starting from a hypergeometric
function that is smooth at $\rho=1$. Finally, we can match $\varphi_{\inn}$
and $\varphi_{\out}$ for non-generic values of $\lmb$ in view of
the implicit function theorem.

\vspace{5bp}

\noindent \emph{2. Formal invariant subspace decomposition of $\mathbf{M}_{\nu}$.}
Having constructed the eigenpairs $(\lmb_{0},\bm{\varphi}_{0})$ and
$(\lmb_{1},\bm{\varphi}_{1})$, we construct two linear functionals
$\ell_{0}$ and $\ell_{1}$ such that their kernels are invariant
under $\mathbf{M}_{\nu}$ and transversal to the eigenfunctions $\bm{\varphi}_{0}$
and $\bm{\varphi}_{1}$. These linear functionals are crucial ingredients
of the proof of Theorem~\ref{thm:MainThm} because they will serve
as test functions to detect the \emph{refined modulation equations}
in our later blow-up analysis.

In contrast to the parabolic case, a new input is required to find
such linear functionals because $\mathbf{M}_{\nu}$ is highly non-self-adjoint.
In the present work, we find a very simple explicit form of these
linear functionals $\ell_{0}$ and $\ell_{1}$: 
\begin{align}
\ell_{j}(\bm{\eps}) & \coloneqq\langle(\lmb_{j}+\Lmb_{0})\eps+\dot{\eps},g_{j}\varphi_{j}\rangle,\label{eq:def-ell_j}\\
g_{j}(\nu;\rho) & \coloneqq\chf_{(0,1]}(\rho)\cdot(1-\rho^{2})^{\lmb_{j}-\frac{1}{2}}.\label{eq:def-g_j}
\end{align}
The formula is not limited to our situation, and it should be extended
to other wave equations in the self-similar coordinates (see Remark~\ref{rem:ell_j-extension-exact-self-sim}).
The appearance of a weight function of the form $(1-\rho^{2})^{p}$
in self-similar coordinates is in fact classical (see e.g., \cite{AntoniniMerle2001IMRN}),
but the formula of the invariant linear functional \eqref{eq:def-ell_j}
itself does not seem to appear explicitly in the literature. Thus
we take this opportunity to state this explicitly and use it to derive
refined modulation equations in Section~\ref{sec:Blow-up-analysis}.

\vspace{5bp}

\noindent \emph{3. Review of the Raphaël--Rodnianski blow-up solutions.}
As mentioned above, we work with finite-time blow-up solutions $\bm{u}$
constructed in \cite{RaphaelRodnianski2012Publ.Math.}. In particular,
we already know that $u$ blows up in finite time $T$, and it moreover
admits the decomposition 
\[
u(t,r)=[P^{\RR}(\wh b(t);\cdot)+\eps^{\RR}(t,\cdot)]\Big(\frac{r}{\wh{\lmb}(t)}\Big)
\]
with some modified profile $P^{\RR}$ for the blow-up part and $\wh b(t)\coloneqq-\wh{\lmb}_{t}(t)>0$,
where $\wh{\lmb}$ and $\wh b$ satisfy (see Appendix \ref{sec:Proof-lmb-b-rel})
\begin{equation}
\Big|\frac{\wh{\lmb}}{T-t}-\wh b\Big|\lesssim\frac{\wh b}{|\log\wh b|},\label{eq:rough-compat-intro}
\end{equation}
and $\eps^{\RR}$ satisfies the \emph{smallness estimate inside the
lightcone}: (see Section~\ref{subsec:RR-blow-up-sol} for more precise
statements and notation)
\begin{equation}
\|\eps^{\RR}\|_{\dot{H}^{2}(y\lesssim\wh b^{-1})}+\|\dot{\eps}^{\RR}\|_{\dot{H}^{1}(y\lesssim\wh b^{-1})}\lesssim\frac{\wh b^{2}}{|\log\wh b|}.\label{eq:smallness-intro}
\end{equation}
The logarithmic gain in the above display is crucial. This gain was
one of the key observations in \cite{RaphaelRodnianski2012Publ.Math.}
and is also indispensable for our error analysis. To put the above
information into our setting, we translate it into the first-order
formulation and in terms of the self-similar coordinates, say 
\[
\bm{v}(\tau,\rho)=\bm{P}_{\wh{\nu}(\tau)}^{\RR}(\wh b(\tau);\rho)+\bm{\eps}_{\wh{\nu}(\tau)}^{\RR}(\tau,\rho).
\]

\vspace{5bp}

\noindent \emph{4. Decomposition of solutions. }Next, we introduce
a new decomposition of $\bm{v}$: 
\[
\bm{v}(\tau,\rho)=\bm{P}(\nu(\tau),b(\tau);\rho)+\bm{\eps}(\tau,\rho),
\]
where $\bm{P}=\bm{P}(\nu,b;\rho)$ is a new modified profile 
\[
\bm{P}=\bm{Q}_{\nu}+\frac{b}{2}\Big\{(\bm{\varphi}_{0}-\bm{\varphi}_{1})+\begin{bmatrix}0\\
\Lmb Q_{\nu}
\end{bmatrix}\Big\}.
\]
$\bm{P}$ is chosen to achieve the following two goals: (i) $\bm{P}$
is described in terms of $Q$ and $\bm{\varphi}_{j}$ so that it is
fitted for applying the previous linear analysis (and computations)
of $\mathbf{M}_{\nu}$, and (ii) $\bm{P}$ is sufficiently close to
$\bm{P}^{\RR}$ (after correcting parameters) so that the smallness
estimates \eqref{eq:smallness-intro} for $\bm{\eps}^{\RR}$ can be
transferred to $\bm{\eps}$. Note that $\bm{\eps}$ does not satisfy
orthogonality conditions anymore, but its smallness will suffice.

\vspace{5bp}

\noindent \emph{5. Refined modulation equations.} To obtain sharp
evolution laws of the modulation parameters $\nu$ and $b$, we simply
test the evolution equation of $\bm{v}$ against each $\ell_{j}$.
This improves the precision of the modulation equations in \cite{RaphaelRodnianski2012Publ.Math.}.

First, as our definition of $\ell_{j}$ is sharply localized to (the
inside of) the light cone $\rho\leq1$ with the correct weight function
$g_{j}$, the computations do not encounter any cutoff errors even
at the self-similar scale $\rho\sim1$. This allows us to obtain the
modulation equations 
\begin{equation}
\left\{ \begin{aligned}\frac{\nu_{\tau}}{\nu}+\Big(\frac{b}{\nu}-1\Big)+\frac{\frac{1}{3}}{|\log\nu|} & =(\bm{\eps}\text{-term)}+O\Big(\frac{1}{|\log\nu|^{2}}\Big),\\
\frac{b_{\tau}}{b}+\Big(\frac{b}{\nu}-1\Big)\frac{\frac{1}{2}}{|\log\nu|}+\frac{\frac{1}{2}}{|\log\nu|}+\frac{\frac{5}{12}-\frac{\log2}{2}}{|\log\nu|^{2}} & =(\bm{\eps}\text{-term)}+O\Big(\frac{1}{|\log\nu|^{3}}\Big).
\end{aligned}
\right.\label{eq:mod-eqn-intro}
\end{equation}
The first improvement is the determination of the precise $\frac{1}{|\log\nu|^{2}}$-order
term in \eqref{eq:mod-eqn-intro} for $b_{\tau}$. See also Remark
\ref{rem:mod-est}.

The second improvement is in the structure of $(\bm{\eps}\text{-term})$,
thanks to the invariance of $\ell_{j}$. We refer to Section~\ref{subsec:Modulation-estimates}
for details. We remark that $(\bm{\eps}\text{-term})$ itself cannot
be considered as a perturbative term. However, the additional structure
allows us to integrate the refined modulation equations \eqref{eq:mod-eqn-intro}
by absorbing $(\bm{\eps}\text{-term})$ as a correction to the modulation
parameters. This part crucially uses the logarithmic gain in the smallness
estimate \eqref{eq:smallness-intro}.

\vspace{5bp}

\noindent \emph{6. Integration of the modulation equations.} To integrate
the modulation equations \eqref{eq:mod-eqn-intro}, assuming $\bm{\eps}=0$
temporarily, we reduce \eqref{eq:mod-eqn-intro} into a single equation
of $b$. This requires a more refined relation between $\nu$ and
$b$ than \eqref{eq:rough-compat-intro}. For this purpose, we observe
that the quantity 
\[
\frac{b}{\nu}-1,
\]
which is a priori of size $O(\frac{1}{|\log\nu|})$ by \eqref{eq:rough-compat-intro},
is exponentially unstable under the evolution \eqref{eq:mod-eqn-intro}.
Thanks to this exponential instability, a backward-in-time integration
of the evolution equation for $\frac{b}{\nu}-1$ with the boundary
condition $\frac{b}{\nu}-1\to0$ as $\tau\to+\infty$ gives a refined
relation 
\[
\frac{b}{\nu}-1\approx\frac{\frac{1}{6}}{|\log\nu|}+(\bm{\eps}\text{-term})
\]
with a certain structure on $(\bm{\eps}\text{-term})$. Substituting
this into the $b_{\tau}$-equation of \eqref{eq:mod-eqn-intro} gives
\[
\frac{b_{\tau}}{b}+\frac{\frac{1}{2}}{|\log b|}+\frac{\frac{1}{2}-\frac{\log2}{2}}{|\log b|^{2}}=(\bm{\eps}\text{-term})+O\Big(\frac{1}{|\log b|^{\frac{5}{2}}}\Big).
\]
Integrating this equation by absorbing $(\bm{\eps}\text{-term})$
as a correction to $b$ and using $\nu=b(1+o_{\tau\to\infty}(1))$
again, the sharp asymptotics of the scaling parameter $\nu(\tau)$
follows. Substituting this into $\lmb(t)=(T-t)\nu(\tau)$ completes
the proof of Theorem~\ref{thm:MainThm}.

\subsection{\label{subsec:Organization}Organization of the paper}

Sections~\ref{sec:First-two-eigenpairs}-\ref{sec:Formal-invariant-subspace}
are devoted to the linear analysis of the operator $\mathbf{M}_{\nu}$.
In Section~\ref{sec:First-two-eigenpairs}, we construct and study
the refined properties of the first two eigenpairs $(\lmb_{0},\bm{\varphi}_{0})$
and $(\lmb_{1},\bm{\varphi}_{0})$. In Section~\ref{sec:Formal-invariant-subspace},
we use these eigenpairs to obtain a formal invariant subspace decomposition
for $\mathbf{M}_{\nu}$. In particular, we construct key linear functionals
$\ell_{0}$ and $\ell_{1}$ that are invariant under $\mathbf{M}_{\nu}$.
Section~\ref{sec:Blow-up-analysis} is devoted to the blow-up analysis
and the proof of Theorem~\ref{thm:MainThm}.

\subsection{\label{subsec:Notations}Notation}

For $A\in\bbR$ and $B>0$, we use the standard asymptotic notation
$A\aleq B$ or $A=O(B)$ to denote the relation $|A|\leq CB$ for
some positive constant $C$. The dependencies of $C$ are specified
by subscripts, e.g., $A\aleq_{E}B\Leftrightarrow A=O_{E}(B)\Leftrightarrow|A|\leq C(E)B$.
We also introduce the shorthands 
\[
\langle x\rangle=(1+x^{2})^{\frac{1}{2}},\quad\log_{+}x=\max\{0,\log x\},\quad\log_{-}x=\max\{0,-\log x\}.
\]

We let $\chi$ be a smooth radial cutoff function such that $\chi(r)=1$
for $r\leq1$ and $\chi(r)=0$ for $r\geq2$. For any $R>0$, we define
$\chi_{R}(r)\coloneqq\chi(r/R)$ and $\chi_{\gtrsim R}\coloneqq1-\chi_{R}$.
We also use the sharp cutoff function on a set $A$, denoted by $\chf_{A}$.

For $f:(0,\infty)\to\bbR$, we use the shorthand for integrals: 
\[
\int f\coloneqq\int f(r)rdr=\frac{1}{2\pi}\int_{\bbR^{2}}f(|x|)dx.
\]
For functions $f,g:(0,\infty)\to\bbR$, their $L^{2}$ inner product
is defined by 
\[
\langle f,g\rangle\coloneqq\int fg.
\]
All the $L^{p}$ norms are equipped with the $rdr$-measure unless
otherwise stated. For $\lmb>0$, we define 
\[
f_{\lmb}(r)\coloneqq f\Big(\frac{r}{\lmb}\Big).
\]
For $s\in\bbR$, let $\Lmb_{s}$ be the infinitesimal generator of
the $\dot{H}^{s}$-invariant scaling: 
\begin{align*}
\Lmb_{s}f & \coloneqq\frac{d}{d\lmb}\Big|_{\lmb=1}\lmb^{1-s}f(\lmb\cdot)=(r\rd_{r}+1-s)f,\\
\Lmb f & \coloneqq\Lmb_{1}f.
\end{align*}
For a vector $\bm{f}=(f,\dot{f})^{t}$, we define 
\[
\bm{f}_{\lmb}(r)\coloneqq\begin{bmatrix}f(\frac{r}{\lmb})\\
\frac{1}{\lmb}\dot{f}(\frac{r}{\lmb})
\end{bmatrix}.
\]

We also use the 1-equivariant Laplacian $\Delta_{1}\coloneqq\rd_{rr}+\frac{1}{r}\rd_{r}-\frac{1}{r^{2}}$.
For a function $f$ and a norm $\|\cdot\|_{X}$, we write $f=O_{X}(B)$
to denote $\|f\|_{X}\lesssim B$. For $k\in\bbN$, we define 
\[
|f|_{k}\coloneqq\sup_{0\leq\ell\leq k}|r^{\ell}\rd_{r}^{\ell}f|\qquad\text{and}\qquad|f|_{-k}\coloneqq\sup_{0\leq\ell\leq k}|r^{-\ell}\rd_{r}^{k-\ell}f|.
\]

We also need notation related to hypergeometric functions. We denote
by $\Gamma$ the usual gamma function. We also use Pochhammer's symbol
\begin{equation}
(z)_{n}=\frac{\Gamma(z+n)}{\Gamma(z)}=\begin{cases}
1 & \text{if }n=0,\\
z(z+1)\cdots(z+n-1) & \text{if }n=1,2,\dots.
\end{cases}\label{eq:def-Pochhammer}
\end{equation}
We will also need the \emph{digamma function} 
\begin{equation}
\psi(z)\coloneqq\frac{\Gamma'(z)}{\Gamma(z)},\label{eq:def-digamma}
\end{equation}
i.e., the logarithmic derivative of the gamma function $\Gamma$.
Let us record several known properties of $\psi$ (see \cite[Section 6.3]{AbramowitzStegun1964}):
\begin{equation}
\left\{ \begin{aligned}\psi(1) & =-\gamma, & \psi\Big(\frac{1}{2}\Big) & =-\gamma-2\log2,\\
\psi(2) & =-\gamma+1, & \psi\Big(\frac{3}{2}\Big) & =-\gamma-2\log2+2,
\end{aligned}
\right.\label{eq:digamma-values}
\end{equation}
where $\gamma$ is the Euler--Mascheroni constant, and 
\begin{equation}
\psi(z+1)=\psi(z)+\frac{1}{z},\quad\psi(z)-\psi(w)=\sum_{n=0}^{\infty}\Big(\frac{1}{n+w}-\frac{1}{n+z}\Big).\label{eq:digamma-props}
\end{equation}

\noindent \vspace{5bp}

\noindent \mbox{\textbf{Acknowledgements.} }The author would like
to thank Charles~Collot, \mbox{Sung-Jin}~Oh, and Pierre~Raphaël
for helpful discussions and suggestions for this paper. The author
is supported by Huawei Young Talents Programme at IHES.\vspace{5bp}

\section{\label{sec:First-two-eigenpairs}First two eigenpairs for ${\bf M}_{\nu}$}

In this section, we need two small parameters $\nu^{\ast}$ and $\delta_{0}$
to be chosen in the course of the proof, satisfying the parameter
dependence $0<\nu^{\ast}\ll\delta_{0}\ll1$. Next, $\nu$ will vary
in the range $(0,\nu^{\ast})$. Also, $\lmb$ in this section always
means a spectral parameter which ranges in either $|\lmb|\aleq\frac{1}{|\log\nu|}$
or $|\lmb-1|\aleq\frac{1}{|\log\nu|}$.\footnote{In the blow-up analysis, $\lmb$ or $\lmb(t)$ in general denotes
the scale of the soliton $\bm{Q}_{\lmb(t)}$.}

The goal (Proposition~\ref{prop:FirstTwoEigenpairs}) of this section
is to construct and study the refined properties of the first two
\emph{smooth} eigenpairs $(\lmb_{0},\bm{\varphi}_{0})$ and $(\lmb_{1},\bm{\varphi}_{1})$
to the linear operator 
\[
{\bf M}_{\nu}=\begin{bmatrix}-\Lmb & 1\\
-H_{\nu} & -\Lmb_{0}
\end{bmatrix}\tag{\ref{eq:def-M_nu}}
\]
in the regime $0<\nu\ll1$ such that $\lmb_{0}\approx1$, $\lmb_{1}\approx0$,
and each $\bm{\varphi}_{j}$ is smooth (on $[0,1]$, in particular).
Note that $\lmb_{j}$ and $\bm{\varphi}_{j}$ depend on $\nu$. Recall
that $\mathbf{M}_{\nu}$ naturally appears in the linearization of
\eqref{eq:WM} in the self-similar coordinates.

The two eigenpairs $(\lmb_{0},\bm{\varphi}_{0})$ and $(\lmb_{1},\bm{\varphi}_{1})$
are crucial ingredients for our blow-up analysis. They will directly
appear in our blow-up profile ansatz \eqref{eq:def-mod-prof} and
in the explicit construction of the linear functionals $\ell_{j}$
\eqref{eq:def-ell_j} that are invariant under $\mathbf{M}_{\nu}$.
The latter functionals will be the key testing functions that enable
the refined modulation estimates as well as the sharp universal blow-up
rate \eqref{eq:sharp-blow-up-rate}.

To be explained in Section~\ref{subsec:Refined-inner-eigenfunctions},
the first two eigenvalues $\lmb_{0}$ and $\lmb_{1}$ turn out to
be approximate solutions to $p(\nu;\lmb)=0$, where\footnote{Recall the digamma function $\psi=\Gamma'/\Gamma$.}
\begin{align}
p(\nu;\lmb) & \coloneqq\lmb(\lmb-1)\Big(|\log\nu|-1-\frac{d_{0}(\lmb)}{2}\Big)+\lmb-\frac{5}{6},\label{eq:def-p}\\
d_{0}(\lmb) & \coloneqq-\psi(1)-\psi(2)+\psi\Big(\frac{\lmb}{2}+1\Big)+\psi\Big(\frac{\lmb+1}{2}\Big).\label{eq:def-d0}
\end{align}
For each $j\in\{0,1\}$, it is easy to see by the implicit function
theorem that 
\begin{equation}
\exists!\wh{\lmb}_{j}=\wh{\lmb}_{j}(\nu)\text{ satisfying }p(\nu;\wh{\lmb}_{j})=0\text{ in the class }|\wh{\lmb}_{j}-(1-j)|\lesssim\frac{1}{|\log\nu|}.\label{eq:def-lmb-hat}
\end{equation}
Moreover, since (due to \eqref{eq:digamma-values}) 
\[
d_{0}(1-j)=1-2\log2-2j,
\]
each $\wh{\lmb}_{j}$ satisfies the following $\frac{1}{|\log\nu|}$-expansion:
\begin{equation}
\left\{ \begin{aligned}\wh{\lmb}_{0} & =1+\frac{-\frac{1}{6}}{|\log\nu|}+\frac{-\frac{1}{9}+\frac{1}{6}\log2}{|\log\nu|^{2}}+O\Big(\frac{1}{|\log\nu|^{3}}\Big),\\
\wh{\lmb}_{1} & =\frac{-\frac{5}{6}}{|\log\nu|}+\frac{-\frac{5}{9}+\frac{5}{6}\log2}{|\log\nu|^{2}}+O\Big(\frac{1}{|\log\nu|^{3}}\Big).
\end{aligned}
\right.\label{eq:lmb-hat-asymp}
\end{equation}
The above expansion also holds after taking $\nu\rd_{\nu}$ to the
both sides. We can now state the main result of this section.
\begin{prop}[First two eigenpairs]
\label{prop:FirstTwoEigenpairs}There exists $\nu^{\ast}>0$ such
that for any $\nu\in(0,\nu^{\ast})$ and $j\in\{0,1\}$, there exist
unique eigenvalue $\lmb_{j}=\lmb_{j}(\nu)$ and corresponding eigenfunction
$\bm{\varphi}_{j}=\bm{\varphi}_{j}(\nu;\rho)$ to the linear operator
${\bf M}_{\nu}$ with the following properties.
\begin{itemize}
\item (Smooth eigenfunction $\bm{\varphi}_{j}$) The eigenfunctions $\bm{\varphi}_{j}(\nu;\cdot)$
are globally defined and analytic. $\bm{\varphi}_{j}$ has the structure
\begin{equation}
\bm{\varphi}_{j}=\begin{bmatrix}\varphi_{j}\\
(\lmb_{j}+\Lmb)\varphi_{j}
\end{bmatrix}\label{eq:bmphi-to-phi}
\end{equation}
and $\varphi_{j}$ solves the differential equation 
\begin{equation}
[H_{\nu}+(\lmb_{j}+\Lmb_{0})(\lmb_{j}+\Lmb)]\varphi_{j}=0.\label{eq:eigen-ft-rel}
\end{equation}
\item (Estimates for the eigenvalues) The eigenvalues $\lmb_{j}=\lmb_{j}(\nu)$
satisfy 
\begin{align}
|\lmb_{j}-\wh{\lmb}_{j}| & \aleq\nu^{2}|\log\nu|,\label{eq:sharp-eig-val}\\
|\nu\rd_{\nu}\lmb_{j}| & \aleq\frac{1}{|\log\nu|^{2}}.\label{eq:nudnu-eig-val}
\end{align}
In particular, the expansion \eqref{eq:lmb-hat-asymp} holds for $\lmb_{j}$.
\item (Sharp estimates of $\varphi_{j}$ inside the light cone) For $\rho\in(0,1]$,
we have 
\begin{equation}
\begin{aligned}\varphi_{j}(\nu;\rho) & =\Big[\frac{1}{\nu}\Lmb Q+\nu T_{1}+(2\lmb_{j}-1)\nu S_{1}+\lmb_{j}(\lmb_{j}-1)\nu U_{1}\Big]\Big(\frac{\rho}{\nu}\Big)\\
 & \peq+\chi_{\ageq\nu}(\rho)\lmb_{j}(\lmb_{j}-1)U_{\infty}(\lmb_{j};\rho)+\td{\varphi}_{j}(\nu;\rho),
\end{aligned}
\label{eq:SharpEigftDecomp}
\end{equation}
where the profiles $T_{1}(y)$, $S_{1}(y)$, $U_{1}(y)$, and $U_{\infty}(\lmb_{j};\rho)$
are defined in \eqref{eq:def-T1S1U1} and \eqref{eq:def-U-infty},
and the remainder $\td{\varphi}_{j}$ satisfies for any $k\in\bbN$
\begin{equation}
\chf_{(0,1]}(|\td{\varphi}_{j}|_{k}+|\nu\rd_{\nu}\td{\varphi}_{j}|_{k})\aleq_{k}\nu^{2}\Big(\chf_{(0,\nu]}\rho\Big(\frac{\rho}{\nu}\Big)^{4}+\chf_{[\nu,1]}\rho\big\langle\log(\frac{\rho}{\nu})\big\rangle^{2}\Big).\label{eq:SharpEigftEst}
\end{equation}
\item (Rough estimates of $\varphi_{j}$) For any $k\in\bbN$, we have 
\begin{align}
\chf_{(0,1]}|\varphi_{j}|_{k} & \aleq_{k}\chf_{(0,\nu]}\frac{\rho}{\nu^{2}}+\chf_{[\nu,1]}\frac{1}{\rho},\label{eq:EigenFuncEst-1}\\
\chf_{(0,1]}\Big|\varphi_{j}-\frac{1}{\nu}\Lmb Q_{\nu}\Big|_{k} & \aleq_{k}\chf_{(0,\nu]}\rho\Big(\frac{\rho}{\nu}\Big)^{2}+\chf_{[\nu,1]}\frac{\rho\langle\log\rho\rangle}{|\log\nu|},\label{eq:EigenFuncEst-2}
\end{align}
and 
\begin{align}
\chf_{(0,1]}|\nu\rd_{\nu}\varphi_{j}|_{k} & \aleq_{k}\chf_{(0,\nu]}\frac{\rho}{\nu^{2}}+\chf_{[\nu,1]}\Big(\frac{\rho\langle\log\rho\rangle}{|\log\nu|^{2}}+\frac{\nu^{2}}{\rho^{3}}\Big),\label{eq:EigenFuncEst-3}\\
\chf_{(0,1]}\Big|\nu\rd_{\nu}\varphi_{j}+\frac{1}{\nu}\Lmb_{0}\Lmb Q_{\nu}\Big|_{k} & \aleq_{k}\chf_{(0,\nu]}\frac{\rho^{3}}{\nu^{2}}+\chf_{[\nu,1]}\Big(\frac{\rho\langle\log\rho\rangle}{|\log\nu|^{2}}+\frac{\nu^{2}\langle\log(\frac{\rho}{\nu})\rangle}{\rho}\Big).\label{eq:EigenfuncEst-4}
\end{align}
\end{itemize}
\end{prop}

\begin{rem}[$\frac{1}{|\log\nu|}$-expansion]
Our spectral problem has a critical nature as in \cite{CollotGhoulMasmoudiNguyen2019arXiv2};
$\wh{\lmb}_{j}$ (and hence $\lmb_{j})$ does not have a polynomial
expansion, but rather has an $\frac{1}{|\log\nu|}$-expansion. In
our derivation of the refined modulation equations (see Section~\ref{subsec:Modulation-estimates}),
the precise $\frac{1}{|\log\nu|^{2}}$-order term of $\lmb_{1}$ is
necessary. However, we only need the $\frac{1}{|\log\nu|}$-order
term for $\lmb_{0}$.
\end{rem}

\begin{rem}[Structure of $\varphi_{j}$]
The profiles in the first line of RHS\eqref{eq:SharpEigftDecomp}
will be derived in the inner scale analysis, and this ansatz is sufficiently
accurate in the region $\rho\ll1$. However, the ansatz is no more
accurate in the region $\rho\sim1$, in which region the eigenfunction
$\varphi_{j}(\nu;\rho)$ in fact behaves like a hypergeometric function
that is smooth at the light cone $\rho=1$. The $U_{\infty}$-correction
term of RHS\eqref{eq:SharpEigftDecomp} is introduced to incorporate
this mismatch.
\end{rem}

\begin{rem}[Technical bounds on eigenfunctions]
Roughly speaking, by substituting $\rho\sim1$, the bounds \eqref{eq:SharpEigftEst},
\eqref{eq:EigenFuncEst-2}, and \eqref{eq:EigenFuncEst-1} might be
considered as $O(\nu^{2}|\log\nu|^{2})$, $O(\frac{1}{|\log\nu|})$,
and $O(1)$ bounds, respectively. Thus \eqref{eq:SharpEigftEst} is
a very sharp bound. However, the rough bounds \eqref{eq:EigenFuncEst-1}-\eqref{eq:EigenfuncEst-4}
will suffice in many places of our blow-up analysis.
\end{rem}

\begin{rem}[Comparison with the spectral property for exact self-similar solutions]
\label{rem:comparison-exact-self-sim}For higher-dimensional wave
maps, which are energy-supercritical, there are exact self-similar
(type-I) blow-up solutions \cite{BizonBiernat2015CMP}. There is always
an unstable eigenvalue $1$, corresponding to the time-translation
symmetry. Except this eigenvalue, the \emph{mode stability} result
of \cite{CostinDonningerGlogic2017CMP} shows that there is no other
eigenvalue $\lmb$ with $\Re(\lmb)\geq0$ for some exact self-similar
blow-up solutions. In particular, $0$ is not an eigenvalue for that
problem. For our spectral problem, although our blow-up solution is
not exactly self-similar, there is an eigenvalue $\lmb_{0}$ near
$1$. The eigenvalue $\lmb_{1}$ near $0$ (in fact $\lmb_{1}(\nu)\to0$
as $\nu\to0$) is special to our situation and responsible for our
almost self-similar blow-up.
\end{rem}

The rest of this section is devoted to the proof of Proposition~\ref{prop:FirstTwoEigenpairs}.
We use the matching argument as explained in Section~\ref{subsec:Strategy-of-the}.

\subsection{\label{subsec:Rough-inner-eigenfunctions}Rough inner eigenfunctions}

The goal of this subsection is to construct inner eigenfunctions to
the problem \eqref{eq:eigen-ft-rel-no-j} in the inner region $\rho\leq2\delta_{0}$.
In this region, in the sense of constructing the ansatz, the operator
$(\lmb+\Lmb_{0})(\lmb+\Lmb)$ of \eqref{eq:eigen-ft-rel-no-j} is
perturbative with respect to $H_{\nu}$; we rewrite the eigenproblem
as 
\begin{equation}
H_{\nu}\varphi=-(\Lmb_{0}+\lmb)(\Lmb+\lmb)\varphi,\label{eq:eigen-varphi-inner}
\end{equation}
Since $H_{\nu}\varphi$ is the main linear term, it is convenient
to zoom in to the soliton scale. Thus we introduce the inner variables
\begin{equation}
\varphi_{\inn}(\lmb,\nu;\rho)=\frac{1}{\nu}\phi_{\inn}(\lmb,\nu;y),\qquad y=\frac{\rho}{\nu},\label{eq:def-y}
\end{equation}
and rewrite the eigenproblem \eqref{eq:eigen-varphi-inner} as 
\begin{equation}
H\phi_{\inn}=-\nu^{2}(\Lmb_{0}+\lmb)(\Lmb+\lmb)\phi_{\inn}.\label{eq:eigen-phi-inner}
\end{equation}
We then construct the inner eigenfunction by considering the RHS\eqref{eq:eigen-phi-inner}
as a perturbative term.

If RHS\eqref{eq:eigen-phi-inner} were zero, $\phi$ must be a kernel
element of $H$. As we want $\phi$ to be smooth at the origin $y=0$,
we may choose $\phi=\Lmb Q$. Starting from this, it is natural to
consider an expansion of the form
\[
\phi_{\inn}=\Lmb Q+\nu^{2}(T_{1}+(2\lmb-1)S_{1}+\lmb(\lmb-1)U_{1})+O(\nu^{4}).
\]
Substituting this expansion into \eqref{eq:eigen-phi-inner} and taking
the $\nu^{2}$-order terms, we obtain the equations for the profiles
$T_{1},S_{1},U_{1}$: 
\begin{equation}
\left\{ \begin{aligned}HT_{1}+\Lmb_{0}\Lmb_{0}\Lmb Q & =0,\\
HS_{1}+\Lmb_{0}\Lmb Q & =0,\\
HU_{1}+\Lmb Q & =0.
\end{aligned}
\right.\label{eq:T1,S1,U1-eqn}
\end{equation}

To construct these profiles, we need to invert the operator $H$.
A fundamental system associated with $H$ is $\{J_{1},J_{2}\}$ with
\begin{equation}
J_{1}\coloneqq\Lmb Q\quad\text{and}\quad J_{2}\coloneqq\Lmb Q\int_{1}^{y}\frac{dy'}{y'(\Lmb Q)^{2}}.\label{eq:def-J1J2}
\end{equation}
Note that $J_{2}$ is singular at the origin $y=0$. As mentioned
above, we require $\phi$ to be smooth at $y=0$. Thus we determine
the higher order expansions using the outgoing Green's function of
$H$; we choose a formal right inverse of $H$ by the formula 
\begin{equation}
H^{-1}f\coloneqq J_{1}\int_{0}^{y}fJ_{2}y'dy'-J_{2}\int_{0}^{y}fJ_{1}y'dy'.\label{eq:def-HQ-inv}
\end{equation}
so that repeated applications of $H^{-1}$ improve regularity (or,
degeneracy) at the origin $y=0$. Using this $H^{-1}$, we define
the smooth profiles $T_{1},S_{1},U_{1}$ by 
\begin{equation}
\left\{ \begin{aligned}T_{1} & \coloneqq-H^{-1}\Lmb_{0}\Lmb_{0}\Lmb Q,\\
S_{1} & \coloneqq-H^{-1}\Lmb_{0}\Lmb Q,\\
U_{1} & \coloneqq-H^{-1}\Lmb Q.
\end{aligned}
\right.\label{eq:def-T1S1U1}
\end{equation}
The following pointwise estimates are immediate consequences of \eqref{eq:T1,S1,U1-eqn}
and \eqref{eq:def-HQ-inv}, whose proof is omitted.
\begin{lem}[Pointwise estimates for inner profiles]
For any $k\in\bbN$, the following estimates hold.
\begin{itemize}
\item (Pointwise bounds) We have 
\begin{align}
|\Lmb Q|_{k}=|J_{1}|_{k} & \aleq_{k}\chf_{(0,1]}y+\chf_{[1,+\infty)}y^{-1},\label{eq:LmbQ-bound}\\
|J_{2}|_{k} & \aleq_{k}\chf_{(0,1]}y^{-1}+\chf_{[1,+\infty)}y,\label{eq:J2-bound}
\end{align}
and 
\begin{equation}
\left\{ \begin{aligned}|T_{1}|_{k}+|S_{1}|_{k} & \aleq_{k}\chf_{(0,1]}y^{3}+\chf_{[1,+\infty)}y,\\
|U_{1}|_{k} & \aleq_{k}\chf_{(0,1]}y^{3}+\chf_{[1,+\infty)}y\langle\log y\rangle,
\end{aligned}
\right.\label{eq:T1S1U1-bound}
\end{equation}
\item (Asymptotics and degeneracy estimates as $y\to\infty$) For $y\geq1$,
we have 
\begin{equation}
\Big|\Lmb Q-\frac{2}{y}\Big|_{k}+|\Lmb_{0}\Lmb Q|_{k}\aleq_{k}\frac{1}{y^{3}}.\label{eq:LmbQ-asymptotics}
\end{equation}
and 
\begin{equation}
\left\{ \begin{aligned}\Big|T_{1}+\frac{1}{3}y\Big|_{k}+\Big|S_{1}-\frac{1}{2}y\Big|_{k}+|U_{1}-(y\log y-y)|_{k} & \aleq_{k}\frac{\langle\log y\rangle}{y},\\
|\Lmb_{2}T_{1}|_{k}+|\Lmb_{2}S_{1}|_{k}+|\Lmb_{2}U_{1}-y|_{k} & \aleq_{k}\frac{\langle\log y\rangle}{y}.
\end{aligned}
\right.\label{eq:T1S1U1-asymptotics}
\end{equation}
\item (Mapping property of $H^{-1}$) We have 
\begin{equation}
\begin{aligned}|H^{-1}f|_{2} & \aleq|J_{1}|_{2}\int_{0}^{y}|fJ_{2}|y'dy'+|J_{2}|_{2}\int_{0}^{y}|fJ_{1}|y'dy'\\
 & \peq+y^{2}\Big(|(y\rd_{y}J_{1})J_{2}|+|(y\rd_{y}J_{2})J_{1}|\Big)|f|.
\end{aligned}
\label{eq:MappingHQinv}
\end{equation}
\end{itemize}
\end{lem}

By a simple fixed-point argument, we can construct inner eigenfunctions:
\begin{lem}[Rough inner eigenfunctions]
\label{lem:RoughInnerEigenfunction}For any $\lmb$ with either $|\lmb|\aleq\frac{1}{|\log\nu|}$
or $|\lmb-1|\aleq\frac{1}{|\log\nu|}$, there exists unique smooth
solution $\phi_{\inn}(\lmb,\nu;y)$ to \eqref{eq:eigen-phi-inner}
in the region $y\in(0,\frac{2\delta_{0}}{\nu}]$, which admits the
decomposition 
\[
\phi_{\inn}(\lmb,\nu;y)=\Lmb Q(y)+\nu^{2}\{T_{1}(y)+(2\lmb-1)S_{1}(y)+\lmb(\lmb-1)U_{1}(y)\}+\td{\phi}_{\rough}(\lmb,\nu;y)
\]
with the pointwise estimates 
\[
|\td{\phi}_{\rough}|_{k}\aleq_{k}\nu^{4}(\chf_{(0,1]}y^{5}+\chf_{[1,2\delta_{0}/\nu]}y^{3}),\qquad\forall k\in\bbN.
\]
\end{lem}

\begin{proof}
By \eqref{eq:T1,S1,U1-eqn}, $\td{\phi}_{\rough}$ should solve the
equation 
\begin{equation}
H\td{\phi}_{\rough}=-\Psi_{\rough}-\nu^{2}(\Lmb_{0}+\lmb)(\Lmb+\lmb)\td{\phi}_{\rough},\label{eq:RoughInnEig-5}
\end{equation}
where $\Psi_{\rough}$ is the inhomogeneous error term from our ansatz:
\begin{equation}
\Psi_{\rough}\coloneqq\nu^{4}(\Lmb_{0}+\lmb)(\Lmb+\lmb)(T_{1}+(2\lmb-1)S_{1}+\lmb(\lmb-1)U_{1}).\label{eq:def-Psi-rough}
\end{equation}
Thus we may construct $\td{\phi}_{\rough}$ via the integral equation
\begin{equation}
\td{\phi}_{\rough}=-H^{-1}\{\Psi_{\rough}+\nu^{2}(\Lmb_{0}+\lmb)(\Lmb+\lmb)\td{\phi}_{\rough}\}.\label{eq:RoughInnEig-1}
\end{equation}

Next, we set up a fixed-point argument for \eqref{eq:RoughInnEig-1}.
For a function $f$ defined on $(0,\frac{2\delta_{0}}{\nu}]$, define
the norm $\|f\|_{X}$ by the least number satisfying 
\begin{equation}
|f|_{2}(y)\leq\|f\|_{X}(\chf_{(0,1]}y^{5}+\chf_{[1,2\delta_{0}/\nu]}y^{3}).\label{eq:RoughInnEig-4}
\end{equation}
With this norm, we claim that 
\begin{align}
\|H^{-1}\Psi_{\rough}\|_{X} & \aleq\nu^{4},\label{eq:RoughInnEig-2}\\
\|H^{-1}\{\nu^{2}(\Lmb_{0}+\lmb)(\Lmb+\lmb)f\}\|_{X} & \aleq(\delta_{0})^{2}\|f\|_{X}.\label{eq:RoughInnEig-3}
\end{align}

\emph{Proof of \eqref{eq:RoughInnEig-2}.} By $|\lmb(\lmb-1)|\aleq\frac{1}{|\log\nu|}$
and \eqref{eq:T1S1U1-bound}, we have 
\begin{equation}
\begin{aligned}|\Psi_{\rough}|_{k} & \aleq_{k}\nu^{4}\Big(|T_{1}|_{k+2}+|S_{1}|_{k+2}+\frac{1}{|\log\nu|}|U_{1}|_{k+2}\Big)\\
 & \aleq_{k}\nu^{4}(\chf_{(0,1]}y^{3}+\chf_{[1,2\delta_{0}/\nu]}y)
\end{aligned}
\label{eq:RoughInnEig-6}
\end{equation}
for any $k\in\bbN$. Applying \eqref{eq:MappingHQinv}, we have 
\[
|H^{-1}\Psi_{\rough}|_{2}\aleq\nu^{4}(\chf_{(0,1]}y^{5}+\chf_{[1,2\delta_{0}/\nu]}y^{3}).
\]
This completes the proof of \eqref{eq:RoughInnEig-2}.

\emph{Proof of \eqref{eq:RoughInnEig-3}.} By \eqref{eq:MappingHQinv},
\begin{align*}
 & |H^{-1}\{\nu^{2}(\Lmb_{0}+\lmb)(\Lmb+\lmb)f\}|_{2}\\
 & \quad\aleq\nu^{2}\Big\{|J_{1}|_{2}\int_{0}^{y}|f|_{2}|J_{2}|y'dy'+|J_{2}|_{2}\int_{0}^{y}|f|_{2}|J_{1}|y'dy'\\
 & \qquad\qquad\qquad+y^{2}\Big(|(y\rd_{y}J_{1})J_{2}|+|(y\rd_{y}J_{2})J_{1}|\Big)|f|_{2}\Big\}.
\end{align*}
Substituting the estimates \eqref{eq:LmbQ-bound}-\eqref{eq:J2-bound}
and \eqref{eq:RoughInnEig-4} into the above, we have 
\begin{align*}
 & \chf_{(0,2\delta_{0}/\nu]}|H^{-1}\{\nu^{2}(\Lmb_{0}+\lmb)(\Lmb+\lmb)f\}|_{2}\\
 & \peq\aleq\nu^{2}\|f\|_{X}(\chf_{(0,1]}y^{7}+\chf_{[1,2\delta_{0}/\nu]}y^{5})\\
 & \peq\aleq(\delta_{0})^{2}\|f\|_{X}(\chf_{(0,1]}y^{5}+\chf_{[1,2\delta_{0}/\nu]}y^{3}).
\end{align*}
This completes the proof of \eqref{eq:RoughInnEig-3}.

Having settled the proofs of the claims \eqref{eq:RoughInnEig-2}-\eqref{eq:RoughInnEig-3},
provided that $\delta_{0}\ll1$, an application of the contraction
principle yields the unique existence of $\td{\phi}_{\rough}$ in
the space $\|\td{\phi}_{\rough}\|_{X}\aleq\nu^{4}$.

It remains to show the smoothness of $\phi_{\inn}$ and higher derivative
estimates of $\td{\phi}_{\rough}$ in the region $y\in(0,\frac{2\delta_{0}}{\nu}]$.
Substituting the identity 
\begin{align*}
 & H+\nu^{2}(\Lmb_{0}+\lmb)(\Lmb+\lmb)\\
 & \peq=-(1-(\nu y)^{2})\rd_{yy}-\Big(\frac{1}{y}-\nu^{2}(2\lmb+2)y\Big)\rd_{y}+\Big(\frac{V}{y^{2}}+\nu^{2}\lmb(\lmb+1)\Big)
\end{align*}
into \eqref{eq:RoughInnEig-5}, we obtain 
\begin{align*}
\rd_{yy}\td{\phi}_{\rough} & =\frac{1}{(1-(\nu y)^{2})}\Big\{\Psi_{\rough}\\
 & \qquad-\Big(\frac{1}{y}-\nu^{2}(2\lmb+2)y\Big)\rd_{y}\phi_{\rough}+\Big(\frac{V}{y^{2}}+\nu^{2}\lmb(\lmb+1)\Big)\phi_{\rough}\Big\}.
\end{align*}
Repeated applications of the above identity with \eqref{eq:RoughInnEig-6}
completes the proof.
\end{proof}
The rough ansatz in Lemma~\ref{lem:RoughInnerEigenfunction} is not
quite accurate in the region $\nu y\sim1$ and it is \emph{insufficient}
to derive several delicate bounds in Proposition~\ref{prop:FirstTwoEigenpairs}.
In Section~\ref{subsec:Refined-inner-eigenfunctions}, we will find
a higher order correction to our rough ansatz to understand the refined
structure of the inner eigenfunctions.

\subsection{\label{subsec:Outer-eigenfunctions}Outer eigenfunctions}

The goal of this subsection is to construct outer eigenfunctions to
the problem \eqref{eq:eigen-ft-rel-no-j} in the outer region $\rho\geq\frac{1}{2}\delta_{0}$.
In this region, the operator $(\lmb+\Lmb_{0})(\lmb+\Lmb)$ of \eqref{eq:eigen-ft-rel-no-j}
is no longer perturbative and it must be included into the main linear
operator. On the other hand, we may approximate the potential $\frac{V_{\nu}}{\rho^{2}}$
by $\frac{1}{\rho^{2}}$: 
\[
\frac{V_{\nu}}{\rho^{2}}=\frac{1}{\rho^{2}}-\frac{8\nu^{2}}{(\nu^{2}+\rho^{2})^{2}}=\frac{1}{\rho^{2}}+O_{\delta_{0}}(\nu^{2}).
\]
This motivates us to rewrite the eigenproblem \eqref{eq:eigen-ft-rel-no-j}
as\footnote{Recall the $1$-equivariant Laplacian $\Delta_{1}=\rd_{\rho\rho}+\frac{1}{\rho}\rd_{\rho}-\frac{1}{\rho^{2}}$.}
\begin{equation}
(-\Delta_{1}+(\Lmb_{0}+\lmb)(\Lmb+\lmb))\varphi=\frac{8\nu^{2}}{(\nu^{2}+\rho^{2})^{2}}\varphi.\label{eq:outer-varphi-prob}
\end{equation}
We will construct outer eigenfunctions by treating the RHS as a perturbative
term.

The homogeneous linear differential equation associated with \eqref{eq:outer-varphi-prob},
if written in the variable $\rho^{2}$, has three regular singular
points at $\rho^{2}=0,1,\infty$. Hence it can be transformed into
the \emph{hypergeometric differential equation}. Indeed, if we write
\[
\varphi_{\out}(\lmb,\nu;\rho)=\frac{1}{\rho}\phi_{\out}(\lmb,\nu;z),\qquad z=1-\rho^{2},
\]
then \eqref{eq:outer-varphi-prob} is rewritten as
\begin{align}
\calK_{\lmb}\phi_{\out} & =-\frac{2\nu^{2}}{(\nu^{2}+1-z)^{2}}\phi_{\out},\label{eq:eigen-phi-outer}\\
\calK_{\lmb} & \coloneqq z(1-z)\rd_{zz}+\Big(\lmb+\frac{1}{2}\Big)(1-z)\rd_{z}-\frac{\lmb(\lmb-1)}{4}.\label{eq:def-K-lmb}
\end{align}
In order to construct $\varphi_{\out}$ that is smooth at the light
cone $\rho=1$, we construct $\phi_{\out}$ that is smooth at $z=0$.

The homogeneous linear equation associated with \eqref{eq:eigen-phi-outer},
i.e., 
\begin{equation}
\calK_{\lmb}h=0,\label{eq:homogen-eqn}
\end{equation}
has a unique solution that is smooth at $z=0$ and has value 1 at
$z=0$.\footnote{See $h_{2}$ below for a singular solution.} We denote
it by 
\begin{equation}
h_{1}(\lmb;z)\coloneqq F\Big(\frac{\lmb}{2},\frac{\lmb-1}{2};\lmb+\frac{1}{2};z\Big),\label{eq:def-h1}
\end{equation}
where $F$ is the (Gaussian) \emph{hypergeometric function}\footnote{Recall Pochhammer's symbols \eqref{eq:def-Pochhammer}.}
\begin{equation}
F(a,b;c;z)\coloneqq{}_{2}F_{1}(a,b;c;z)\coloneqq\sum_{n=0}^{\infty}\frac{(a)_{n}(b)_{n}}{(c)_{n}}\frac{z^{n}}{n!}.\label{eq:def-hypergeometric}
\end{equation}
Note that $\lmb+\frac{1}{2}$ is never a negative integer for $\lmb\approx1$
or $\lmb\approx0$. Thus the series converges absolutely when $|z|<1$.
Moreover, as \eqref{eq:homogen-eqn} has no singularities in $(-\infty,0)$,
the function $h_{1}(\lmb;\cdot)$ analytically extends over the region
$(-\infty,0)$. Thus $h_{1}(\lmb;\cdot)$ is defined and analytic
on $(-\infty,1)$. Note that there is another solution to \eqref{eq:homogen-eqn}
that is linearly independent of $h_{1}$: 
\begin{equation}
h_{2}(\lmb;z)\coloneqq z^{\frac{1}{2}-\lmb}F\Big(\frac{1-\lmb}{2},-\frac{\lmb}{2};\frac{3}{2}-\lmb;z\Big).\label{eq:def-h2}
\end{equation}
Notice that $h_{2}$ is defined for $z\in(0,1)$ and is not smooth
at $z=0$ due to the asymptotics $h_{2}(\lmb;z)\approx z^{\frac{1}{2}-\lmb}$
as $z\to0^{+}$. 

In Lemma~\ref{lem:Outer-z-var} below, we construct the outer eigenfunction
$\phi_{\out}$ using the Frobenius method. For this purpose, we need
to ensure certain growth bounds (in fact, the polynomial growth) on
the Taylor coefficients of $h_{1}$ and find it convenient to introduce
the following definition. For a function $f(\lmb,\nu;z)$ defined
for $\lmb$ with either $|\lmb|\lesssim\frac{1}{|\log\nu|}$ or $|\lmb-1|\aleq\frac{1}{|\log\nu|}$,
$\nu\in(0,\nu^{\ast})$, $|z|<1$, and $K\in\bbR$, we say 
\begin{equation}
f\in\calP_{K}\quad\Leftrightarrow\quad\left\{ \begin{aligned} & f(\lmb,\nu;z)=\sum_{n=0}^{\infty}f^{(n)}(\lmb,\nu)z^{n},\quad\text{and}\\
 & |f^{(n)}(\lmb,\nu)|\leq C(n+1)^{K}\quad\text{for some }C>0.
\end{aligned}
\right.\label{eq:def-polygrowth-K}
\end{equation}
We also say 
\[
f\in\calP\quad\Leftrightarrow\quad f\in\calP_{K}\quad\text{for some }K\in\bbR.
\]
We record some useful asymptotic formulas and pointwise estimates
of $h_{1}$:
\begin{lem}[Estimates for $h_{1}$]
\label{lem:OuterProfileEstimates}For any $\lmb$ with either $|\lmb|\aleq\frac{1}{|\log\nu|}$
or $|\lmb-1|\aleq\frac{1}{|\log\nu|}$, the following estimates hold.
\begin{itemize}
\item (An expansion for $\frac{1}{\rho}h_{1}(\lmb;1-\rho^{2})$) We have
the expansion 
\begin{equation}
\frac{1}{\rho}h_{1}(\lmb;1-\rho^{2})=\frac{\Gamma(\lmb+\frac{1}{2})}{2\Gamma(\frac{\lmb}{2}+1)\Gamma(\frac{\lmb+1}{2})}\Big\{\frac{2}{\rho}+\lmb(\lmb-1)\rho\sum_{n=0}^{\infty}c_{n}\rho^{2n}[\log\rho+\frac{d_{n}}{2}]\Big\}.\label{eq:h1-expn-rho-variable}
\end{equation}
Here, the constants $c_{n}=c_{n}(\lmb)$ and $d_{n}=d_{n}(\lmb)$
are defined by\footnote{Recall the digamma function $\psi=\Gamma'/\Gamma$.}
\begin{equation}
\left\{ \begin{aligned}c_{n}(\lmb) & \coloneqq\frac{(\frac{\lmb}{2}+1)_{n}(\frac{\lmb+1}{2})_{n}}{n!(n+1)!},\\
d_{n}(\lmb) & \coloneqq-\psi(n+1)-\psi(n+2)+\psi\Big(\frac{\lmb}{2}+1+n\Big)+\psi\Big(\frac{\lmb+1}{2}+n\Big).
\end{aligned}
\right.\label{eq:def-cn-dn}
\end{equation}
\item (Polynomial growth of Taylor coefficients) We have 
\begin{equation}
h_{1},\rd_{\lmb}h_{1}\in\calP.\label{eq:h1-poly-coefficient}
\end{equation}
\item (Rough pointwise bounds) For any $k\in\bbN$ and $z\in(0,1-\frac{1}{4}\delta_{0}^{2}]$,
we have 
\begin{equation}
|\rd_{z}^{k}h_{1}|+|\rd_{z}^{k}(\rd_{\lmb}h_{1})|\aleq_{k,\delta_{0}}1,\label{eq:h1-bound}
\end{equation}
\end{itemize}
\end{lem}

\begin{rem}
It is possible to derive sharper pointwise estimates for $\rd_{z}^{k}h_{1}$
and $\rd_{z}^{k}(\rd_{\lmb}h_{1})$ in the region $z\in(0,1)$ using
the expansion \eqref{eq:h1-expn-z-variable} below. However, we chose
to state the rough estimate \eqref{eq:h1-bound} for simplicity.
\end{rem}

\begin{proof}
We first show \eqref{eq:h1-expn-rho-variable}. We use the connection
formula for the hypergeometric function \cite[p.559 (15.3.11)]{AbramowitzStegun1964}
to have 
\begin{equation}
\begin{aligned}h_{1}(\lmb;z) & =\frac{\Gamma(\lmb+\frac{1}{2})}{\Gamma(\frac{\lmb}{2}+1)\Gamma(\frac{\lmb+1}{2})}\\
 & \peq\times\Big\{1+\frac{\lmb(\lmb-1)}{4}(1-z)\sum_{n=0}^{\infty}c_{n}(1-z)^{n}[\log(1-z)+d_{n}]\Big\}.
\end{aligned}
\label{eq:h1-expn-z-variable}
\end{equation}
Substituting $z=1-\rho^{2}$ and multiplying the above by $\rho^{-1}$,
\eqref{eq:h1-expn-rho-variable} follows.

The proof of \eqref{eq:h1-poly-coefficient} is immediate from the
following two easy facts: (i) For any $a,b\notin\{-1,-2,\dots\}$,
the sequence of ratios $((a)_{n}/(b)_{n})_{n\in\bbN}$ is of polynomial
growth. (ii) The digamma function $\psi(z)$ diverges logarithmically
as $z\to+\infty$ and the trigamma function $\psi'(z)$ behaves like
$\frac{1}{z}$ as $z\to+\infty$ (in particular they are of polynomial
growth). Finally, \eqref{eq:h1-bound} follows directly from \eqref{eq:h1-poly-coefficient}.
This completes the proof.
\end{proof}
Using the estimates in Lemma~\ref{lem:OuterProfileEstimates}, we
are now ready to construct the solutions $\phi_{\out}$ to the problem
\eqref{eq:eigen-phi-outer}.
\begin{lem}[Outer eigenfunctions in the $z$-variable]
\label{lem:Outer-z-var}For any $\lmb$ with either $|\lmb|\aleq\frac{1}{|\log\nu|}$
or $|\lmb-1|\aleq\frac{1}{|\log\nu|}$, there exists unique analytic
solution $\phi_{\out}(\lmb,\nu;z)$ to \eqref{eq:eigen-phi-outer}
in the region $z\in(-\infty,1)$ such that it admits the decomposition
\[
\phi_{\out}(\lmb,\nu;z)=h_{1}(\lmb;z)+\td{\phi}_{\out}(\lmb,\nu;z)
\]
with the following pointwise estimates in the region $z\in[0,1-\frac{\delta_{0}^{2}}{4}]$:
\begin{equation}
\left\{ \begin{aligned}|\td{\phi}_{\out}|+|\nu\rd_{\nu}\td{\phi}_{\out}|+|\rd_{\lmb}\td{\phi}_{\out}| & \aleq_{\delta_{0}}\nu^{2}z,\\
|\rd_{z}^{k}\td{\phi}_{\out}|+|\rd_{z}^{k}(\nu\rd_{\nu}\td{\phi}_{\out})|+|\rd_{z}^{k}(\rd_{\lmb}\td{\phi}_{\out})| & \aleq_{k,\delta_{0}}\nu^{2},\quad\forall k\in\bbN.
\end{aligned}
\right.\label{eq:OuterEig-1}
\end{equation}
\end{lem}

\begin{proof}
The equation for $\td{\phi}_{\out}$ is 
\begin{equation}
\calK_{\lmb}\td{\phi}_{\out}=-\frac{2\nu^{2}}{(\nu^{2}+1-z)^{2}}(h_{1}+\td{\phi}_{\out}).\label{eq:OuterEig-2}
\end{equation}
In order to construct $\td{\phi}_{\out}$ (as well as $\nu\rd_{\nu}\td{\phi}_{\out}$
and $\rd_{\lmb}\td{\phi}_{\out}$) using the Frobenius method, it
is more convenient to introduce the operator 
\[
K_{\lmb}\coloneqq z\rd_{zz}+\Big(\lmb+\frac{1}{2}\Big)\rd_{z}
\]
and further rewrite \eqref{eq:OuterEig-2} as 
\begin{equation}
K_{\lmb}\td{\phi}_{\out}=\frac{1}{1-z}\Big\{-\frac{2\nu^{2}}{(\nu^{2}+1-z)^{2}}(h_{1}+\td{\phi}_{\out})+\frac{\lmb(\lmb-1)}{4}\td{\phi}_{\out}\Big\}.\label{eq:OuterEig-3}
\end{equation}
This motivates us to introduce 
\begin{align*}
\Psi_{\out}(\lmb,\nu;z) & \coloneqq-\frac{2\nu^{2}h_{1}(\lmb;z)}{(1-z)(\nu^{2}+1-z)^{2}},\\
V_{\out}(\lmb,\nu;z) & \coloneqq-\frac{1}{1-z}\Big(\frac{2\nu^{2}}{(\nu^{2}+1-z)^{2}}-\frac{\lmb(\lmb-1)}{4}\Big),
\end{align*}
and rewrite \eqref{eq:OuterEig-3} as the following system: 
\begin{equation}
\left\{ \begin{aligned}K_{\lmb}\td{\phi}_{\out} & =\Psi_{\out}+V_{\out}\td{\phi}_{\out},\\
K_{\lmb}(\nu\rd_{\nu}\td{\phi}_{\out}) & =(\nu\rd_{\nu}\Psi_{\out})+V_{\out}(\nu\rd_{\nu}\td{\phi}_{\out})+(\nu\rd_{\nu}V_{\out})\td{\phi}_{\out},\\
K_{\lmb}(\rd_{\lmb}\td{\phi}_{\out}) & =(\rd_{\lmb}\Psi_{\out})+V_{\out}(\rd_{\lmb}\td{\phi}_{\out})+(-\rd_{\lmb}K_{\lmb}+\rd_{\lmb}V_{\out})\td{\phi}_{\out}.
\end{aligned}
\right.\label{eq:OuterEig-4}
\end{equation}
Note that the second and third rows of \eqref{eq:OuterEig-4} are
obtained by taking $\nu\rd_{\nu}$ and $\rd_{\lmb}$ to the first
row of \eqref{eq:OuterEig-4}, respectively. Note also that $\rd_{\lmb}K_{\lmb}$
is simply $\rd_{z}$. These equations are necessary to estimate $\nu\rd_{\nu}\td{\phi}_{\out}$
and $\rd_{\lmb}\td{\phi}_{\out}$.

Next, we claim that 
\begin{align}
\Psi_{\out},\nu\rd_{\nu}\Psi_{\out},\rd_{\lmb}\Psi_{\out} & \in\nu^{2}\calP,\label{eq:OuterEig-5}\\
V_{\out},\nu\rd_{\nu}V_{\out},\rd_{\lmb}V_{\out} & \in\calP_{1}.\label{eq:OuterEig-6}
\end{align}
The claim \eqref{eq:OuterEig-5} easily follows from \eqref{eq:h1-poly-coefficient}
and $(\nu^{2}+1-z)^{-2}\in\calP$ (and also for its $\nu\rd_{\nu}$-derivative),
and the algebra property of $\calP$ (i.e., $\calP\times\calP\to\calP$).
For the proof of \eqref{eq:OuterEig-6}, since $(1-z)^{-1}\in\calP_{0}$
and $\calP_{0}\times\calP_{0}\to\calP_{1}$, it suffices to show 
\[
\frac{\nu^{2}}{(\nu^{2}+1-z)^{2}},\ \frac{\nu^{4}}{(\nu^{2}+1-z)^{3}}\in\calP_{0},
\]
which follow from the expansions 
\begin{align*}
\frac{\nu^{2}}{(\nu^{2}+1-z)^{2}} & =\sum_{n=0}^{\infty}\frac{(n+1)\nu^{2}}{(1+\nu^{2})^{n+2}}z^{n},\\
\frac{\nu^{4}}{(\nu^{2}+1-z)^{3}} & =\sum_{n=0}^{\infty}\frac{\frac{1}{2}(n+1)(n+2)\nu^{4}}{(1+\nu^{2})^{n+3}}z^{n},
\end{align*}
and the inequalities $(1+\nu^{2})^{n+2}\geq(n+2)\nu^{2}$ and $(1+\nu^{2})^{n+3}\geq\frac{1}{2}(n+2)(n+3)\nu^{4}$.

By the claims \eqref{eq:OuterEig-5}-\eqref{eq:OuterEig-6}, if we
write $(\td{\phi}_{\out},\nu\rd_{\nu}\td{\phi}_{\out},\rd_{\lmb}\td{\phi}_{\out})=(f_{1},f_{2},f_{3})$
and $f_{k}=\sum_{n=0}^{\infty}f_{k}^{(n)}z^{n}$, then \eqref{eq:OuterEig-4}
takes the form 
\begin{equation}
\left\{ \begin{aligned}K_{\lmb}\sum_{n=0}^{\infty}f_{1}^{(n)}z^{n} & =\nu^{2}\sum_{n=0}^{\infty}\Psi_{1}^{(n)}z^{n}+\sum_{n=0}^{\infty}\Big(\sum_{\ell+m=n}V_{1}^{(\ell)}f_{1}^{(m)}\Big)z^{n}\\
K_{\lmb}\sum_{n=0}^{\infty}f_{2}^{(n)}z^{n} & =\nu^{2}\sum_{n=0}^{\infty}\Psi_{2}^{(n)}z^{n}+\sum_{n=0}^{\infty}\Big(\sum_{\ell+m=n}(V_{2}^{(\ell)}f_{1}^{(m)}+V_{1}^{(\ell)}f_{2}^{(m)})\Big)z^{n}\\
K_{\lmb}\sum_{n=0}^{\infty}f_{3}^{(n)}z^{n} & =\nu^{2}\sum_{n=0}^{\infty}\Psi_{3}^{(n)}z^{n}+\sum_{n=0}^{\infty}\Big(-(n+1)f_{1}^{(n+1)}\\
 & \qquad\qquad\qquad\qquad\quad+\sum_{\ell+m=n}(V_{3}^{(\ell)}f_{1}^{(m)}+V_{1}^{(\ell)}f_{3}^{(m)})\Big)z^{n},
\end{aligned}
\right.\label{eq:OuterEig-7}
\end{equation}
where $\Psi_{k}^{(n)}$ and $V_{k}^{(n)}$ are polynomially growing
sequences and $V_{k}^{(n)}$ satisfy the linear growth bounds $|V_{k}^{(n)}|\aleq n+1$.
In view of 
\[
K_{\lmb}\sum_{n=0}^{\infty}f_{k}^{(n)}z^{n}=\sum_{n=0}^{\infty}(n+1)(n+\lmb+\frac{1}{2})f_{k}^{(n+1)}z^{n},
\]
the system \eqref{eq:OuterEig-7} gives recurrence relations for the
coefficients $f_{k}^{(n+1)}$.

Now, we can construct a solution to the system \eqref{eq:OuterEig-7}
starting from the initial condition
\[
f_{k}^{(0)}=0,\qquad\forall k\in\{1,2,3\}.
\]
The recurrence relations uniquely\footnote{In fact, one can prove uniqueness in a much larger function space;
by showing that the linear differential equation \eqref{eq:eigen-phi-outer}
has a fundamental system $\{\td h_{1},\td h_{2}\}$ with $\td h_{1}=1+O(z)$
and $\td h_{2}=z^{\frac{1}{2}-\lmb}(1+O(z))$, the uniqueness holds
in any function space that removes the freedom of adding $\td h_{1}$
or $\td h_{2}$.} define \emph{formal} power series solutions $f_{k}$. We show that
$f_{k}\in\nu^{2}\calP$. To show $f_{1}\in\nu^{2}\calP$, we show
the bound $f_{1}^{(n)}\leq C_{1}\nu^{2}(n+1)^{K_{1}}$ for some constants
$C_{1},K_{1}\gg1$ (to be chosen later) by induction. This inductive
bound combined with the recurrence relations give the following inequality
(where we fixed some $K$ such that $|\Psi_{1}^{(n)}|\aleq(n+1)^{K}$):
\begin{align*}
|f_{k}^{(n+1)}| & =\Big|\frac{\nu^{2}\Psi_{1}^{(n)}+\sum_{\ell+m=n}V_{1}^{(\ell)}f_{1}^{(m)}}{(n+1)(n+\lmb+\frac{1}{2})}\Big|\\
 & \aleq\nu^{2}\frac{(n+1)^{K}+C_{1}\sum_{\ell=0}^{n}(\ell+1)(n+1-\ell)^{K_{1}}}{(n+2)^{2}}\\
 & \aleq\nu^{2}\Big((n+2)^{K-2}+\frac{C_{1}}{K_{1}}(n+2)^{K_{1}}\Big).
\end{align*}
The right hand side can be bounded by $C_{1}\nu^{2}(n+2)^{K_{1}}$
if $C_{1}$ and $K_{1}$ are chosen sufficiently large. We remark
that the property $V_{\out}\in\calP_{1}$ (or, $|V_{1}^{(n)}|\aleq n+1$)
is crucially used. By induction, we obtain $f_{1}\in\nu^{2}\calP$.
We omit the proof of $f_{2},f_{3}\in\calP$, which can be proved in
a similar fashion.

If we set $\td{\phi}_{\out}=f_{1}\in\nu^{2}\calP$, then $\td{\phi}_{\out}$
exists (as an analytic function of $z$) in the region $z\in(-1,1)$,
and $\nu\rd_{\nu}\td{\phi}_{\out}=f_{2}$ and $\rd_{\lmb}\td{\phi}_{\out}=f_{3}$
by uniqueness. The estimate \eqref{eq:OuterEig-1} follows from the
fact that $\td{\phi}_{\out}$, $\nu\rd_{\nu}\td{\phi}_{\out}$, and
$\rd_{\lmb}\td{\phi}_{\out}\in\nu^{2}\calP$ and the zero initial
condition. Finally, the solution $\phi_{\out}=h_{1}+\td{\phi}_{\out}$
analytically extends over the region $z\in(-\infty,1)$ because the
linear differential equation \eqref{eq:eigen-phi-outer} does not
have singular points in $(-\infty,0)$. This completes the proof.
\end{proof}
In terms of the self-similar variable $\rho$, we have the following.
\begin{cor}[Outer eigenfunctions]
\label{cor:OuterEigenfunctions}For any $\lmb$ with either $|\lmb|\aleq\frac{1}{|\log\nu|}$
or $|\lmb-1|\aleq\frac{1}{|\log\nu|}$, define 
\begin{equation}
\varphi_{\out}(\lmb,\nu;\rho)\coloneqq\frac{1}{\rho}\phi_{\out}(\lmb,\nu;1-\rho^{2})\label{eq:OuterDef}
\end{equation}
for all $\rho\in(0,\infty)$. Then, $\varphi_{\out}$ solves \eqref{eq:outer-varphi-prob}.
Moreover, in the region $\rho\in[\frac{1}{2}\delta_{0},1]$, it admits
the decomposition 
\begin{equation}
\varphi_{\out}(\lmb,\nu;\rho)=\frac{1}{\rho}h_{1}(\lmb;1-\rho^{2})+\td{\varphi}_{\out}(\lmb,\nu;\rho)\label{eq:OuterEigDecomp}
\end{equation}
and satisfies the pointwise estimates 
\begin{equation}
|\td{\varphi}_{\out}|_{k}+|\nu\rd_{\nu}\td{\varphi}_{\out}|_{k}+|\rd_{\lmb}\td{\varphi}_{\out}|_{k}\aleq_{k,\delta_{0}}\nu^{2},\qquad\forall k\in\bbN.\label{eq:OuterEigEst}
\end{equation}
\end{cor}

\begin{proof}
This follows from the definition \eqref{eq:OuterDef} and Lemma~\ref{lem:Outer-z-var}.
Note that it is convenient to use 
\[
\rho\rd_{\rho}\Big(\frac{1}{\rho}f(\lmb,\nu;1-\rho^{2})\Big)=\frac{1}{\rho}[(2z-2)\rd_{z}f-f](\lmb,\nu;1-\rho^{2}),
\]
where $z=1-\rho^{2}$. We omit the details.
\end{proof}

\subsection{\label{subsec:Refined-inner-eigenfunctions}Refined inner eigenfunctions}

The goal of this subsection is to obtain refined description of the
inner eigenfunctions. As mentioned in Section~\ref{subsec:Rough-inner-eigenfunctions},
the ansatz for the inner eigenfunction used in Lemma~\ref{lem:RoughInnerEigenfunction}
is not quite accurate in the region $\nu y\sim1$ (i.e., $\rho\sim1$)
and its higher order expansion is still missing. In this subsection,
we will look more carefully at the structure of this higher order
expansion. As a byproduct, we also motivate the definitions of $p(\nu;\lmb)$
and $\wh{\lmb}_{j}$ introduced at the beginning of this subsection.

In view of Corollary~\ref{cor:OuterEigenfunctions}, the outer eigenfunction
$\varphi_{\out}$ in the region $\rho\sim\delta_{0}$ is well-approximated
by $\frac{1}{\rho}h_{1}(\lmb;1-\rho^{2})$, which has the following
$\rho$-expansion \eqref{eq:h1-expn-rho-variable}: 
\begin{equation}
\varphi_{\out}\approx\frac{\Gamma(\lmb+\frac{1}{2})}{2\Gamma(\frac{\lmb}{2}+1)\Gamma(\frac{\lmb+1}{2})}\Big\{\frac{2}{\rho}+\lmb(\lmb-1)\rho\sum_{n=0}^{\infty}c_{n}\rho^{2n}[\log\rho+\frac{d_{n}}{2}]\Big\}.\label{eq:varphi-out-asymp}
\end{equation}
On the other hand, using the rough ansatz from Lemma~\ref{lem:RoughInnerEigenfunction}
and \eqref{eq:T1S1U1-asymptotics}, the inner eigenfunction $\varphi_{\inn}$
(see \eqref{eq:def-y}) in the region $\rho\sim\delta_{0}$ has \emph{rough}
asymptotics 
\begin{equation}
\varphi_{\inn}(\lmb,\nu;\rho)=\frac{1}{\nu}\phi_{\inn}(\lmb,\nu;\frac{\rho}{\nu})\approx\frac{2}{\rho}-\frac{1}{3}\rho+(2\lmb-1)\frac{1}{2}\rho+\lmb(\lmb-1)\rho(\log(\frac{\rho}{\nu})-1).\label{eq:varphi-in-rough-asymp}
\end{equation}

If the matching happens at $\rho=\delta_{0}\ll1$ for some $\lmb$,
then the inner eigenfunction (which has a \emph{rough} asymptotics
\eqref{eq:varphi-in-rough-asymp}) and the outer eigenfunction (which
has the asymptotics \eqref{eq:varphi-out-asymp}) must be parallel
in the region $\rho\sim\delta_{0}$. First, looking at the $\frac{1}{\rho}$-order
terms of \eqref{eq:varphi-out-asymp} and \eqref{eq:varphi-in-rough-asymp},
we may define our connection coefficient 
\begin{equation}
c_{\conn}(\lmb)\coloneqq\frac{2\Gamma(\frac{\lmb}{2}+1)\Gamma(\frac{\lmb+1}{2})}{\Gamma(\lmb+\frac{1}{2})}\label{eq:def-c-conn}
\end{equation}
so that \eqref{eq:varphi-out-asymp} multiplied by $c_{\conn}$ matches
\eqref{eq:varphi-in-rough-asymp} at the $\frac{1}{\rho}$-order.
Next, we look at the $\rho$-order terms. We rewrite RHS\eqref{eq:varphi-in-rough-asymp}
as 
\begin{equation}
\varphi_{\inn}(\lmb,\nu;\rho)\approx\Big\{\frac{2}{\rho}+\lmb(\lmb-1)\rho\cdot(\log\rho+\frac{d_{0}}{2})\Big\}+p(\nu;\lmb)\rho\label{eq:varphi-in-rough-asymp2}
\end{equation}
with 
\[
p(\nu;\lmb)=\lmb(\lmb-1)\Big(|\log\nu|-1-\frac{d_{0}(\lmb)}{2}\Big)+\lmb-\frac{5}{6}\tag{\ref{eq:def-p}}
\]
and observe that the terms in the curly bracket of \eqref{eq:varphi-in-rough-asymp2}
match the $\frac{1}{\rho}$- and $\rho$-order terms in the curly
bracket of \eqref{eq:varphi-out-asymp}. Therefore, it is natural
to expect that the inner eigenfunction would have large $y$ asymptotics
of the form 
\[
\Big\{\frac{2}{\rho}+\lmb(\lmb-1)\rho\sum_{n=0}^{\infty}c_{n}\rho^{2n}[\log\rho+\frac{d_{n}}{2}]\Big\}+p(\nu;\lmb)(\rho+\cdots).
\]
Therefore, (i) the missing part in our rough ansatz for the inner
eigenfunction is $\lmb(\lmb-1)U_{\infty}$, where $U_{\infty}=U_{\infty}(\lmb;\rho)$
is defined by\footnote{Note that this series converges absolutely for all $\rho\in(0,1)$.
See also Lemma~\ref{lem:U-infty} below.} 
\begin{equation}
U_{\infty}(\lmb;\rho)\coloneqq\rho\sum_{n=1}^{\infty}c_{n}\rho^{2n}[\log\rho+\frac{d_{n}}{2}],\label{eq:def-U-infty}
\end{equation}
and (ii) the eigenvalues $\lmb$ enabling the matching procedure is
an approximate solution to the equation 
\[
p(\nu;\lmb)=0.
\]
This motivates the definitions of $p(\nu;\lmb)$ and $\wh{\lmb}_{j}$.

In the following lemma, we make the above discussion into rigorous
estimates.
\begin{lem}[$U_{\infty}$-correction]
\label{lem:U-infty}Let $\lmb$ satisfy either $|\lmb|\aleq\frac{1}{|\log\nu|}$
or $|\lmb-1|\aleq\frac{1}{|\log\nu|}$.
\begin{itemize}
\item (Connection coefficient) We have 
\begin{equation}
c_{\conn}=2+O\Big(\frac{1}{|\log\nu|}\Big)\quad\text{and}\quad|\rd_{\lmb}c_{\conn}|+|\rd_{\lmb\lmb}c_{\conn}|\aleq1.\label{eq:c-conn-est}
\end{equation}
\item (Analytic extension and pointwise estimates of $U_{\infty}$) The
function $U_{\infty}(\lmb;\rho)$ extends analytically over the region
$\rho\in(0,+\infty)$. Moreover, for $\rho\in(0,1]$ and $k\in\bbN$,
we have the pointwise estimates 
\begin{equation}
|U_{\infty}|_{k}+|\rd_{\lmb}U_{\infty}|_{k}\aleq_{k}\rho^{3}\langle\log\rho\rangle.\label{eq:U-infty-est}
\end{equation}
\item (Connection to the outer eigenfunction) Let 
\begin{equation}
\begin{aligned}\wh{\Psi}_{\conn}(\lmb,\nu;\rho) & \coloneqq\Big[\frac{1}{\nu}\Lmb Q_{\nu}+\nu T_{1}+(2\lmb-1)\nu S_{1}+\lmb(\lmb-1)\nu U_{1}\Big]\Big(\frac{\rho}{\nu}\Big)\\
 & \quad+\lmb(\lmb-1)U_{\infty}(\lmb;\rho)-c_{\conn}(\lmb)\frac{1}{\rho}h_{1}(\lmb;1-\rho^{2})-p(\nu;\lmb)\rho.
\end{aligned}
\label{eq:def-Psi-hat-conn}
\end{equation}
Then, we have the following pointwise estimates for $\rho\in[\nu,1]$
and $k\in\bbN$: 
\begin{equation}
|\wh{\Psi}_{\conn}|_{k}+|\nu\rd_{\nu}\wh{\Psi}_{\conn}|_{k}+|\rd_{\lmb}\wh{\Psi}_{\conn}|_{k}\aleq_{k}\nu^{2}\Big(\frac{1}{\rho^{3}}+\frac{\langle\log(\frac{\rho}{\nu})\rangle}{\rho}\Big).\label{eq:Psi-hat-conn-est}
\end{equation}
\end{itemize}
\end{lem}

\begin{proof}
(1) \eqref{eq:c-conn-est} is clear from the definition of \eqref{eq:def-c-conn};
we omit the proof.

(2) Since the power series \eqref{eq:def-U-infty} converges absolutely
when $\rho<1$, $U_{\infty}(\lmb;\rho)$ is well-defined for $\rho\in(0,1)$.
In particular, the pointwise estimate \eqref{eq:U-infty-est} in the
region $\rho\leq\frac{1}{2}$ directly follows from the series expansion
\eqref{eq:def-U-infty}.

For larger values of $\rho$, we recall that the definition of $U_{\infty}$
is related to $h_{1}$; by \eqref{eq:h1-expn-rho-variable}, \eqref{eq:def-c-conn},
and \eqref{eq:def-U-infty}, we have 
\begin{equation}
\begin{aligned} & c_{\conn}(\lmb)\frac{1}{\rho}h_{1}(\lmb;1-\rho^{2})\\
 & \quad=\frac{2}{\rho}+\lmb(\lmb-1)\rho(\log\rho+\frac{d_{0}(\lmb)}{2})+\lmb(\lmb-1)U_{\infty}(\lmb;\rho).
\end{aligned}
\label{eq:U-infty-identity}
\end{equation}
The above identity holds a priori for $\rho\in(0,1)$ but suggests
the following alternative definition of $U_{\infty}$: 
\begin{equation}
U_{\infty}(\lmb;\rho)=\frac{1}{\lmb(\lmb-1)}\cdot\frac{1}{\rho}\big(c_{\conn}(\lmb)h_{1}(\lmb;1-\rho^{2})-2\big)-\rho(\log\rho+\frac{d_{0}(\lmb)}{2}).\label{eq:def2-U-infty}
\end{equation}
In this alternative definition, $\lmb=0$ and $\lmb=1$ are \emph{removable
singularities} due to 
\[
c_{\conn}h_{1}-2=c_{\conn}(h_{1}-1)+(c_{\conn}-2),
\]
\eqref{eq:c-conn-est}, and (letting $z=1-\rho^{2}$) 
\[
h_{1}(\lmb;z)-1=\frac{\lmb(\lmb-1)}{4\lmb+2}F\Big(\frac{\lmb}{2}+1,\frac{\lmb+1}{2};\lmb+\frac{3}{2};z\Big)\cdot z.
\]
Note that the function $F(\frac{\lmb}{2}+1,\frac{\lmb+1}{2};\lmb+\frac{3}{2};z)$
extends analytically over the region $z\in(-\infty,1)$ by the same
reason as $h_{1}$. Therefore, the alternative definition \eqref{eq:def2-U-infty}
defines an analytic function on the region $z\in(-\infty,1)$, i.e.,
$\rho\in(0,+\infty)$. In particular, we have the pointwise estimate
\eqref{eq:U-infty-est} also for $\rho\in[\frac{1}{2},1]$.

(3) Let $\rho\in[\nu,1]$. Using the identity \eqref{eq:U-infty-identity},
we have 
\begin{align*}
\wh{\Psi}_{\conn} & =\Big[\frac{1}{\nu}\Lmb Q_{\nu}+\nu T_{1}+(2\lmb-1)\nu S_{1}+\lmb(\lmb-1)\nu U_{1}\Big](\frac{\rho}{\nu})\\
 & \qquad-\Big(\frac{2}{\rho}+\lmb(\lmb-1)\rho[\log\rho+\frac{d_{0}(\lmb)}{2}]\Big)-p(\nu;\lmb)\rho
\end{align*}
We reaarrange the above using the definition \eqref{eq:def-p} of
$p(\nu;\lmb)$ as 
\begin{align*}
\wh{\Psi}_{\conn} & =\Big(\frac{1}{\nu}\Lmb Q(\frac{\rho}{\nu})-\frac{2}{\rho}\Big)+\Big(\nu T_{1}(\frac{\rho}{\nu})+\frac{1}{3}\rho\Big)+(2\lmb-1)\Big(\nu S_{1}(\frac{\rho}{\nu})-\frac{1}{2}\rho\Big)\\
 & \qquad+\lmb(\lmb-1)\Big(\nu U_{1}(\frac{\rho}{\nu})-\rho\Big(\log(\frac{\rho}{\nu})-1\Big)\Big).
\end{align*}
Applying \eqref{eq:T1S1U1-asymptotics}, we have 
\[
|\wh{\Psi}_{\conn}|_{k}+|\nu\rd_{\nu}\wh{\Psi}_{\conn}|_{k}+|\rd_{\lmb}\wh{\Psi}_{\conn}|_{k}\aleq_{k}\frac{1}{\nu}\frac{1}{y^{3}}+\nu\frac{\langle\log y\rangle}{y}\aleq_{k}\nu^{2}\Big(\frac{1}{\rho^{3}}+\frac{\langle\log(\frac{\rho}{\nu})\rangle}{\rho}\Big).
\]
This completes the proof of \eqref{eq:Psi-hat-conn-est}.
\end{proof}
We are now ready to obtain a refined description of the inner eigenfunctions.
\begin{lem}[Refined inner eigenfunctions in the $y$-variable]
\label{lem:RefinedInner-y-var}For any $\lmb$ with either $|\lmb|\aleq\frac{1}{|\log\nu|}$
or $|\lmb-1|\aleq\frac{1}{|\log\nu|}$, there exists unique eigenfunction
$\phi_{\inn}(\lmb,\nu;y)$ in the region $y\in(0,\frac{2\delta_{0}}{\nu}]$
with the following properties.
\begin{itemize}
\item (Decomposition) The inner eigenfunction $\phi_{\inn}(\lmb,\nu;y)$
admits the decomposition 
\begin{align*}
\phi_{\inn}(\lmb,\nu;y) & =\Lmb Q(y)+\nu^{2}T_{1}(y)+(2\lmb-1)\nu^{2}S_{1}(y)+\lmb(\lmb-1)\nu^{2}U_{1}(y)\\
 & \quad+\lmb(\lmb-1)\nu\chi_{\gtrsim1}(y)U_{\infty}(\lmb;\nu y)+\td{\phi}_{\inn}(\lmb,\nu;y)
\end{align*}
with the following structure of the remainder 
\[
\td{\phi}_{\inn}(\lmb,\nu;y)=\td{\phi}_{\inn,1}(\lmb,\nu;y)+p(\nu;\lmb)\td{\phi}_{\inn,2}(\lmb,\nu;y).
\]
\item (Estimates for the remainder) For any $k\in\bbN$, we have 
\begin{align*}
|\td{\phi}_{\inn,1}|_{k}+|\nu\rd_{\nu}\td{\phi}_{\inn,1}|_{k} & \aleq_{k}\nu^{4}(\chf_{(0,1]}y^{5}+\chf_{[1,2\delta_{0}/\nu]}y\langle\log y\rangle^{2}),\\
|\rd_{\lmb}\td{\phi}_{\inn,1}|_{k} & \aleq_{k}\nu^{4}(\chf_{(0,1]}y^{5}+\chf_{[1,2\delta_{0}/\nu]}|\log\nu|y\langle\log y\rangle),
\end{align*}
and 
\[
|\td{\phi}_{\inn,2}|_{k}+|\nu\rd_{\nu}\td{\phi}_{\inn,2}|_{k}+|\rd_{\lmb}\td{\phi}_{\inn,2}|_{k}\aleq_{k}\nu^{4}\chf_{[1,2\delta_{0}/\nu]}y^{3}.
\]
\end{itemize}
\end{lem}

\begin{proof}
In the proof, we always assume $y\in(0,\frac{2\delta_{0}}{\nu}]$
and write $\rho=\nu y$. Let us write 
\begin{align*}
\wh{\phi}_{\rough}(\lmb,\nu;y) & \coloneqq\Lmb Q(y)+\nu^{2}T_{1}(y)+(2\lmb-1)\nu^{2}S_{1}(y)+\lmb(\lmb-1)\nu^{2}U_{1}(y),\\
\wh{\phi}_{\inn}(\lmb,\nu;y) & \coloneqq\wh{\phi}_{\rough}(\lmb,\nu;y)+\lmb(\lmb-1)\nu\chi_{\gtrsim1}(y)U_{\infty}(\lmb;\nu y).
\end{align*}

\textbf{Step 1.} Computation of the inhomogeneous error term.

In this step, we compute the total inhomogeneous error term 
\[
\Psi_{\inn}\coloneqq\big[H+\nu^{2}(\Lmb_{0}+\lmb)(\Lmb+\lmb)\big]\wh{\phi}_{\inn}.
\]
More precisely, our aim is to derive \eqref{eq:RefInner-8} below.

We begin by recalling the inhomogeneous error from the rough ansatz
\eqref{eq:def-Psi-rough}: 
\begin{equation}
\big[H+\nu^{2}(\Lmb_{0}+\lmb)(\Lmb+\lmb)\big]\wh{\phi}_{\rough}=\nu^{2}(\Lmb_{0}+\lmb)(\Lmb+\lmb)(\wh{\phi}_{\rough}-\Lmb Q).\label{eq:RefInner-1}
\end{equation}
Next, in order to incorporate the $U_{\infty}$-correction, we note
that 
\begin{equation}
\begin{aligned} & \big[H+\nu^{2}(\Lmb_{0}+\lmb)(\Lmb+\lmb)\big]\big(\lmb(\lmb-1)\nu\chi_{\ageq1}(y)U_{\infty}(\nu y)\big)\\
 & =\lmb(\lmb-1)\nu^{3}\chi_{\ageq1}(y)\big[\big(-\Delta_{1}+(\Lmb_{0}+\lmb)(\Lmb+\lmb)\big)U_{\infty}\big](\rho)+\Psi_{1},
\end{aligned}
\label{eq:RefInner-2}
\end{equation}
where 
\begin{equation}
\begin{aligned}\Psi_{1} & \coloneqq-\lmb(\lmb-1)\nu\Big\{\frac{8}{(1+y^{2})^{2}}\chi_{\ageq1}(y)\\
 & \qquad\qquad\qquad\qquad+\big[-\Delta_{1}+\nu^{2}(\Lmb_{0}+\lmb)(\Lmb+\lmb),\chi\big]\Big\}\big(U_{\infty}(\lmb;\nu y)\big).
\end{aligned}
\label{eq:RefInner-3}
\end{equation}
The first term in \eqref{eq:RefInner-2} can be computed using the
definitions of $U_{\infty}$, $c_{n}$, and $d_{n}$ (see \eqref{eq:def-U-infty}
and \eqref{eq:def-cn-dn}): 
\begin{align*}
 & \big(-\Delta_{1}+(\Lmb_{0}+\lmb)(\Lmb+\lmb)\big)U_{\infty}(\lmb;\rho)\\
 & =-\Delta_{1}\big\{ c_{1}\rho^{3}(\log\rho+\frac{d_{1}}{2})\big\}\\
 & =-(\Lmb_{0}+\lmb)(\Lmb+\lmb)\big\{\rho(\log\rho+\frac{d_{0}}{2})\big\}.
\end{align*}
Substituting this into \eqref{eq:RefInner-2} yields 
\begin{equation}
\begin{aligned} & \big[H+\nu^{2}(\Lmb_{0}+\lmb)(\Lmb+\lmb)\big]\big(\lmb(\lmb-1)\nu\chi_{\ageq1}(y)U_{\infty}(\nu y)\big)\\
 & =-\lmb(\lmb-1)\nu^{3}\chi_{\ageq1}(y)\cdot(\Lmb_{0}+\lmb)(\Lmb+\lmb)\big\{\rho(\log\rho+\frac{d_{0}}{2})\big\}+\Psi_{1}.
\end{aligned}
\label{eq:RefInner-4}
\end{equation}
Summing up \eqref{eq:RefInner-1} and \eqref{eq:RefInner-4}, we obtain
the total inhomogeneous error: 
\begin{equation}
\Psi_{\inn}=\Psi_{1}+\nu^{4}(\Lmb_{0}+\lmb)(\Lmb+\lmb)\Psi_{2},\label{eq:RefInner-5}
\end{equation}
where (recall that $\Psi_{1}$ was defined in \eqref{eq:RefInner-3})
\begin{equation}
\Psi_{2}\coloneqq\nu^{-2}\big(\wh{\phi}_{\rough}-\Lmb Q\big)-\lmb(\lmb-1)\chi_{\ageq1}(y)\cdot\big(y\log(\nu y)+\frac{d_{0}}{2}y\big).\label{eq:RefInner-6}
\end{equation}
Finally, we rewrite $\Psi_{2}$ as (c.f. the proof of \eqref{eq:Psi-hat-conn-est})
\begin{equation}
\begin{aligned}\Psi_{2} & =\Big(T_{1}+\chi_{\ageq1}\cdot\frac{1}{3}y\Big)+(2\lmb-1)\Big(S_{1}-\chi_{\ageq1}\cdot\frac{1}{2}y\Big)\\
 & \quad+\lmb(\lmb-1)\Big(U_{1}-\chi_{\ageq1}(y\log y-y)\Big)+\chi_{\ageq1}p(\nu;\lmb)y\\
 & \eqqcolon\Psi_{2,1}+\chi_{\ageq1}p(\nu;\lmb)y.
\end{aligned}
\label{eq:RefInner-7}
\end{equation}
Substituting \eqref{eq:RefInner-7} into \eqref{eq:RefInner-5}, we
finally obtain 
\begin{equation}
\Psi_{\inn}=\wh{\Psi}_{1}+p(\nu;\lmb)\wh{\Psi}_{2},\label{eq:RefInner-8}
\end{equation}
where 
\begin{align}
\wh{\Psi}_{1} & \coloneqq\Psi_{1}+\nu^{4}(\Lmb_{0}+\lmb)(\Lmb+\lmb)\Psi_{2,1},\label{eq:RefInner-9}\\
\wh{\Psi}_{2} & \coloneqq\nu^{4}(\Lmb_{0}+\lmb)(\Lmb+\lmb)(\chi_{\ageq1}\cdot y),\label{eq:RefInner-10}
\end{align}
and $\Psi_{1}$ and $\Psi_{2,1}$ were defined in \eqref{eq:RefInner-3}
and \eqref{eq:RefInner-7}.

\textbf{Step 2.} Estimates for $\Psi_{\inn}$.

We first estimate $\wh{\Psi}_{1}$. Recall that 
\[
\wh{\Psi}_{1}=\Psi_{1}+\nu^{4}(\Lmb_{0}+\lmb)(\Lmb+\lmb)\Psi_{2,1}.
\]
To estimate $\Psi_{1}$, we use \eqref{eq:RefInner-3}, $|\lmb(\lmb-1)|\aleq\frac{1}{|\log\nu|}$,
and the pointwise estimate \eqref{eq:U-infty-est} for $U_{\infty}$,
we have 
\begin{align*}
 & |\log\nu|\big(|\Psi_{1}|_{k}+|\nu\rd_{\nu}\Psi_{1}|_{k}\big)+|\rd_{\lmb}\Psi_{1}|_{k}\\
 & \lesssim_{k}\nu\Big(\chf_{[1,2\delta_{0}/\nu]}\frac{1}{y^{4}}|U_{\infty}(\lmb;\nu y)|_{k+1}+\chf_{[1,2]}|U_{\infty}(\lmb;\nu y)|_{k+2}\Big)\\
 & \qquad+\frac{\nu}{|\log\nu|}\Big(\chf_{[1,2\delta_{0}/\nu]}\frac{1}{y^{4}}|\rd_{\lmb}U_{\infty}(\lmb;\nu y)|_{k}+\chf_{[1,2]}|\rd_{\lmb}U_{\infty}(\lmb;\nu y)|_{k+1}\Big)\\
 & \lesssim_{k}\nu^{4}\chf_{[1,2\delta_{0}/\nu]}\frac{\langle\log(\nu y)\rangle}{y}.
\end{align*}
Next, using the definition \eqref{eq:RefInner-7} of $\Psi_{2,1}$
and the asymptotics \eqref{eq:T1S1U1-asymptotics} for the profiles
$T_{1},S_{1},U_{1}$, we obtain 
\[
|\Psi_{2,1}|_{k}+|\nu\rd_{\nu}\Psi_{2,1}|_{k}+|\rd_{\lmb}\Psi_{2,1}|_{k}\aleq_{k}\chf_{(0,1]}y^{3}+\chf_{[1,2\delta_{0}/\nu]}\frac{\langle\log y\rangle}{y}.
\]
Substituting the above estimates into \eqref{eq:RefInner-9}, we have
\begin{equation}
\left\{ \begin{aligned}|\wh{\Psi}_{1}|_{k}+|\nu\rd_{\nu}\wh{\Psi}_{1}|_{k} & \aleq_{k}\nu^{4}\Big(\chf_{(0,1]}y^{3}+\chf_{[1,2\delta_{0}/\nu]}\frac{\langle\log y\rangle}{y}\Big),\\
|\rd_{\lmb}\wh{\Psi}_{1}|_{k} & \aleq_{k}\nu^{4}\Big(\chf_{(0,1]}y^{3}+\chf_{[1,2\delta_{0}/\nu]}|\log\nu|\frac{1}{y}\Big).
\end{aligned}
\right.\label{eq:RefInner-11}
\end{equation}

We turn to estimate $\wh{\Psi}_{2}$. We recall from \eqref{eq:RefInner-10}
that 
\[
\wh{\Psi}_{2}=\nu^{4}(\Lmb_{0}+\lmb)(\Lmb+\lmb)(\chi_{\ageq1}\cdot y).
\]
Thus we easily have 
\begin{equation}
|\wh{\Psi}_{2}|_{k}+|\nu\rd_{\nu}\wh{\Psi}_{2}|_{k}+|\rd_{\lmb}\wh{\Psi}_{2}|_{k}\aleq_{k}\nu^{4}\chf_{[1,2\delta_{0}/\nu]}y.\label{eq:RefInner-12}
\end{equation}

\textbf{Step 3.} Completion of the proof.

In view of the structure \eqref{eq:RefInner-8} of $\Psi_{\inn}$,
we can construct $\td{\phi}_{\inn}$ by 
\[
\td{\phi}_{\inn}=\td{\phi}_{\inn,1}+p(\nu;\lmb)\td{\phi}_{\inn,2},
\]
where $\td{\phi}_{\inn,1}$ and $\td{\phi}_{\inn,2}$ solve the differential
equations 
\begin{align*}
H\td{\phi}_{\inn,1} & =-[\wh{\Psi}_{1}+\nu^{2}(\Lmb_{0}+\lmb)(\Lmb+\lmb)\td{\phi}_{\inn,1}],\\
H\td{\phi}_{\inn,2} & =-[\wh{\Psi}_{2}+\nu^{2}(\Lmb_{0}+\lmb)(\Lmb+\lmb)\td{\phi}_{\inn,2}],
\end{align*}
respectively. Introducing the short-hand notation
\[
H_{\mathrm{rem}}\coloneqq\nu^{2}(\Lmb_{0}+\lmb)(\Lmb+\lmb)
\]
and taking $\nu\rd_{\nu}$ and $\rd_{\lmb}$ to the above differential
equations, we arrive at the system 
\[
\left\{ \begin{aligned}\td{\phi}_{\inn,\ell} & =-H^{-1}\big[\wh{\Psi}_{\ell}+H_{\mathrm{rem}}\td{\phi}_{\inn,\ell}\big],\\
(\nu\rd_{\nu}\td{\phi}_{\inn,\ell}) & =-H^{-1}\big[(\nu\rd_{\nu}\wh{\Psi}_{\ell})+(\nu\rd_{\nu}H_{\mathrm{rem}})\td{\phi}_{\inn,\ell}+H_{\mathrm{rem}}(\nu\rd_{\nu}\td{\phi}_{\inn,\ell})\big],\\
(\rd_{\lmb}\td{\phi}_{\inn,\ell}) & =-H^{-1}\big[(\rd_{\lmb}\wh{\Psi}_{\ell})+(\rd_{\lmb}H_{\mathrm{rem}})\td{\phi}_{\inn,\ell}+H_{\mathrm{rem}}(\rd_{\lmb}\td{\phi}_{\inn,\ell})\big].
\end{aligned}
\right.
\]
for each $\ell\in\{1,2\}$. As in the proof of Lemma~\ref{lem:RoughInnerEigenfunction},
we set up the function space $X_{\ell}$ whose norm is defined by
the smallest number satisfying
\begin{align*}
|f_{1}|_{2}+|f_{2}|_{2} & \leq\|f\|_{X_{1}}(\chf_{(0,1]}y^{5}+\chf_{[1,2\delta_{0}/\nu]}y\langle\log y\rangle^{2}),\\
|f_{3}|_{2} & \leq\|f\|_{X_{1}}(\chf_{(0,1]}y^{5}+\chf_{[1,2\delta_{0}/\nu]}|\log\nu|y\langle\log y\rangle),
\end{align*}
and 
\[
|f_{1}|_{2}+|f_{2}|_{2}+|f_{3}|_{2}\leq\|f\|_{X_{2}}\cdot\chf_{[1,2\delta_{0}/\nu]}y^{3},
\]
respectively. Here, $f_{1},f_{2},f_{3}$ correspond to $\td{\phi}_{\inn,\ell},\nu\rd_{\nu}\td{\phi}_{\inn,\ell},\nu\rd_{\nu}\td{\phi}_{\inn,\ell}$.
Thanks to the estimates \eqref{eq:RefInner-11}-\eqref{eq:RefInner-12}
for $\wh{\Psi}_{\ell}$, \eqref{eq:MappingHQinv} for $H^{-1}$, and
the smallness $\delta_{0}\ll1$, one can apply the contraction principle.
As a result, one obtains the unique existence of $\td{\phi}_{\inn,\ell}$
in $X_{\ell}$ as well as the desired pointwise estimates. Finally,
smoothness and higher derivative estimates follow in a similar fashion
as in the proof of Lemma~\ref{lem:RoughInnerEigenfunction}, which
we will not repeat. This completes the proof.
\end{proof}
Translating the above results in terms of the self-similar variable
$\rho$, we get:
\begin{cor}[Refined inner eigenfunctions]
\label{cor:refined-inner-eigenfunction}For any $\lmb$ with either
$|\lmb|\aleq\frac{1}{|\log\nu|}$ or $|\lmb-1|\aleq\frac{1}{|\log\nu|}$,
the inner eigenfunction defined by 
\begin{equation}
\varphi_{\inn}(\lmb,\nu;\rho)\coloneqq\frac{1}{\nu}\phi_{\inn}(\lmb,\nu;\frac{\rho}{\nu})\label{eq:def-Inner-Eig}
\end{equation}
in the region $(0,2\delta_{0}]$ satisfies the following properties.
\begin{itemize}
\item (Decomposition) The inner eigenfunction $\varphi_{\inn}(\lmb,\nu;y)$
admits the decomposition 
\begin{equation}
\begin{aligned}\varphi_{\inn}(\lmb,\nu;\rho) & =\Big[\frac{1}{\nu}\Lmb Q+\nu T_{1}+(2\lmb-1)\nu S_{1}+\lmb(\lmb-1)\nu U_{1}\Big]\Big(\frac{\rho}{\nu}\Big)\\
 & \quad+\chi_{\gtrsim\nu}(\rho)\lmb(\lmb-1)U_{\infty}(\lmb;\rho)+\td{\varphi}_{\inn}(\lmb,\nu;\rho)
\end{aligned}
\label{eq:RefInnDecomp}
\end{equation}
with the following structure of the remainder 
\begin{equation}
\td{\varphi}_{\inn}(\lmb,\nu;\rho)=\td{\varphi}_{\inn,1}(\lmb,\nu;\rho)+p(\nu;\lmb)\td{\varphi}_{\inn,2}(\lmb,\nu;\rho).\label{eq:RefInnDecomp2}
\end{equation}
\item (Estimates for the remainder) For any $k\in\bbN$, we have 
\begin{equation}
\left\{ \begin{aligned}|\td{\varphi}_{\inn,1}|_{k}+|\nu\rd_{\nu}\td{\varphi}_{\inn,1}|_{k} & \aleq_{k}\nu^{2}\Big(\chf_{(0,\nu]}\rho(\frac{\rho}{\nu})^{4}+\chf_{[\nu,2\delta_{0}]}\rho\langle\log(\frac{\rho}{\nu})\rangle^{2}\Big),\\
|\rd_{\lmb}\td{\varphi}_{\inn,1}|_{k} & \aleq_{k}\nu^{2}\Big(\chf_{(0,\nu]}\rho(\frac{\rho}{\nu})^{4}+\chf_{[\nu,2\delta_{0}]}|\log\nu|\rho\langle\log(\frac{\rho}{\nu})\rangle\Big),
\end{aligned}
\right.\label{eq:RefInnEst1}
\end{equation}
and 
\begin{equation}
|\td{\varphi}_{\inn,2}|_{k}+|\nu\rd_{\nu}\td{\varphi}_{\inn,2}|_{k}+|\rd_{\lmb}\td{\varphi}_{\inn,2}|_{k}\aleq_{k}\chf_{[\nu,2\delta_{0}]}\rho^{3}.\label{eq:RefInnEst2}
\end{equation}
\end{itemize}
\end{cor}

\begin{proof}
This follows from the definition \eqref{eq:def-Inner-Eig} and Lemma~\ref{lem:RefinedInner-y-var}.
Note that 
\[
\nu\rd_{\nu}\Big[\frac{1}{\nu}f(\lmb,\nu;\frac{\rho}{\nu})\Big]=\frac{1}{\nu}[(\nu\rd_{\nu}-\Lmb_{0})f](\lmb,\nu;\frac{\rho}{\nu}).
\]
We omit the details.
\end{proof}

\subsection{\label{subsec:Matching}Matching}

In the previous subsections, we have constructed inner and outer eigenfunctions
\emph{for any} $\lmb$ with $|\lmb|\aleq\frac{1}{|\log\nu|}$ or $|\lmb-1|\aleq\frac{1}{|\log\nu|}$.
For any $\lmb$, we may consider the function 
\[
\varphi(\lmb,\nu;\rho)=\begin{cases}
\varphi_{\inn}(\lmb,\nu;\rho) & \text{if }\rho\leq\delta_{0},\\{}
[\varphi_{\inn}/\varphi_{\out}](\lmb,\nu;\delta_{0})\cdot\varphi_{\out}(\lmb,\nu;\rho) & \text{if }\rho\geq\delta_{0},
\end{cases}
\]
where $\varphi$ is well-defined at $\rho=\delta_{0}$. In general,
this function has different left and right derivatives at $\rho=\delta_{0}$,
and hence is not a solution to \eqref{eq:eigen-ft-rel}. However,
for non-generic values of $\lmb$, it is possible that $\varphi$
has the same left and right derivatives, which we call matching. When
the matching happens, $\varphi_{\inn}(\rho)=[\varphi_{\inn}/\varphi_{\out}](\delta_{0})\cdot\varphi_{\out}(\rho)$
for all $\rho\in(\frac{1}{2}\delta_{0},2\delta_{0})$ (by the uniqueness
of solutions to second-order ODEs) and hence $\varphi$ becomes a
smooth solution to \eqref{eq:eigen-ft-rel}.

The first goal of this subsection is to show that the matching happens
for some unique $\lmb=\lmb_{j}$ near $1-j$ and in fact $\lmb_{j}\approx\wh{\lmb}_{j}$
(see \eqref{eq:lmb-hat-asymp}). Using this, we also finish the proof
of Proposition~\ref{prop:FirstTwoEigenpairs}.
\begin{lem}[Some preliminaries for matching]
\label{lem:properties-for-matching}For any $\lmb$ with either $|\lmb|\aleq\frac{1}{|\log\nu|}$
or $|\lmb-1|\aleq\frac{1}{|\log\nu|}$, the following hold for $\rho\in[\frac{1}{2}\delta_{0},2\delta_{0}]$.
\begin{itemize}
\item (Relation between inner and outer eigenfunctions) We have 
\begin{equation}
\begin{aligned}\varphi_{\inn}(\lmb,\nu;\rho) & =c_{\conn}(\lmb)\varphi_{\out}(\lmb,\nu;\rho)+p(\nu;\lmb)(\rho+\td{\varphi}_{\inn,2}(\lmb,\nu;\rho))\\
 & \quad+\Psi_{\conn}(\lmb,\nu;\rho),
\end{aligned}
\label{eq:relation-inner-outer}
\end{equation}
where we recall $c_{\conn}(\lmb)$ from \eqref{eq:def-c-conn} and
$\Psi_{\conn}$ satisfies
\begin{equation}
|\Psi_{\conn}|_{k}+|\nu\rd_{\nu}\Psi_{\conn}|_{k}+|\rd_{\lmb}\Psi_{\conn}|_{k}\aleq_{k}\nu^{2}|\log\nu|^{2}\delta_{0},\qquad\forall k\in\bbN.\label{eq:Psi-conn-est}
\end{equation}
\item (Decay rates of outer eigenfunctions) We have 
\begin{align}
\Big|\frac{\varphi_{\out}'(\lmb,\nu;\rho)}{\varphi_{\out}(\lmb,\nu;\rho)}+\frac{1}{\rho}\Big| & \aleq\frac{1}{|\log\nu|}\delta_{0}|\log\delta_{0}|,\label{eq:EigftDecayRate1}\\
\Big|\nu\rd_{\nu}\Big(\frac{\varphi_{\out}'(\lmb,\nu;\rho)}{\varphi_{\out}(\lmb,\nu;\rho)}\Big)\Big| & \aleq_{\delta_{0}}\nu^{2},\label{eq:EigftDecayRate2}\\
\Big|\rd_{\lmb}\Big(\frac{\varphi_{\out}'(\lmb,\nu;\rho)}{\varphi_{\out}(\lmb,\nu;\rho)}\Big)\Big| & \aleq\delta_{0}|\log\delta_{0}|.\label{eq:EigftDecayRate3}
\end{align}
\item (Properties of $p(\nu;\lmb)$) We have 
\begin{align}
|\log\nu||\nu\rd_{\nu}p|+\big(|p|+|\nu\rd_{\nu}\rd_{\lmb}p|\big)+\frac{1}{|\log\nu|}\big(|\rd_{\lmb}p|+|\rd_{\lmb\lmb}p\big|) & \aleq1,\label{eq:properties-p}\\
\partial_{\lmb}p & \ageq|\log\nu|.\label{eq:d_lmb-p-lower-bound}
\end{align}
\end{itemize}
\end{lem}

\begin{proof}
In the proof, we always assume $\rho\in[\frac{1}{2}\delta_{0},2\delta_{0}]$
or equivalently $y\in[\frac{\delta_{0}}{2\nu},\frac{2\delta_{0}}{\nu}]$.
In particular, we can ignore the cutoff $\chi_{\gtrsim\nu}(\rho)$
in the refined ansatz for the inner eigenfunction.

First, we show \eqref{eq:Psi-conn-est}. By \eqref{eq:relation-inner-outer},
\eqref{eq:OuterEigDecomp}, \eqref{eq:RefInnDecomp}, and \eqref{eq:def-Psi-hat-conn},
we have 
\[
\Psi_{\conn}=\varphi_{\inn}-c_{\conn}\varphi_{\out}-p(\nu;\lmb)(\rho+\td{\varphi}_{\inn,2})=\wh{\Psi}_{\conn}+\td{\varphi}_{\inn,1}-c_{\conn}\td{\varphi}_{\out}.
\]
Applying \eqref{eq:Psi-hat-conn-est}, \eqref{eq:RefInnEst1}, and
\eqref{eq:OuterEigEst} to the above (with the parameter dependence
$\nu^{\ast}\ll\delta_{0}$) completes the proof of \eqref{eq:Psi-conn-est}.

Next, we show \eqref{eq:EigftDecayRate1}-\eqref{eq:EigftDecayRate3}.
Recall that 
\[
\varphi_{\out}(\lmb,\nu;\rho)=\frac{1}{\rho}h_{1}(\lmb;1-\rho^{2})+\td{\varphi}_{\out}(\lmb,\nu;\rho).
\]
Using the expansion \eqref{eq:h1-expn-rho-variable} of $h_{1}(\lmb;1-\rho^{2})$,
we have 
\[
\frac{1}{\rho}h_{1}(\lmb;1-\rho^{2})=\frac{\Gamma(\lmb+\frac{1}{2})}{\Gamma(\frac{\lmb}{2}+1)\Gamma(\frac{\lmb+1}{2})}\cdot\frac{1}{\rho}\Big(1+\td h_{1}(\lmb;\rho)\Big),
\]
where 
\[
\td h_{1}(\lmb;\rho)=\frac{\lmb(\lmb-1)}{2}\rho^{2}\sum_{n=0}^{\infty}c_{n}\rho^{2n}[\log\rho+\frac{d_{n}}{2}].
\]
Thus 
\[
\frac{\varphi_{\out}'}{\varphi_{\out}}=\frac{\rd_{\rho}\{\frac{1}{\rho}(1+\td h_{1})+\frac{1}{2}c_{\conn}\td{\varphi}_{\out}\}}{\frac{1}{\rho}(1+\td h_{1})+\frac{1}{2}c_{\conn}\td{\varphi}_{\out}}.
\]
Now \eqref{eq:EigftDecayRate1}-\eqref{eq:EigftDecayRate3} follow
from applying the pointwise estimates 
\begin{equation}
|\log\nu||\td h_{1}|_{k}+|\rd_{\lmb}\td h_{1}|_{k}\aleq_{k}\delta_{0}^{2}|\log\delta_{0}|,\qquad\forall k\in\bbN,\label{eq:h-tilde-1-est}
\end{equation}
$\nu\rd_{\nu}\td h_{1}\equiv0$, \eqref{eq:OuterEigEst}, and \eqref{eq:c-conn-est}.

Finally, \eqref{eq:properties-p} and \eqref{eq:d_lmb-p-lower-bound}
are immediate consequences of the definition \eqref{eq:def-p} of
$p(\nu;\lmb)$ and $|\lmb(\lmb-1)|\aleq\frac{1}{|\log\nu|}$. This
completes the proof.
\end{proof}
With Lemma~\ref{lem:properties-for-matching} at hand, we are now
ready to achieve the matching.
\begin{lem}[Matching]
For each $j\in\{0,1\}$, there exists unique $\lmb_{j}=\lmb_{j}(\nu)$
in the class $|\lmb_{j}-(1-j)|\aleq\frac{1}{|\log\nu|}$ such that
the matching condition 
\begin{equation}
\varphi_{\inn}'(\lmb_{j},\nu;\delta_{0})=\varphi_{\out}'(\lmb_{j},\nu;\delta_{0})\frac{\varphi_{\inn}(\lmb_{j},\nu;\delta_{0})}{\varphi_{\out}(\lmb_{j},\nu;\delta_{0})}\label{eq:matching-cond}
\end{equation}
holds. Moreover, we have the following.
\begin{itemize}
\item (Refined estimates for the eigenvalues) We have 
\begin{align}
|\lmb_{j}-\wh{\lmb}_{j}| & \aleq\nu^{2}|\log\nu|,\label{eq:eigval-1}\\
|\nu\rd_{\nu}(\lmb_{j}-\wh{\lmb}_{j})| & \aleq\nu^{2}|\log\nu|,\label{eq:eigval-2}
\end{align}
where we recall $\wh{\lmb}_{j}$ from \eqref{eq:def-lmb-hat}. In
particular, we have 
\begin{equation}
|\nu\rd_{\nu}\lmb_{j}|\aleq\frac{1}{|\log\nu|^{2}}.\label{eq:eigval-3}
\end{equation}
\item (Estimates for $p(\nu;\lmb_{j})$) We have 
\begin{align}
|p(\nu;\lmb_{j})| & \aleq\nu^{2}|\log\nu|^{2},\label{eq:p-bound-1}\\
|\nu\rd_{\nu}[p(\nu;\lmb_{j})]| & \aleq\nu^{2}|\log\nu|^{2}.\label{eq:p-bound-2}
\end{align}
\item (Estimates for the matching constants) The matching constants defined
by 
\begin{equation}
c_{j,\match}(\nu)\coloneqq\frac{\varphi_{\inn}(\lmb_{j}(\nu),\nu;\delta_{0})}{\varphi_{\out}(\lmb_{j}(\nu),\nu;\delta_{0})},\qquad j\in\{0,1\}\label{eq:def-c-match}
\end{equation}
satisfy 
\begin{equation}
|c_{j,\match}-c_{\conn}(\wh{\lmb}_{j})|+|\nu\rd_{\nu}(c_{j,\match}-c_{\conn}(\wh{\lmb}_{j}))|\aleq\nu^{2}|\log\nu|^{2}.\label{eq:c-match-est}
\end{equation}
\end{itemize}
\end{lem}

\begin{proof}
Define 
\[
\Phi(\nu,\delta_{0};\lmb)\coloneqq\varphi_{\inn}'(\lmb,\nu;\delta_{0})-\varphi_{\out}'(\lmb,\nu;\delta_{0})\frac{\varphi_{\inn}(\lmb,\nu;\delta_{0})}{\varphi_{\out}(\lmb,\nu;\delta_{0})}.
\]
We rewrite $\Phi$ as 
\begin{equation}
\begin{aligned}\Phi(\nu,\delta_{0};\lmb) & =\Big[\Big(\rd_{\rho}-\frac{\varphi_{\out}'}{\varphi_{\out}}\Big)\varphi_{\inn}\Big](\lmb,\nu;\delta_{0})\\
 & =\Big[\Big(\rd_{\rho}-\frac{\varphi_{\out}'}{\varphi_{\out}}\Big)(\varphi_{\inn}-c_{\conn}\varphi_{\out})\Big](\lmb,\nu;\delta_{0})\\
 & =\Big[\Big(\rd_{\rho}-\frac{\varphi_{\out}'}{\varphi_{\out}}\Big)(p\cdot(\rho+\td{\varphi}_{\inn,2})+\Psi_{\conn})\Big](\lmb,\nu;\delta_{0}),
\end{aligned}
\label{eq:matching-1}
\end{equation}
where in the last equality we used \eqref{eq:relation-inner-outer}.
Note that the matching condition \eqref{eq:matching-cond} is equivalent
to 
\begin{equation}
\Phi(\nu,\delta_{0};\lmb_{j})=0.\label{eq:matching-2}
\end{equation}

\textbf{Step 1.} Preliminary estimates for $\Phi$.

In this step, we claim the following estimates for $\Phi$: 
\begin{align}
\Phi & =p\cdot\Big(\rd_{\rho}-\frac{\varphi_{\out}'}{\varphi_{\out}}\Big)(\rho+\td{\varphi}_{\inn,2})+O(\nu^{2}|\log\nu|^{2}),\label{eq:matching-3}\\
\nu\rd_{\nu}\Phi & =\nu\rd_{\nu}p\cdot\Big(\rd_{\rho}-\frac{\varphi_{\out}'}{\varphi_{\out}}\Big)(\rho+\td{\varphi}_{\inn,2})+p\cdot O(\delta_{0}^{2})+O(\nu^{2}|\log\nu|^{2}),\label{eq:matching-4}\\
\rd_{\lmb}\Phi & =\rd_{\lmb}p\cdot\Big(\rd_{\rho}-\frac{\varphi_{\out}'}{\varphi_{\out}}\Big)(\rho+\td{\varphi}_{\inn,2})+p\cdot O(\delta_{0}^{2}|\log\delta_{0}|)+O(\nu^{2}|\log\nu|^{2}).\label{eq:matching-5}
\end{align}
Indeed, \eqref{eq:matching-3} follows from \eqref{eq:matching-1}
and \eqref{eq:Psi-conn-est}. As the proofs of \eqref{eq:matching-4}
and \eqref{eq:matching-5} are very similar, so we only show \eqref{eq:matching-5}.
We start from writing 
\begin{equation}
\begin{aligned}\rd_{\lmb}\Phi & =\rd_{\lmb}p\cdot\Big(\rd_{\rho}-\frac{\varphi_{\out}'}{\varphi_{\out}}\Big)(\rho+\td{\varphi}_{\inn,2})\\
 & \quad+\Big(\rd_{\rho}-\frac{\varphi_{\out}'}{\varphi_{\out}}\Big)(p\cdot\rd_{\lmb}\td{\varphi}_{\inn,2}+\rd_{\lmb}\Psi_{\conn})\\
 & \quad-\rd_{\lmb}\Big(\frac{\varphi_{\out}'}{\varphi_{\out}}\Big)(p\cdot(\rho+\td{\varphi}_{\inn,2})+\Psi_{\conn}).
\end{aligned}
\label{eq:matching-6}
\end{equation}
We keep the first line of RHS\eqref{eq:matching-6}. For the second
line of RHS\eqref{eq:matching-6}, we use \eqref{eq:EigftDecayRate1},
\eqref{eq:RefInnEst2}, and \eqref{eq:Psi-conn-est} to have 
\[
\Big(\rd_{\rho}-\frac{\varphi_{\out}'}{\varphi_{\out}}\Big)(p\cdot\rd_{\lmb}\td{\varphi}_{\inn,2}+\rd_{\lmb}\Psi_{\conn})=p\cdot O(\delta_{0}^{2})+O(\nu^{2}|\log\nu|^{2}).
\]
For the third line of RHS\eqref{eq:matching-6}, we use \eqref{eq:EigftDecayRate3},
\eqref{eq:RefInnEst2}, and \eqref{eq:Psi-conn-est} to have 
\[
-\rd_{\lmb}\Big(\frac{\varphi_{\out}'}{\varphi_{\out}}\Big)(p\cdot(\rho+\td{\varphi}_{\inn,2})+\Psi_{\conn})=p\cdot O(\delta_{0}^{2}|\log\delta_{0}|)+O(\nu^{2}|\log\nu|^{2}).
\]
Substituting the above two displays into \eqref{eq:matching-6} completes
the proof of \eqref{eq:matching-5}.

\textbf{Step 2.} Unique existence of $\lmb_{j}$, and proofs of \eqref{eq:eigval-1}
and \eqref{eq:p-bound-1}.

We first claim that 
\begin{align}
|\Phi(\nu,\delta_{0};\wh{\lmb}_{j})| & \aleq\nu^{2}|\log\nu|^{2},\label{eq:matching-7}\\
\rd_{\lmb}\Phi(\nu,\delta_{0};\lmb) & \ageq|\log\nu|.\label{eq:matching-8}
\end{align}
Indeed, \eqref{eq:matching-7} follows from \eqref{eq:matching-1}
and $p(\nu;\wh{\lmb}_{j})=0$. The proof of \eqref{eq:matching-8}
follows from applying \eqref{eq:d_lmb-p-lower-bound}, \eqref{eq:EigftDecayRate1},
\eqref{eq:RefInnEst2}, and \eqref{eq:properties-p} to \eqref{eq:matching-5}:
\begin{align*}
\rd_{\lmb}\Phi & =\rd_{\lmb}p\cdot\Big(\rd_{\rho}-\frac{\varphi_{\out}'}{\varphi_{\out}}\Big)(\rho+\td{\varphi}_{\inn,2})+p\cdot O(\delta_{0}^{2}|\log\delta_{0}|)+O(\nu^{2}|\log\nu|^{2})\\
 & =\rd_{\lmb}p\cdot(2+O(\delta_{0}^{2}))+O(\delta_{0}^{2}|\log\delta_{0}|)\\
 & \gtrsim|\log\nu|.
\end{align*}

Having settled the proofs of the claims \eqref{eq:matching-7}-\eqref{eq:matching-8},
the unique existence of $\lmb_{j}$ satisfying $\Phi(\nu,\delta_{0};\lmb_{j})=0$
in the class $|\lmb-j|\aleq\frac{1}{|\log\nu|}$ and the bound \eqref{eq:eigval-1}
follow from the (quantitative) implicit function theorem. \eqref{eq:p-bound-1}
now follows from \eqref{eq:properties-p} and \eqref{eq:eigval-1}:
\[
|p(\nu;\lmb_{j})|=|p(\nu;\lmb_{j})-p(\nu;\lmb)|\aleq|\log\nu|\cdot\nu^{2}|\log\nu|\aleq\nu^{2}|\log\nu|^{2}.
\]

\textbf{Step 3.} Proofs of the derivative bounds \eqref{eq:eigval-2},
\eqref{eq:eigval-3}, and \eqref{eq:p-bound-2}.

We first show \eqref{eq:eigval-2}. We start from the identities 
\[
\nu\rd_{\nu}\lmb_{j}=-\Big[\frac{\nu\rd_{\nu}\Phi}{\rd_{\lmb}\Phi}\Big](\nu,\delta_{0};\lmb_{j})\qquad\text{and}\qquad\nu\rd_{\nu}\wh{\lmb}_{j}=-\Big[\frac{\nu\rd_{\nu}p}{\rd_{\lmb}p}\Big](\nu;\wh{\lmb}_{j}).
\]
This implies 
\begin{equation}
\nu\rd_{\nu}(\lmb_{j}-\wh{\lmb}_{j})=\Big[\frac{\nu\rd_{\nu}p\cdot\rd_{\lmb}\Phi-\rd_{\lmb}p\cdot\nu\rd_{\nu}\Phi}{\rd_{\lmb}p\cdot\rd_{\lmb}\Phi}\Big](\nu,\delta_{0};\lmb_{j})+\Big[\frac{\nu\rd_{\nu}p}{\rd_{\lmb}p}\Big]\bigg|_{\lmb=\lmb_{j}}^{\wh{\lmb}_{j}}.\label{eq:matching-9}
\end{equation}
The second term of RHS\eqref{eq:matching-9} is of size $O(\nu^{2})$
due to \eqref{eq:properties-p}, \eqref{eq:d_lmb-p-lower-bound},
and \eqref{eq:eigval-1}. For the first term of RHS\eqref{eq:matching-9},
we notice that its denominator has size $\ageq|\log\nu|^{2}$ due
to \eqref{eq:d_lmb-p-lower-bound} and \eqref{eq:matching-8}. Therefore,
we have proved that 
\begin{equation}
|\nu\rd_{\nu}(\lmb_{j}-\wh{\lmb}_{j})|\aleq\frac{|\nu\rd_{\nu}p\cdot\rd_{\lmb}\Phi-\rd_{\lmb}p\cdot\nu\rd_{\nu}\Phi|(\nu,\delta_{0};\lmb_{j})}{|\log\nu|^{2}}+\nu^{2}.\label{eq:matching-10}
\end{equation}
We then recall from \eqref{eq:matching-4} and \eqref{eq:matching-5}
(after applying \eqref{eq:p-bound-1}) that 
\begin{align*}
 & \nu\rd_{\nu}\Phi=\nu\rd_{\nu}p\cdot\Big(\rd_{\rho}-\frac{\varphi_{\out}'}{\varphi_{\out}}\Big)(\rho+\td{\varphi}_{\inn,2})+O(\nu^{2}|\log\nu|^{2}),\\
 & \rd_{\lmb}\Phi=\rd_{\lmb}p\cdot\Big(\rd_{\rho}-\frac{\varphi_{\out}'}{\varphi_{\out}}\Big)(\rho+\td{\varphi}_{\inn,2})+O(\nu^{2}|\log\nu|^{2}).
\end{align*}
Therefore, we have 
\[
|\nu\rd_{\nu}p\cdot\rd_{\lmb}\Phi-\rd_{\lmb}p\cdot\nu\rd_{\nu}\Phi|(\nu,\delta_{0};\lmb_{j})\aleq\nu^{2}|\log\nu|^{3}.
\]
Substituting this into \eqref{eq:matching-10} completes the proof
of \eqref{eq:eigval-2}.

Note that \eqref{eq:eigval-3} is immediate from \eqref{eq:eigval-2}
and $|\nu\rd_{\nu}\wh{\lmb}_{j}|\lesssim\frac{1}{|\log\nu|^{2}}$.
For the proof of \eqref{eq:p-bound-2}, note that 
\begin{align*}
\nu\rd_{\nu}[p(\nu;\lmb_{j})] & =\nu\rd_{\nu}[p(\nu;\lmb_{j})-p(\nu;\wh{\lmb}_{j})]\\
 & =[\nu\rd_{\nu}p](\nu;\lmb_{j})-[\nu\rd_{\nu}p](\nu;\wh{\lmb}_{j})\\
 & \peq+\nu\rd_{\nu}\lmb_{j}\cdot[\rd_{\lmb}p](\nu;\lmb_{j})-\nu\rd_{\nu}\wh{\lmb}_{j}\cdot[\rd_{\lmb}p](\nu;\wh{\lmb}_{j}).
\end{align*}
Thus \eqref{eq:p-bound-2} follows from \eqref{eq:properties-p},
\eqref{eq:eigval-1}, and \eqref{eq:eigval-2}: 
\begin{align*}
 & |\nu\rd_{\nu}[p(\nu;\lmb_{j})]|\\
 & \aleq(\sup|\nu\rd_{\nu}\rd_{\lmb}p|+\frac{1}{|\log\nu|^{2}}\sup|\rd_{\lmb\lmb}p|)|\lmb_{j}-\wh{\lmb}_{j}|+(\sup|\rd_{\lmb}p|)|\nu\rd_{\nu}(\lmb_{j}-\wh{\lmb}_{j})|\\
 & \aleq\nu^{2}|\log\nu|^{2},
\end{align*}
where the supremum is taken over $\lmb$ with $|\lmb-(1-j)|\aleq\frac{1}{|\log\nu|}$.

\textbf{Step 4.} Proof of \eqref{eq:c-match-est}.

In this last step, we show the estimate \eqref{eq:c-match-est} for
the matching constants. We start from writing 
\begin{align}
 & c_{j,\match}(\nu)-c_{\conn}(\wh{\lmb}_{j})\nonumber \\
 & =\Big[\frac{1}{\varphi_{\out}}\big(\varphi_{\inn}-c_{\conn}(\lmb_{j})\varphi_{\out}\big)\Big](\lmb_{j},\nu;\delta_{0})+\big(c_{\conn}(\lmb_{j})-c_{\conn}(\wh{\lmb}_{j})\big)\nonumber \\
 & =\Big[\frac{1}{\varphi_{\out}}\big(p\cdot(\rho+\td{\varphi}_{\inn,2})+\Psi_{\conn}\big)\Big](\lmb_{j},\nu;\delta_{0})+\big(c_{\conn}(\lmb_{j})-c_{\conn}(\wh{\lmb}_{j})\big),\label{eq:matching-11}
\end{align}
where in the last equality we used \eqref{eq:relation-inner-outer}.
Using $\frac{1}{\varphi_{\out}}\sim\rho\sim\delta_{0}$, \eqref{eq:p-bound-1},
and \eqref{eq:Psi-conn-est}, the first term of RHS\eqref{eq:matching-11}
is of size $O(\delta_{0}^{2}\nu^{2}|\log\nu|^{2})$. The same bound
also holds for $\nu\rd_{\nu}$ taken on the first term of RHS\eqref{eq:matching-11}.
On the other hand, using \eqref{eq:eigval-1}, the second term of
RHS\eqref{eq:matching-11} is of size $O(\nu^{2}|\log\nu|)$. The
same bound also holds after taking $\nu\rd_{\nu}$. This completes
the proof.
\end{proof}
We finish the proof of Proposition~\ref{prop:FirstTwoEigenpairs}.
\begin{proof}[Proof of Proposition~\ref{prop:FirstTwoEigenpairs}]
We fix small parameters $0<\delta_{0}\ll1$ and $0<\nu^{\ast}=\nu^{\ast}(\delta_{0})\ll1$
so that all the previous lemmas hold. Now we can safely ignore any
dependence on $\delta_{0}$ in the following analysis.

(1) Define 
\[
\varphi_{j}(\nu;\rho)\coloneqq\begin{cases}
\varphi_{\inn}(\lmb_{j}(\nu),\nu;\rho) & \text{if }\rho<\delta_{0},\\
c_{j,\match}(\nu)\varphi_{\out}(\lmb_{j}(\nu),\nu;\rho) & \text{if }\rho\geq\delta_{0},
\end{cases}
\]
and define $\td{\varphi}_{j}$ via the formula \eqref{eq:SharpEigftDecomp}.
By the definitions \eqref{eq:def-c-match} and \eqref{eq:matching-cond},
the values and the first derivatives of $\varphi_{\inn}(\lmb_{j}(\nu),\nu;\rho)$
and $c_{j,\match}(\nu)\varphi_{\out}(\lmb_{j}(\nu),\nu;\rho)$ at
$\rho=\delta_{0}$ are the same. Since both are solutions to the second-order
differential equation \eqref{eq:eigen-ft-rel}, whose singular points
do not belong to $(0,1)$, they are in fact equal for all $\rho\in(0,2\delta_{0}]$.
Therefore, $\varphi_{j}$ is a globally-defined analytic solution
to \eqref{eq:eigen-ft-rel}.

(2) The estimates \eqref{eq:sharp-eig-val} and \eqref{eq:nudnu-eig-val}
are already proved in \eqref{eq:eigval-1}-\eqref{eq:eigval-3}.

(3) We show the sharp eigenfunction estimate \eqref{eq:SharpEigftEst}.
In the inner region $\rho\in(0,\delta_{0}]$, we have 
\[
\td{\varphi}_{j}(\nu;\rho)=\td{\varphi}_{\inn}(\lmb_{j}(\nu),\nu;\rho).
\]
Substituting the structure \eqref{eq:RefInnDecomp2} of $\td{\varphi}_{\inn}$
into the above, we have
\begin{equation}
\td{\varphi}_{j}(\nu;\rho)=\td{\varphi}_{\inn,1}(\lmb_{j},\nu;\rho)+p(\nu;\lmb_{j})\td{\varphi}_{\inn,2}(\lmb_{j},\nu;\rho).\label{eq:SharpEigft-1}
\end{equation}
Taking $\nu\rd_{\nu}$, we also have 
\begin{align}
 & [\nu\rd_{\nu}\td{\varphi}_{j}](\nu;\rho)\nonumber \\
 & =\big[\nu\rd_{\nu}\td{\varphi}_{\inn,1}+(\nu\rd_{\nu}\lmb_{j})\rd_{\lmb}\td{\varphi}_{\inn,1}\big](\lmb_{j},\nu;\rho)+\nu\rd_{\nu}[p(\nu;\lmb_{j})]\cdot\td{\varphi}_{\inn,2}(\lmb_{j},\nu;\rho)\label{eq:SharpEigft-2}\\
 & \peq+p(\nu;\lmb_{j})\cdot\big[\big(\nu\rd_{\nu}\td{\varphi}_{\inn,2}+(\nu\rd_{\nu}\lmb_{j})\rd_{\lmb}\td{\varphi}_{\inn,2}\big](\lmb_{j},\nu;\rho).\nonumber 
\end{align}
Applying the estimates \eqref{eq:RefInnEst1}-\eqref{eq:RefInnEst2}
of $\td{\varphi}_{\inn}$ and \eqref{eq:p-bound-1}-\eqref{eq:p-bound-2}
of $p(\nu;\lmb_{j})$ to the equations \eqref{eq:SharpEigft-1}-\eqref{eq:SharpEigft-2},
we obtain \eqref{eq:SharpEigftEst} for $\rho\in(0,\delta_{0}]$.

In the outer region $\rho\in[\delta_{0},1]$, we use $\varphi_{j}=c_{j,\match}\varphi_{\out}$
and \eqref{eq:def-Psi-hat-conn} to have 
\begin{align*}
\td{\varphi}_{j}(\nu;\rho) & =c_{j,\match}(\nu)\varphi_{\out}(\lmb_{j},\nu;\rho)-c_{\conn}(\lmb_{j})\rho^{-1}h_{1}(\lmb_{j};1-\rho^{2})\\
 & \quad-\wh{\Psi}_{\conn}(\lmb_{j},\nu;\rho)-p(\nu;\lmb_{j})\rho.
\end{align*}
We rearrange the above as 
\begin{align}
\td{\varphi}_{j}(\nu;\rho) & =\big(c_{j,\match}(\nu)-c_{\conn}(\lmb_{j})\big)\rho^{-1}h_{1}(\lmb_{j};1-\rho^{2})\label{eq:SharpEigft-3}\\
 & \peq+c_{j,\match}(\nu)\td{\varphi}_{\out}(\lmb_{j},\nu;\rho)-\wh{\Psi}_{\conn}(\lmb_{j},\nu;\rho)-p(\nu;\lmb_{j})\rho.\nonumber 
\end{align}
Taking $\nu\rd_{\nu}$, we also have 
\begin{align}
[\nu\rd_{\nu}\td{\varphi}_{j}](\nu;\rho) & =\nu\rd_{\nu}\big(c_{j,\match}(\nu)-c_{\conn}(\lmb_{j})\big)\cdot\rho^{-1}h_{1}(\lmb_{j};1-\rho^{2})\label{eq:SharpEigft-4}\\
 & \peq+\big(c_{j,\match}(\nu)-c_{\conn}(\lmb_{j})\big)\cdot(\nu\rd_{\nu}\lmb_{j})\rho^{-1}[\rd_{\lmb}h_{1}](\lmb_{j};1-\rho^{2})\nonumber \\
 & \peq+[\nu\rd_{\nu}c_{j,\match}](\nu)\td{\varphi}_{\out}(\lmb_{j},\nu;\rho)\nonumber \\
 & \peq+c_{j,\match}(\nu)[\nu\rd_{\nu}\td{\varphi}_{\out}+(\nu\rd_{\nu}\lmb_{j})\rd_{\lmb}\td{\varphi}_{\out}](\lmb_{j},\nu;\rho)\nonumber \\
 & \peq-[\nu\rd_{\nu}\wh{\Psi}_{\conn}+(\nu\rd_{\nu}\lmb_{j})\rd_{\lmb}\wh{\Psi}_{\conn}](\lmb_{j},\nu;\rho)-\nu\rd_{\nu}[p(\nu;\lmb_{j})]\rho.\nonumber 
\end{align}
We note as a consequence of \eqref{eq:c-match-est} and $|\nu\rd_{\nu}\wh{\lmb}_{j}|\aleq|\log\nu|^{-2}$
that 
\[
|c_{j,\match}|\lesssim1\quad\text{and}\quad|\nu\rd_{\nu}c_{j,\match}|\lesssim|\log\nu|^{-2}.
\]
Applying the estimates \eqref{eq:c-match-est}, \eqref{eq:OuterEigEst},
\eqref{eq:Psi-hat-conn-est}, \eqref{eq:p-bound-1}-\eqref{eq:p-bound-2},
and \eqref{eq:nudnu-eig-val} to the displays \eqref{eq:SharpEigft-3}-\eqref{eq:SharpEigft-4}
above, we obtain \eqref{eq:SharpEigftEst} for $\rho\in[\delta_{0},1]$.

(4) Finally, we show \eqref{eq:EigenFuncEst-1}-\eqref{eq:EigenfuncEst-4}.
We note that \eqref{eq:EigenFuncEst-1} and \eqref{eq:EigenFuncEst-3}
are immediate from \eqref{eq:EigenFuncEst-2} and \eqref{eq:EigenfuncEst-4},
respectively. Henceforth, we show \eqref{eq:EigenFuncEst-2} and \eqref{eq:EigenfuncEst-4}.

Let $\rho\in(0,1]$. For the proof of \eqref{eq:EigenFuncEst-2},
we apply $|\lmb_{j}(\lmb_{j}-1)|\lesssim\frac{1}{|\log\nu|}$, \eqref{eq:U-infty-est},
and \eqref{eq:SharpEigftEst} to \eqref{eq:SharpEigftDecomp} to have
\begin{equation}
\Big|\Big(\varphi_{j}-\frac{1}{\nu}\Lmb Q_{\nu}\Big)-\nu\Big[T_{1}+(2\lmb_{j}-1)S_{1}+\lmb_{j}(\lmb_{j}-1)U_{1}\Big]_{\nu}\Big|_{k}\lesssim_{k}\frac{\rho^{3}\langle\log\rho\rangle}{|\log\nu|}.\label{eq:EigenFunc-01}
\end{equation}
Next, we use more precise asymptotics 
\begin{equation}
\lmb_{j}(\lmb_{j}-1)=\frac{-\frac{1}{6}+j}{|\log\nu|}+O\Big(\frac{1}{|\log\nu|^{2}}\Big)\label{eq:EigenFunc-02}
\end{equation}
and \eqref{eq:T1S1U1-asymptotics} to have 
\[
\chf_{(0,\nu^{-1}]}\big|T_{1}+(2\lmb_{j}-1)S_{1}+\lmb_{j}(\lmb_{j}-1)U_{1}\big|_{k}\lesssim_{k}\chf_{(0,1]}y^{3}+\chf_{[1,\nu^{-1}]}\frac{y\langle\log(\nu y)\rangle}{|\log\nu|}.
\]
Notice that for large $y\sim\nu^{-1}$ this bound is logarithmically
better than a rough bound $O(y)$ from \eqref{eq:T1S1U1-bound}. Substituting
the above display into \eqref{eq:EigenFunc-01} completes the proof
of \eqref{eq:EigenFuncEst-2}.

For the proof of \eqref{eq:EigenfuncEst-4}, we take $\nu\rd_{\nu}$
to \eqref{eq:SharpEigftDecomp} to have 
\begin{align}
\nu\rd_{\nu}\Big(\varphi_{j}-\frac{1}{\nu}\Lmb Q_{\nu}\Big) & =-\nu\Lmb_{2}\{T_{1}+(2\lmb_{j}-1)S_{1}\}_{\nu}+(\nu\rd_{\nu}\lmb_{j})2\nu[S_{1}]_{\nu}\label{eq:EigenFunc-03}\\
 & \peq+\nu\Big[-\lmb_{j}(\lmb_{j}-1)\Lmb_{2}U_{1}+(\nu\rd_{\nu}\lmb_{j})(2\lmb_{j}-1)U_{1}\Big]_{\nu}\nonumber \\
 & \peq+[\Lmb\chi]_{\nu}(\rho)\lmb_{j}(\lmb_{j}-1)U_{\infty}(\lmb_{j};\rho)+[\nu\rd_{\nu}\td{\varphi}_{j}](\nu;\rho)\nonumber \\
 & \peq+(\nu\rd_{\nu}\lmb_{j})\chi_{\gtrsim\nu}(\rho)\rd_{\lmb}[\lmb(\lmb-1)U_{\infty}](\lmb_{j};\rho).\nonumber 
\end{align}
For the first line of RHS\eqref{eq:EigenFunc-03}, we can apply \eqref{eq:T1S1U1-bound},
\eqref{eq:T1S1U1-asymptotics}, and \eqref{eq:nudnu-eig-val}. For
the third and fourth lines of RHS\eqref{eq:EigenFunc-03}, we can
apply $|\lmb_{j}(\lmb_{j}-1)|\lesssim\frac{1}{|\log\nu|}$, \eqref{eq:U-infty-est},
and \eqref{eq:SharpEigftEst}. As a result, we have 
\begin{align}
\Big|\nu\rd_{\nu}\Big(\varphi_{j}-\frac{1}{\nu}\Lmb Q_{\nu}\Big) & -\nu\Big[-\lmb_{j}(\lmb_{j}-1)\Lmb_{2}U_{1}+(\nu\rd_{\nu}\lmb_{j})(2\lmb_{j}-1)U_{1}\Big]_{\nu}\Big|_{k}\label{eq:EigenFunc-04}\\
 & \lesssim_{k}\chf_{(0,\nu]}\frac{\rho^{3}}{\nu^{2}}+\chf_{[\nu,1]}\Big(\frac{\rho}{|\log\nu|^{2}}+\frac{\nu^{2}\langle\log(\frac{\rho}{\nu})\rangle}{\rho}\Big).\nonumber 
\end{align}
Next, we observe using the asymptotics \eqref{eq:EigenFunc-02}, 
\[
(\nu\rd_{\nu}\lmb_{j})(2\lmb_{j}-1)=\frac{-\frac{1}{6}+j}{|\log\nu|^{2}}+O\Big(\frac{1}{|\log\nu|^{3}}\Big),
\]
and \eqref{eq:T1S1U1-asymptotics} to have 
\begin{align*}
 & \chf_{(0,\nu^{-1}]}\big|-\lmb_{j}(\lmb_{j}-1)\Lmb_{2}U_{1}+(\nu\rd_{\nu}\lmb_{j})(2\lmb_{j}-1)U_{1}\big|_{k}\\
 & \lesssim_{k}\chf_{(0,1]}\frac{y^{3}}{|\log\nu|}+\chf_{[1,\nu^{-1}]}\frac{y\langle\log(\nu y)\rangle}{|\log\nu|^{2}}.
\end{align*}
Substituting this into \eqref{eq:EigenFunc-04} completes the proof
of \eqref{eq:EigenfuncEst-4}. The proof of Proposition~\ref{prop:FirstTwoEigenpairs}
is complete.
\end{proof}

\section{\label{sec:Formal-invariant-subspace}Formal invariant subspace decomposition
for $\mathbf{M}_{\nu}$}

In this section, we exhibit a formal invariant subspace decomposition
for the operator $\mathbf{M}_{\nu}$. The eigenfunctions $\bm{\varphi}_{j}$
constructed in the previous section form a two-dimensional space $\mathrm{span}\{\bm{\varphi}_{0},\bm{\varphi}_{1}\}$
that is invariant under $\mathbf{M}_{\nu}$. From now on, we aim to
find two linear functionals $\ell_{0},\ell_{1}$ such that the kernel
of $\ell_{0},\ell_{1}$ complements the space $\mathrm{span}\{\bm{\varphi}_{0},\bm{\varphi}_{1}\}$,
(i.e., it is of codimension two and transversal to $\mathrm{span}\{\bm{\varphi}_{0},\bm{\varphi}_{1}\}$)
\emph{and} is invariant under $\mathbf{M}_{\nu}$.

Note that we did not specify the total space, say $X$, for the invariant
subspace decomposition. However, whatever the total space $X$ is,
if $X$ contains the eigenfunctions $\bm{\varphi}_{j}$ and the linear
functionals $\ell_{j}$ are well-defined (and $X$ has a dense subspace
where $\mathbf{M}_{\nu}$ is defined), then one should have the invariant
subspace decomposition 
\[
X=\mathrm{span}\{\bm{\varphi}_{0},\bm{\varphi}_{1}\}\oplus\ker_{X}\{\ell_{0},\ell_{1}\},
\]
where $\ker_{X}\{\ell_{0},\ell_{1}\}$ is the kernel of $\ell_{0}$
and $\ell_{1}$ in $X$. In this paper, the total space $H^{2}\times H^{1}$
restricted in the region $\rho\leq1$ will work.

For the computational purposes in our later blow-up analysis, we need
to find \emph{explicit formula} of each $\ell_{j}$. This requires
a new input because the operator $\mathbf{M}_{\nu}$ is in general
\emph{non-self-adjoint} on usual (weighted-)Sobolev spaces. To compare
the situation with the parabolic case (the harmonic map heat flow),
the analogue of $\mathbf{M}_{\nu}$ takes the form 
\begin{equation}
\rd_{\rho\rho}+\frac{1}{\rho}\rd_{\rho}-\frac{1}{2}\rho\rd_{\rho}-\frac{V_{\nu}}{\rho^{2}},\label{eq:Formal-2}
\end{equation}
and it is already a classical fact that this operator is self-adjoint
in the weighted $L^{2}$-space $L^{2}(w\rho d\rho)$ with the Gaussian
weight 
\[
w(\rho)=e^{-\frac{1}{4}\rho^{2}}.
\]
Therefore, one can simply choose $\ell_{j}=\langle\cdot,\varphi_{j}\rangle_{L^{2}(w\rho d\rho)}$
with the eigenfunctions $\varphi_{j}$ for the operator \eqref{eq:Formal-2},
in the parabolic case. In fact, the self-adjointness says more; one
can also obtain a \emph{spectral gap} estimate.

Let us come back to our operator $\mathbf{M}_{\nu}$. There seems
to be no natural Hilbert space that makes $\mathbf{M}_{\nu}$ self-adjoint.
However, we can find an explicit formula for $\ell_{j}$: 
\begin{align*}
\ell_{j}(\bm{\eps}) & =\langle(\lmb_{j}+\Lmb_{0})\eps+\dot{\eps},g_{j}\varphi_{j}\rangle,\tag{\ref{eq:def-ell_j}}\\
g_{j}(\nu;\rho) & =\chf_{(0,1]}(\rho)\cdot(1-\rho^{2})^{\lmb_{j}-\frac{1}{2}}.\tag{\ref{eq:def-g_j}}
\end{align*}
Notice that $g_{j}$ is sharply localized in the interior of the light
cone $\rho\leq1$. In Proposition~\ref{prop:formal-inv-subsp-decomp}
below, we will check the invariance and transversality properties
of $\ell_{j}$.

We formally derive $\ell_{j}$ as follows. We notice that finding
a linear functional $\ell$ whose kernel is invariant under $\mathbf{M}_{\nu}$
is equivalent to finding an eigenfunction of the \emph{transpose}
$\mathbf{M}_{\nu}^{t}$ of $\mathbf{M}_{\nu}$. If $\mathbf{M}_{\nu}$
has a real eigenvalue $\lmb$, then the transpose $\mathbf{M}_{\nu}^{t}$
might also have the eigenvalue $\lmb$. Thus we want to find $\ell$
such that 
\[
(\mathbf{M}_{\nu}^{t}-\lmb)\ell=0.
\]
Now, by considering some weighted $L^{2}$-space $L^{2}(g\rho d\rho)$
($g$ will be chosen soon), we identify $\mathbf{M}_{\nu}^{t}$ by
$\mathbf{M}_{\nu}^{\ast_{g}}$, where $A^{\ast_{g}}$ denotes the
formal adjoint of $A$ with respect to the $L^{2}(g\rho d\rho)$-inner
product. We then compute $\mathbf{M}_{\nu}^{t}-\lmb$. We start from
writing $\mathbf{M}_{\nu}-\lmb$ into a triangular form: 
\[
\mathbf{M}_{\nu}-\lmb=\begin{bmatrix}1 & 0\\
\Lmb+\lmb & 1
\end{bmatrix}\begin{bmatrix}0 & 1\\
-H_{\nu}-(\Lmb_{0}+\lmb)(\Lmb+\lmb) & -\Lmb_{0}-\Lmb-2\lmb
\end{bmatrix}\begin{bmatrix}1 & 0\\
-\Lmb-\lmb & 1
\end{bmatrix}.
\]
Next, we take the $L^{2}(g\rho d\rho)$-adjoint. If we choose (c.f.
\eqref{eq:def-g_j})
\[
g(\rho)=(1-\rho^{2})^{\lmb-\frac{1}{2}},
\]
then the operator $-H_{\nu}-(\Lmb_{0}+\lmb)(\Lmb+\lmb)$ becomes formally
self-adjoint (c.f. \eqref{eq:Invariance-2} below) and hence we have
\[
\mathbf{M}_{\nu}^{\ast_{g}}-\lmb=\begin{bmatrix}1 & -\Lmb^{\ast_{g}}-\lmb\\
0 & 1
\end{bmatrix}\begin{bmatrix}0 & -H_{\nu}-(\Lmb_{0}+\lmb)(\Lmb+\lmb)\\
1 & -\Lmb_{0}^{\ast_{g}}-\Lmb^{\ast_{g}}-2\lmb
\end{bmatrix}\begin{bmatrix}1 & \Lmb^{\ast_{g}}+\lmb\\
0 & 1
\end{bmatrix},
\]
where $\Lmb^{\ast_{g}}$ and $\Lmb_{0}^{\ast_{g}}$ denote the formal
$L^{2}(g\rho d\rho)$-adjoints of $\Lmb$ and $\Lmb_{0}$. From this,
we see that 
\[
\begin{bmatrix}1 & \Lmb^{\ast_{g}}+\lmb\\
0 & 1
\end{bmatrix}^{-1}\begin{bmatrix}(\Lmb_{0}^{\ast_{g}}+\Lmb^{\ast_{g}}+2\lmb)\varphi_{j}\\
\varphi_{j}
\end{bmatrix}=\begin{bmatrix}(\Lmb_{0}^{\ast_{g}}+\lmb)\varphi_{j}\\
\varphi_{j}
\end{bmatrix}
\]
is a kernel element of $\mathbf{M}_{\nu}^{\ast_{g}}-\lmb$. This means
that the linear functional $\ell$ defined by 
\begin{align*}
\ell(\bm{f}) & =\langle f,(\Lmb_{0}^{\ast_{g}}+\lmb)\varphi\rangle_{L^{2}(g\rho d\rho)}+\langle\dot{f},\varphi_{j}\rangle_{L^{2}(g\rho d\rho)}\\
 & \qquad=\langle(\Lmb_{0}+\lmb)f+\dot{f},\varphi\rangle_{L^{2}(g\rho d\rho)}=\langle(\Lmb_{0}+\lmb)f+\dot{f},g\varphi\rangle_{L^{2}}
\end{align*}
is a kernel element of $\mathbf{M}_{\nu}^{t}-\lmb$. This gives the
expression \eqref{eq:def-ell_j}.
\begin{prop}[Formal invariant subspace decomposition]
\label{prop:formal-inv-subsp-decomp}\ 
\begin{itemize}
\item (Invariance) For smooth $\bm{\eps}$, we have
\begin{equation}
\ell_{j}(\mathbf{M}_{\nu}\bm{\eps})=\lmb_{j}\ell_{j}(\bm{\eps}).\label{eq:Invariance}
\end{equation}
\item (Transversality) We have
\begin{equation}
\begin{bmatrix}\ell_{0}(\bm{\varphi}_{0}) & \ell_{1}(\bm{\varphi}_{0})\\
\ell_{0}(\bm{\varphi}_{1}) & \ell_{1}(\bm{\varphi}_{1})
\end{bmatrix}=\begin{bmatrix}4|\log\nu|+O(1) & 0\\
0 & -4|\log\nu|+O(1)
\end{bmatrix}.\label{eq:Transv}
\end{equation}
\end{itemize}
\end{prop}

\begin{rem}[Algebraic multiplicity of $\lmb_{j}$]
Each eigenvalue $\lmb_{j}$ has algebraic multiplicity $1$. Note
that the geometric multiplicity is $1$ because of the unique existence
in Proposition~\ref{prop:FirstTwoEigenpairs}. To show that the algebraic
multiplicity is $1$, suppose that $(\mathbf{M}_{\nu}-\lmb_{j})\bm{\zeta}_{j}=\bm{\varphi}_{j}$
for some $\bm{\zeta}_{j}$. By the invariance \eqref{eq:Invariance}
one has $0=\ell_{j}((\mathbf{M}_{\nu}-\lmb_{j})\bm{\zeta}_{j})=\ell_{j}(\bm{\varphi}_{j})$.
This is absurd due to \eqref{eq:Transv}.
\end{rem}

\begin{rem}[Extension to exact self-similar solutions]
\label{rem:ell_j-extension-exact-self-sim}The above formal derivation
of $\ell_{j}$ can be extended to the linearized operators around
exact self-similar solutions to wave equations. For example, for $(1+d)$-dimensional
wave maps, one may replace $\Lmb_{0}$ by $\rho\rd_{\rho}+1$ and
$g_{j}\rho d\rho$ by $\chf_{(0,1]}(\rho)\cdot(1-\rho^{2})^{\lmb-\frac{d-1}{2}}\rho^{d-1}d\rho$
in the formula of $\ell_{j}$. Note however that there are some subtleties
to use this formula for high dimensions (e.g., $d\geq5$ for $\lmb=1$)
because the weight $g$ is not integrable near $\rho=1$. We believe
that one may perform some renormalization (e.g., integrating by parts
the weight $(1-\rho^{2})^{\lmb-\frac{d-1}{2}}$ to soften the singularity)
to obtain a correct form of $\ell$. Using this, one may also obtain
an explicit formula of the Riesz projection onto its discrete spectrum,
as well as the analogue of the previous remark.
\end{rem}

\begin{rem}[On dissipativity]
\label{rem:on-dissip}This remark requires the notation in Section~\ref{subsec:RR-blow-up-sol}.
For $\bm{f}$, consider the orthogonality condition 
\begin{equation}
\langle f,[\chi_{M}\Lmb Q]_{\nu}\rangle=0=\langle\dot{f},[\chi_{M}\Lmb Q]_{\nu}\rangle,\label{eq:dissip-orthog}
\end{equation}
where $M\gg1$ is some large fixed constant (see Remark \ref{rem:Remark-M}).
We expect that the modified energy functional from \cite[Lemma 6.5]{RaphaelRodnianski2012Publ.Math.}
suggests the existence of some inner product $\langle\bm{f},\bm{g}\rangle_{X_{\nu}}$
defined on the space of functions satisfying \eqref{eq:dissip-orthog}
at the level $H^{2}\times H^{1}$ with the dissipativity property
\begin{equation}
\langle\mathbf{M}_{\nu}\bm{f},\bm{f}\rangle_{X_{\nu}}\leq-(1+o_{\nu\to0}(1))\langle\bm{f},\bm{f}\rangle_{X_{\nu}}.\label{eq:dissip}
\end{equation}
This is why we called $(\lmb_{0},\bm{\varphi}_{0})$ and $(\lmb_{1},\bm{\varphi}_{1})$
the first two eigenpairs. However, the orthogonality condition \eqref{eq:dissip-orthog}
is \emph{not preserved} under the flow of $\rd_{\tau}-\mathbf{M}_{\nu}$.
We do not know how to obtain analogous dissipativity under the orthogonality
condition $\ell_{0}(\bm{f})=0=\ell_{1}(\bm{f})$ which are invariant
under the flow of $\rd_{\tau}-\mathbf{M}_{\nu}$. More technically,
the implicit constants in the coercivity estimates \eqref{eq:loc-coercivity-1}-\eqref{eq:loc-coercivity-2}
start to depend on $\nu$ under this orthogonality condition and could
be dangerous.

Note that in the parabolic case the dissipativity easily manifests
after identifying the eigenvalues, thanks to the self-adjoint structure
(under the weighted Gaussian measure). In the dispersive case, proving
the dissipativity is in fact a difficulty. There are some results
on such dissipativity or decay estimates for dispersive equations
in the self-similar coordinates (in various topologies), with applications
to the (conditional) stability problem of self-similar blow-up. See
for example \cite{DonningerSchorkhuber2017AIHP,Donninger2017Duke,MerleRaphaelRodnianskiSzeftel2019arXiv2}.
Note that we circumvent this problem in our type-II blow-up analysis
by assuming the result of \cite{RaphaelRodnianski2012Publ.Math.}.
\end{rem}

\begin{proof}[Proof of Proposition~\ref{prop:formal-inv-subsp-decomp}]
\ 

\textbf{Step 1. }Proof of the invariance \eqref{eq:Invariance}.

This follows from a direct computation. We start from writing 
\begin{align*}
\ell_{j}(\mathbf{M}_{\nu}\bm{\eps}) & =\big\langle(\lmb_{j}+\Lmb_{0})(-\Lmb\eps+\dot{\eps})+(-H_{\nu}\eps-\Lmb_{0}\dot{\eps}),g_{j}\varphi_{j}\big\rangle\\
 & =-\big\langle[H_{\nu}+(\lmb_{j}+\Lmb_{0})(\lmb_{j}+\Lmb)]\eps,g_{j}\varphi_{j}\big\rangle+\lmb_{j}\ell_{j}(\bm{\eps}).
\end{align*}
Thus the proof of \eqref{eq:Invariance} is complete if we show 
\[
\big\langle[H_{\nu}+(\lmb_{j}+\Lmb_{0})(\lmb_{j}+\Lmb)]\eps,g_{j}\varphi_{j}\big\rangle=0.
\]
Now, we expand 
\begin{align*}
 & H_{\nu}+(\lmb_{j}+\Lmb_{0})(\lmb_{j}+\Lmb)\\
 & =-\rd_{\rho\rho}-\frac{1}{\rho}\rd_{\rho}+\frac{V_{\nu}}{\rho^{2}}+(\rho\rd_{\rho}+\lmb_{j}+1)(\rho\rd_{\rho}+\lmb_{j})\\
 & =-(1-\rho^{2})\rd_{\rho\rho}-\Big\{\frac{1}{\rho}-(2\lmb_{j}+2)\rho\Big\}\rd_{\rho}+\frac{V_{\nu}}{\rho^{2}}+\lmb_{j}(\lmb_{j}+1).
\end{align*}
Using the operator identity (for $\rho<1$) 
\[
(1-\rho^{2})\rd_{\rho\rho}+\Big\{\frac{1}{\rho}-(2\lmb_{j}+2)\rho\Big\}\rd_{\rho}=\frac{1}{\rho g_{j}}\rd_{\rho}\rho(1-\rho^{2})g_{j}\rd_{\rho},
\]
we obtain the following operator identity (for $\rho<1$): 
\begin{equation}
H_{\nu}+(\lmb_{j}+\Lmb_{0})(\lmb_{j}+\Lmb)=-\frac{1}{\rho g_{j}}\rd_{\rho}\rho(1-\rho^{2})g_{j}\rd_{\rho}+\frac{V_{\nu}}{\rho^{2}}+\lmb_{j}(\lmb_{j}+1).\label{eq:Invariance-2}
\end{equation}
Although the function $g_{j}$ is possibly singular at $\rho=1$,
the factor $(1-\rho^{2})$ in the principal symbol is enough to compensate
this singularity. Thus the above identity says that $H_{\nu}+(\lmb_{j}+\Lmb_{0})(\lmb_{j}+\Lmb)$
is formally self-adjoint with respect to the measure $g_{j}\cdot\rho d\rho$.
As a result, we have 
\begin{align*}
 & \big\langle[H_{\nu}+(\lmb_{j}+\Lmb_{0})(\lmb_{j}+\Lmb)]\eps,g_{j}\varphi_{j}\big\rangle\\
 & \qquad=\big\langle\eps,g_{j}[H_{\nu}+(\lmb_{j}+\Lmb_{0})(\lmb_{j}+\Lmb)]\varphi_{j}\big\rangle=0,
\end{align*}
where in the last equality we used \eqref{eq:eigen-ft-rel}. This
completes the proof of \eqref{eq:Invariance}.

\textbf{Step 2. }Proof of transversality \eqref{eq:Transv}.

The proof of \eqref{eq:Transv} for the off-diagonal entries, i.e.,
$\ell_{j}(\bm{\varphi}_{k})=0$ for $j\neq k$, is an easy consequence
of \eqref{eq:Invariance}. Indeed, since $\bm{\varphi}_{k}$ is a
smooth eigenfunction of $\mathbf{M}_{\nu}$ with the eigenvalue $\lmb_{k}$,
we have 
\[
\lmb_{k}\ell_{j}(\bm{\varphi}_{k})=\ell_{j}(\mathbf{M}_{\nu}\bm{\varphi}_{k})=\lmb_{j}\ell_{j}(\bm{\varphi}_{k}),
\]
where in the second equality we used \eqref{eq:Invariance}. Since
$\lmb_{0}\neq\lmb_{1}$, we have $\ell_{j}(\bm{\varphi}_{k})=0$ for
$j\neq k$.

Next, we show \eqref{eq:Transv} for the diagonal entries. Let $j\in\{0,1\}$.
We first claim that 
\begin{equation}
\ell_{j}(\bm{\varphi}_{j})=\Big\langle[2\lmb_{j}-1+2\Lmb_{0}]\frac{1}{\nu}\Lmb Q_{\nu},g_{j}\frac{1}{\nu}\Lmb Q_{\nu}\Big\rangle+O\Big(\frac{1}{|\log\nu|}\Big).\label{eq:Transv-1}
\end{equation}
To show this, we note by \eqref{eq:def-ell_j} and \eqref{eq:bmphi-to-phi}
that 
\[
\ell_{j}(\bm{\varphi}_{j})=\big\langle[2\lmb_{j}-1+2\Lmb_{0}]\varphi_{j},g_{j}\varphi_{j}\big\rangle.
\]
We rewrite this as 
\begin{align*}
\ell_{j}(\bm{\varphi}_{j}) & =\Big\langle[2\lmb_{j}-1+2\Lmb_{0}]\frac{1}{\nu}\Lmb Q_{\nu},g_{j}\frac{1}{\nu}\Lmb Q_{\nu}\Big\rangle\\
 & \peq+\Big\langle[2\lmb_{j}-1+2\Lmb_{0}]\frac{1}{\nu}\Lmb Q_{\nu},g_{j}\Big(\varphi_{j}-\frac{1}{\nu}\Lmb Q_{\nu}\Big)\Big\rangle\\
 & \peq+\Big\langle(2\lmb_{j}-1+2\Lmb_{0})\Big(\varphi_{j}-\frac{1}{\nu}\Lmb Q_{\nu}\Big),g_{j}\varphi_{j}\Big\rangle.
\end{align*}
In the above display, the last two terms are errors; we use \eqref{eq:EigenFuncEst-2},
$|\frac{1}{\nu}\Lmb Q_{\nu}|_{1}+|\varphi_{j}|_{1}\aleq\frac{1}{\rho}$,
and $g_{j}\aleq\chf_{(0,1]}\cdot(1-\rho)^{-\frac{3}{4}}$ to see that
these terms are of size $O(\frac{1}{|\log\nu|})$. Therefore, \eqref{eq:Transv-1}
follows.

Using the formula \eqref{eq:LmbQ-formula} for $\Lmb Q$ and the improved
decay \eqref{eq:LmbQ-asymptotics} for $\Lmb_{0}\Lmb Q$, we have\footnote{In Lemma~\ref{lem:conv-comp-for-mod-eqn} below, we will compute
these quantities more precisely.} 
\begin{align*}
\Big\langle\frac{1}{\nu}\Lmb Q_{\nu},g_{j}\frac{1}{\nu}\Lmb Q_{\nu}\Big\rangle & =4|\log\nu|+O(1),\\
\Big\langle\frac{1}{\nu}\Lmb_{0}\Lmb Q_{\nu},g_{j}\frac{1}{\nu}\Lmb Q_{\nu}\Big\rangle & =O(1).
\end{align*}
Substituting the above display and $|\lmb_{j}-(1-j)|\aleq\frac{1}{|\log\nu|}$
into \eqref{eq:Transv-1} completes the proof of \eqref{eq:Transv}.
\end{proof}
Computations with $O(1)$ errors in \eqref{eq:Transv} will not be
sufficient in the later refined modulation estimates, and we will
need the following more precise computations of relevant quantities.
\begin{lem}[Computations adapted to modulation equations]
\label{lem:conv-comp-for-mod-eqn}For each $j\in\{0,1\}$, define
\begin{align}
\frkb_{j} & \coloneqq\ell_{j}\Big(\frac{1}{\nu}\bm{\Lmb Q}_{\nu}\Big),\label{eq:def-frkb}\\
\frkc_{j} & \coloneqq\frac{1}{2}\Big\{\ell_{j}(\bm{\varphi}_{0}-\bm{\varphi}_{1})+\Big\langle\frac{1}{\nu}\Lmb Q_{\nu},g_{j}\varphi_{j}\Big\rangle\Big\}.\label{eq:def-frkc}
\end{align}
Then, we have the following. 
\begin{itemize}
\item (Expansions of $\frkb_{j}$ and $\frkc_{j}$) We have 
\begin{equation}
\left\{ \begin{aligned}\frkb_{0} & =4|\log\nu|+\Big(-\frac{14}{3}+4\log2\Big)+O\Big(\frac{1}{|\log\nu|}\Big),\\
\frkb_{1} & =-\frac{4}{3}+\frac{-\frac{5}{9}}{|\log\nu|}+O\Big(\frac{1}{|\log\nu|^{2}}\Big),
\end{aligned}
\right.\label{eq:frkb-asymp}
\end{equation}
and 
\begin{equation}
\left\{ \begin{aligned}\frkc_{0} & =4|\log\nu|+\Big(-\frac{14}{3}+4\log2\Big)+O\Big(\frac{1}{|\log\nu|}\Big),\\
\frkc_{1} & =4|\log\nu|+\Big(-\frac{2}{3}+4\log2\Big)+O\Big(\frac{1}{|\log\nu|}\Big).
\end{aligned}
\right.\label{eq:frkc-asymp}
\end{equation}
\item (Bounds for $\nu\rd_{\nu}\frkb_{j}$ and $\nu\rd_{\nu}\frkc_{j}$)
For each $j\in\{0,1\}$, we have 
\begin{equation}
|\nu\rd_{\nu}\frkb_{j}|+|\nu\rd_{\nu}\frkc_{j}|\aleq1.\label{eq:nudnu-frk}
\end{equation}
\item (One more computation) We have 
\begin{equation}
\ell_{j}(\nu\rd_{\nu}(\bm{\varphi}_{0}-\bm{\varphi}_{1}))-\Big\langle\frac{1}{\nu}\Lmb_{0}\Lmb Q_{\nu},g_{j}\varphi_{j}\Big\rangle=-4+O\Big(\frac{1}{|\log\nu|}\Big).\label{eq:OneMoreComp}
\end{equation}
\end{itemize}
\end{lem}

\begin{proof}
\textbf{Step 1.} Computations of $\langle\frac{1}{\nu}\Lmb Q_{\nu},g_{j}\frac{1}{\nu}\Lmb Q_{\nu}\rangle$
and $\langle\frac{1}{\nu}\Lmb_{0}\Lmb Q_{\nu},g_{j}\frac{1}{\nu}\Lmb Q_{\nu}\rangle$.

As seen in the proof of \eqref{eq:Transv}, the key quantities are
$\langle\frac{1}{\nu}\Lmb Q_{\nu},g_{j}\frac{1}{\nu}\Lmb Q_{\nu}\rangle$
and $\langle\frac{1}{\nu}\Lmb_{0}\Lmb Q_{\nu},g_{j}\frac{1}{\nu}\Lmb Q_{\nu}\rangle$.
In this step, we claim 
\begin{align}
\Big\langle\frac{1}{\nu}\Lmb Q_{\nu},g_{j}\frac{1}{\nu}\Lmb Q_{\nu}\Big\rangle & =4|\log\nu|-2+2(\psi(1)-\psi(\lmb_{j}+\tfrac{1}{2}))+O(\nu^{2}|\log\nu|),\label{eq:ExactComp-1}\\
\Big\langle\frac{1}{\nu}\Lmb_{0}\Lmb Q_{\nu},g_{j}\frac{1}{\nu}\Lmb Q_{\nu}\Big\rangle & =2+O(\nu^{2}|\log\nu|).\label{eq:ExactComp-2}
\end{align}

We first show \eqref{eq:ExactComp-1}. We write 
\begin{align*}
\Big\langle\frac{1}{\nu}\Lmb Q_{\nu},g_{j}\frac{1}{\nu}\Lmb Q_{\nu}\Big\rangle & =\Big\langle\frac{1}{\nu}\Lmb Q_{\nu},\chf_{(0,1]}\frac{1}{\nu}\Lmb Q_{\nu}\Big\rangle+\Big\langle\frac{1}{\nu}\Lmb Q_{\nu},(g_{j}-\chf_{(0,1]})\frac{1}{\nu}\Lmb Q_{\nu}\Big\rangle\\
 & =\langle\Lmb Q,\chf_{(0,\nu^{-1}]}\Lmb Q\rangle+\Big\langle\frac{2}{\rho},(g_{j}-\chf_{(0,1]})\frac{2}{\rho}\Big\rangle+O(\nu^{2}|\log\nu|),
\end{align*}
where in the last equality we used $|g_{j}-\chf_{(0,1]}|\aleq\rho^{2}(1-\rho^{2})^{-\frac{3}{4}}$.
We then compute 
\begin{align*}
\langle\Lmb Q,\chf_{(0,\nu^{-1}]}\Lmb Q\rangle=\int_{0}^{\frac{1}{\nu}}\frac{4y^{3}}{(1+y^{2})^{2}}dy & =2\log\Big(1+\frac{1}{\nu^{2}}\Big)-\frac{2}{1+\nu^{2}}\\
 & =4|\log\nu|-2+O(\nu^{2})
\end{align*}
and 
\begin{align*}
\Big\langle\frac{2}{\rho},(g_{j}-\chf_{(0,1]})\frac{2}{\rho}\Big\rangle & =4\int_{0}^{1}((1-\rho^{2})^{\lmb_{j}-\frac{1}{2}}-1)\frac{d\rho}{\rho}\\
 & =2\int_{0}^{1}(z^{\lmb_{j}-\frac{1}{2}}-1)\frac{dz}{1-z}=2(\psi(1)-\psi(\lmb_{j}+\tfrac{1}{2})),
\end{align*}
where in the last equality we used the series expansion $\frac{1}{1-z}=1+z+z^{2}+\cdots$
and \eqref{eq:digamma-props}. This completes the proof of \eqref{eq:ExactComp-1}.

For the proof of \eqref{eq:ExactComp-2}, we can use the improved
spatial decay $|\Lmb_{0}\Lmb Q|\aleq\langle y\rangle^{-3}$ to have
\begin{align*}
 & \Big\langle\frac{1}{\nu}\Lmb_{0}\Lmb Q_{\nu},g_{j}\frac{1}{\nu}\Lmb Q_{\nu}\Big\rangle\\
 & =\Big\langle\frac{1}{\nu}\Lmb_{0}\Lmb Q_{\nu},\chf_{(0,1]}\frac{1}{\nu}\Lmb Q_{\nu}\Big\rangle+\Big\langle\frac{1}{\nu}\Lmb_{0}\Lmb Q_{\nu},(g_{j}-\chf_{(0,1]})\frac{1}{\nu}\Lmb Q_{\nu}\Big\rangle\\
 & =\langle\Lmb_{0}\Lmb Q,\chf_{(0,\nu^{-1}]}\Lmb Q\rangle+O(\nu^{2}|\log\nu|)\\
 & =2+O(\nu^{2}|\log\nu|).
\end{align*}
This completes the proof of \eqref{eq:ExactComp-2}.

\textbf{Step 2.} Expansions for $\langle\frac{1}{\nu}\Lmb Q_{\nu},g_{j}\varphi_{j}\rangle$
and $\langle\frac{1}{\nu}\Lmb_{0}\Lmb Q_{\nu},g_{j}\varphi_{j}\rangle$.

In this step, we claim 
\begin{align}
\Big\langle\frac{1}{\nu}\Lmb Q_{\nu},g_{j}\varphi_{j}\Big\rangle & =4|\log\nu|+(4j-6+4\log2)+O\Big(\frac{1}{|\log\nu|}\Big),\label{eq:ExactComp-3}\\
\Big\langle\frac{1}{\nu}\Lmb_{0}\Lmb Q_{\nu},g_{j}\varphi_{j}\Big\rangle & =2+O(\nu^{2}|\log\nu|).\label{eq:ExactComp-4}
\end{align}
To show \eqref{eq:ExactComp-3}, we use \eqref{eq:EigenFuncEst-2}
and $\frac{1}{\nu}\Lmb Q_{\nu}\aleq\frac{1}{\rho}$ to have 
\[
\Big\langle\frac{1}{\nu}\Lmb Q_{\nu},g_{j}\varphi_{j}\Big\rangle=\Big\langle\frac{1}{\nu}\Lmb Q_{\nu},g_{j}\varphi_{j}\Big\rangle+O\Big(\frac{1}{|\log\nu|}\Big).
\]
Applying \eqref{eq:ExactComp-1} and \eqref{eq:digamma-values} to
the above yields \eqref{eq:ExactComp-3}. To show \eqref{eq:ExactComp-4},
we use \eqref{eq:EigenFuncEst-2} and $\chf_{(0,\nu]}|\frac{1}{\nu}\Lmb_{0}\Lmb Q_{\nu}|\aleq\chf_{(0,\nu]}\frac{\rho}{\nu^{2}}+\chf_{[\nu,1]}\frac{\nu^{2}}{\rho^{3}}$
to have 
\[
\Big\langle\frac{1}{\nu}\Lmb_{0}\Lmb Q_{\nu},g_{j}\varphi_{j}\Big\rangle=\Big\langle\frac{1}{\nu}\Lmb_{0}\Lmb Q_{\nu},g_{j}\frac{1}{\nu}\Lmb Q_{\nu}\Big\rangle+O(\nu^{2}|\log\nu|).
\]
Applying \eqref{eq:ExactComp-2} to the above yields \eqref{eq:ExactComp-4}.

\textbf{Step 3.} Proof of \eqref{eq:frkb-asymp} and \eqref{eq:frkc-asymp}.

For \eqref{eq:frkb-asymp}, we note that 
\begin{equation}
\frkb_{j}=\ell_{j}\Big(\frac{1}{\nu}\bm{\Lmb Q}_{\nu}\Big)=\lmb_{j}\Big\langle\frac{1}{\nu}\Lmb Q_{\nu},g_{j}\varphi_{j}\Big\rangle+\Big\langle\frac{1}{\nu}\Lmb_{0}\Lmb Q_{\nu},g_{j}\varphi_{j}\Big\rangle.\label{eq:ExactComp-5}
\end{equation}
Substituting \eqref{eq:lmb-hat-asymp} and \eqref{eq:ExactComp-3}-\eqref{eq:ExactComp-4}
into the above completes the proof of \eqref{eq:frkb-asymp} (Notice
that $|\lmb_{1}|\aleq\frac{1}{|\log\nu|}$ is used to have the improved
$O(\frac{1}{|\log\nu|^{2}})$ error for $\frkb_{1}$).

For \eqref{eq:frkc-asymp}, we recall that 
\[
\frkc_{j}=\frac{1}{2}\Big\{\ell_{j}(\bm{\varphi}_{0}-\bm{\varphi}_{1})+\big\langle\frac{1}{\nu}\Lmb Q_{\nu},g_{j}\varphi_{j}\big\rangle\Big\}.
\]
Using $\ell_{j}(\bm{\varphi}_{k})=0$ for $j\neq k$, we have 
\begin{align}
\frkc_{j} & =\frac{1}{2}\Big\{(-1)^{j}\ell_{j}(\bm{\varphi}_{j})+\big\langle\frac{1}{\nu}\Lmb Q_{\nu},g_{j}\varphi_{j}\big\rangle\Big\}\nonumber \\
 & =\frac{1}{2}\big\langle(-1)^{j}(2\lmb_{j}-1+2\Lmb_{0})\varphi_{j}+\frac{1}{\nu}\Lmb Q_{\nu},g_{j}\varphi_{j}\big\rangle.\label{eq:ExactComp-6}
\end{align}
We rewrite this as 
\begin{align*}
\frkc_{j} & =\frac{1}{2}\Big\langle\big(1+(-1)^{j}(2\lmb_{j}-1)\big)\frac{1}{\nu}\Lmb Q_{\nu},g_{j}\varphi_{j}\Big\rangle+(-1)^{j}\Big\langle\frac{1}{\nu}\Lmb_{0}\Lmb Q_{\nu},g_{j}\varphi_{j}\Big\rangle\\
 & \peq+\frac{(-1)^{j}}{2}\Big\langle[2\lmb_{j}-1+2\Lmb_{0}]\Big(\varphi_{j}-\frac{1}{\nu}\Lmb Q_{\nu}\Big),g_{j}\varphi_{j}\Big\rangle.
\end{align*}
The last term of the above display is of size $O(\frac{1}{|\log\nu|})$,
thanks to \eqref{eq:EigenFuncEst-2} and 
\begin{equation}
|g_{j}\varphi_{j}|\aleq\chf_{(0,\frac{1}{2}]}\rho^{-1}+\chf_{[\frac{1}{2},1]}\cdot(1-\rho)^{-\frac{3}{4}}.\label{eq:ExactComp-6-2}
\end{equation}
Therefore, we have proved that 
\begin{align*}
\frkc_{0} & =\lmb_{0}\Big\langle\frac{1}{\nu}\Lmb Q_{\nu},g_{0}\varphi_{0}\Big\rangle+\Big\langle\frac{1}{\nu}\Lmb_{0}\Lmb Q_{\nu},g_{0}\varphi_{0}\Big\rangle+O\Big(\frac{1}{|\log\nu|}\Big),\\
\frkc_{1} & =(1-\lmb_{1})\Big\langle\frac{1}{\nu}\Lmb Q_{\nu},g_{1}\varphi_{1}\Big\rangle-\Big\langle\frac{1}{\nu}\Lmb_{0}\Lmb Q_{\nu},g_{1}\varphi_{1}\Big\rangle+O\Big(\frac{1}{|\log\nu|}\Big).
\end{align*}
Substituting \eqref{eq:lmb-hat-asymp} and \eqref{eq:ExactComp-3}-\eqref{eq:ExactComp-4}
into the above completes the proof of \eqref{eq:frkc-asymp}.

\textbf{Step 4.} Proof of \eqref{eq:nudnu-frk}.

Let us first note that 
\[
\nu\rd_{\nu}g_{j}=(\nu\rd_{\nu}\lmb_{j})\chf_{(0,1]}\cdot(1-\rho^{2})^{\lmb_{j}-\frac{1}{2}}\log(1-\rho^{2}).
\]
Using \eqref{eq:nudnu-eig-val}, we have the following estimate: 
\begin{equation}
|(\nu\rd_{\nu}\lmb_{j})g_{j}|+|\nu\rd_{\nu}g_{j}|\aleq\frac{1}{|\log\nu|^{2}}\chf_{(0,1]}\cdot(1-\rho)^{-\frac{3}{4}}.\label{eq:ExactComp-9}
\end{equation}

We show \eqref{eq:nudnu-frk} for $\frkb_{j}$. Taking $\nu\rd_{\nu}$
to \eqref{eq:ExactComp-5}, we have 
\begin{align}
\nu\rd_{\nu}\frkb_{j}=(\nu\rd_{\nu}\lmb_{j})\Big\langle\frac{1}{\nu}\Lmb Q_{\nu},g_{j}\varphi_{j}\Big\rangle & +\Big\langle(\lmb_{j}+\Lmb_{0})\frac{1}{\nu}\Lmb Q_{\nu},(\nu\rd_{\nu}g_{j})\varphi_{j}\Big\rangle\label{eq:ExactComp-7}\\
 & \peq+\Big\langle(\lmb_{j}+\Lmb_{0})\frac{1}{\nu}\Lmb Q_{\nu},g_{j}(\nu\rd_{\nu}\varphi_{j})\Big\rangle.\nonumber 
\end{align}
By \eqref{eq:ExactComp-9} and \eqref{eq:EigenFuncEst-1}, the first
line of RHS\eqref{eq:ExactComp-7} is of size $O(\frac{1}{|\log\nu|})$.
By \eqref{eq:EigenFuncEst-3}, the second line of RHS\eqref{eq:ExactComp-7}
is of size $O(1)$. This shows \eqref{eq:nudnu-frk} for $\frkb_{j}$.

We turn to show \eqref{eq:nudnu-frk} for $\frkc_{j}$. Taking $\nu\rd_{\nu}$
to \eqref{eq:ExactComp-6}, we have 
\begin{align}
\nu\rd_{\nu}\frkc_{j} & =(-1)^{j}\big\langle(\nu\rd_{\nu}\lmb_{j})\varphi_{j},g_{j}\varphi_{j}\big\rangle\label{eq:ExactComp-8}\\
 & \peq+\frac{1}{2}\big\langle(-1)^{j}(2\lmb_{j}-1+2\Lmb_{0})\varphi_{j}+\frac{1}{\nu}\Lmb Q_{\nu},(\nu\rd_{\nu}g_{j})\varphi_{j}\big\rangle\nonumber \\
 & \peq+\frac{1}{2}\big\langle(-1)^{j}[2\lmb_{j}-1+2\Lmb_{0}](\nu\rd_{\nu}\varphi_{j})-\frac{1}{\nu}\Lmb_{0}\Lmb Q_{\nu},g_{j}\varphi_{j}\big\rangle\nonumber \\
 & \peq+\frac{1}{2}\big\langle(-1)^{j}(2\lmb_{j}-1+2\Lmb_{0})\varphi_{j}+\frac{1}{\nu}\Lmb Q_{\nu},g_{j}(\nu\rd_{\nu}\varphi_{j})\big\rangle.\nonumber 
\end{align}
By \eqref{eq:ExactComp-9} and \eqref{eq:EigenFuncEst-1}, the first
two lines of RHS\eqref{eq:ExactComp-8} are of size $O(\frac{1}{|\log\nu|})$.
By \eqref{eq:EigenFuncEst-3}, the last two lines of RHS\eqref{eq:ExactComp-8}
are of size $O(1)$. This shows \eqref{eq:nudnu-frk} for $\frkc_{j}$.

\textbf{Step 5.} Proof of \eqref{eq:OneMoreComp}.

To prove \eqref{eq:OneMoreComp}, note that 
\begin{align*}
\text{LHS\eqref{eq:OneMoreComp}} & =\Big\langle(\lmb_{j}+\Lmb_{0})\nu\rd_{\nu}(\varphi_{0}-\varphi_{1})\\
 & \qquad+\nu\rd_{\nu}\big((\lmb_{0}+\Lmb)\varphi_{0}-(\lmb_{1}+\Lmb)\varphi_{1}\big)-\frac{1}{\nu}\Lmb_{0}\Lmb Q_{\nu},g_{j}\varphi_{j}\Big\rangle
\end{align*}
We rearrange this as 
\begin{align}
\text{LHS\eqref{eq:OneMoreComp}} & =\Big\langle-(\lmb_{0}-\lmb_{1}+1)\frac{1}{\nu}\Lmb_{0}\Lmb Q_{\nu},g_{j}\varphi_{j}\Big\rangle\label{eq:decomposition-temp6}\\
 & \peq+\Big\langle[\lmb_{j}+\Lmb_{0}+\Lmb]\nu\rd_{\nu}(\varphi_{0}-\varphi_{1}),g_{j}\varphi_{j}\Big\rangle\nonumber \\
 & \peq+\Big\langle\lmb_{0}\nu\rd_{\nu}\Big(\varphi_{0}-\frac{1}{\nu}\Lmb Q_{\nu}\Big)-\lmb_{1}\nu\rd_{\nu}\Big(\varphi_{1}-\frac{1}{\nu}\Lmb Q_{\nu}\Big),g_{j}\varphi_{j}\Big\rangle\nonumber \\
 & \peq+\Big\langle(\nu\rd_{\nu}\lmb_{0})\varphi_{0}-(\nu\rd_{\nu}\lmb_{1})\varphi_{1},g_{j}\varphi_{j}\Big\rangle.\nonumber 
\end{align}
By \eqref{eq:ExactComp-4}, the first line of RHS\eqref{eq:decomposition-temp6}
becomes 
\[
=-(\lmb_{0}-\lmb_{1}+1)\Big\{2+O\Big(\frac{1}{|\log\nu|}\Big)\Big\}=-4+O\Big(\frac{1}{|\log\nu|}\Big).
\]
The remaining terms are considered as errors. Indeed, using \eqref{eq:ExactComp-6-2}
and \eqref{eq:EigenfuncEst-4}, the second and third lines of RHS\eqref{eq:decomposition-temp6}
are of size $O(\frac{1}{|\log\nu|^{2}})$. Using \eqref{eq:nudnu-eig-val}
and \eqref{eq:EigenFuncEst-1}, the last line of RHS\eqref{eq:decomposition-temp6}
is of size $O(\frac{1}{|\log\nu|})$. This completes the proof of
\eqref{eq:OneMoreComp}. The proof is complete.
\end{proof}

\section{\label{sec:Blow-up-analysis}Blow-up analysis}

In this section, we prove Theorem~\ref{thm:MainThm}. We will perform
modulation analysis upon the result of Raphaël--Rodnianski \cite{RaphaelRodnianski2012Publ.Math.}.
In particular, we directly work with a finite-time blow-up solution
$\bm{u}$ as in Theorem~\ref{thm:RR-blow-up-intro}, whose construction
is established in \cite{RaphaelRodnianski2012Publ.Math.}. We then
combine this a priori information on $\bm{u}$ and the linear analysis
of the operator $\mathbf{M}_{\nu}$ in Sections \ref{sec:First-two-eigenpairs}
and \ref{sec:Formal-invariant-subspace} to derive refined modulation
equations, and hence the sharp blow-up rate in Theorem~\ref{thm:MainThm}.

We begin by recalling Raphaël--Rodnianski blow-up solutions (Section~\ref{subsec:RR-blow-up-sol}).
In Section~\ref{subsec:Decomposition-of-solutions}, we introduce
a new decomposition 
\[
\bm{v}=\bm{P}+\bm{\eps},
\]
where $\bm{v}$ is the renormalized solution in the self-similar coordinates
and $\bm{P}(\nu,b;\rho)$ is a new modified profile, which is carefully
chosen to resemble the modified profiles in \cite{RaphaelRodnianski2012Publ.Math.}.
In Section~\ref{subsec:Modulation-estimates}, we test the evolution
equation for $\bm{\eps}$ against $\ell_{j}$ to obtain refined modulation
equations for $\nu$ and $b$. These two subsections are the points
where all the properties of $\lmb_{j},\varphi_{j},\ell_{j}$ (Propositions
\ref{prop:FirstTwoEigenpairs} and \ref{prop:formal-inv-subsp-decomp})
as well as the explicit computations (Lemma~\ref{lem:conv-comp-for-mod-eqn})
gather. Finally, in Section~\ref{subsec:Proof-of-Theorem}, we integrate
the resulting modulation equations to finish the proof of Theorem~\ref{thm:MainThm}.

\subsection{\label{subsec:RR-blow-up-sol}Raphaël--Rodnianski blow-up solutions}

In this subsection, we recall refined properties of Raphaël--Rodnianski
blow-up solutions constructed in \cite{RaphaelRodnianski2012Publ.Math.}.
The readers familiar with it may assume Corollary~\ref{cor:RR-blow-up-self-similar}
below and jump to Section~\ref{subsec:Decomposition-of-solutions}.

Raphaël--Rodnianski blow-up solutions can be roughly described as
follows. One decomposes a solution $u$ to \eqref{eq:WM} as 
\[
u(t,r)=P^{\RR}(b(t);\frac{r}{\lmb(t)})+w^{\RR}(t,r),
\]
where $b=-\lmb_{t}$, $P^{\RR}(b;\cdot)$ is some modified profile
satisfying $P^{\RR}(0;\cdot)=Q$, and $w^{\RR}$ is an error (or,
radiative) part of $u$. At the initial time $t=0$, one assumes that
$w^{\RR}(0)$ is sufficiently small and $u_{t}(0)\approx\frac{b}{\lmb}\Lmb Q_{\lmb}$
so that $u$ contracts forwards in time.

The modified profile $P^{\RR}$ almost mimics the profile of formal
self-similar solutions to \eqref{eq:WM} (i.e., the solution to \eqref{eq:WM}
with $b(t)\equiv b$ constant), but it \emph{is different} at the
self-similar scale $\frac{r}{\lmb}\sim b^{-1}$; the profile $P^{\RR}$
as well as the formal $b_{t}$-equation 
\begin{equation}
\frac{b_{t}}{\lmb}+\frac{b^{2}}{2|\log b|}\approx0\label{eq:formal-bt-eq}
\end{equation}
are derived so that the profile $P^{\RR}(b;y)$ sufficiently degenerates
at the self-similar scale $y\sim b^{-1}$. This is the \emph{tail
computation} developed through the seminal works \cite{RaphaelRodnianski2012Publ.Math.,MerleRaphaelRodnianski2013InventMath,MerleRaphaelRodnianski2015CambJMath}.

On the other hand, the radiative part $w^{\RR}$ is controlled forwards
in time with the help of the \emph{repulsivity} of some conjugated
linearized operator. In more detail, let us begin with the \emph{self-dual
factorization} of $H_{\lmb}$: 
\[
H_{\lmb}=A_{\lmb}^{\ast}A_{\lmb},
\]
where $A_{\lmb}$ is a first-order differential operator 
\[
A_{\lmb}=\rd_{r}-\Big(\frac{1-(r/\lmb)^{2}}{1+(r/\lmb)^{2}}\Big)\frac{1}{r}.
\]
Then, the super-symmetric conjugate of $H_{\lmb}$ defined by 
\[
\td H_{\lmb}\coloneqq A_{\lmb}A_{\lmb}^{\ast}=-\rd_{rr}-\frac{1}{r}\rd_{r}+\frac{\td V_{\lmb}}{r^{2}},\quad\text{where}\quad\td V=\frac{4}{1+y^{2}},
\]
is nonnegative and \emph{repulsive} \cite{RodnianskiSterbenz2010Ann.Math.}:
\[
\td V_{\lmb}\geq0,\quad\rd_{r}\td V_{\lmb}\leq0,\quad\rd_{t}\td V_{\lmb}\leq0,
\]
where the last one is guaranteed in the blow-up regime $b=-\lmb_{t}>0$.
The conjugated linearized operator $\td H_{\lmb}$ naturally arises
after conjugating $A_{\lmb}$ to the equation for $w^{\RR}$. The
repulsivity enables certain Morawetz-type monotonicity for $A_{\lmb}w^{\RR}$
and $\rd_{t}(A_{\lmb}w^{\RR})$ and hence the radiative part $w^{\RR}$
can be controlled forwards in time; see \cite{RodnianskiSterbenz2010Ann.Math.,RaphaelRodnianski2012Publ.Math.}
for more details. After a careful bootstrap argument, one in particular
justifies the formal dynamics $\lmb_{t}=-b$ and \eqref{eq:formal-bt-eq},
which lead to the blow-up rate \eqref{eq:RR-blow-up-rate-intro}.

Let us simply write $A=A_{1}$, $\td H=\td H_{1}$, and so on. From
now on, we recall the precise definition of the modified profile in
\cite{RaphaelRodnianski2012Publ.Math.}. For $0<b\ll1$, this modified
profile is defined in the following form 
\begin{equation}
P^{\RR}(b;y)\coloneqq(1-\chi_{B_{1}}(y))\pi+\chi_{B_{1}}(y)(Q(y)+b^{2}T^{\RR}(b;y)),\label{eq:def-PRR}
\end{equation}
where $B_{1}\coloneqq|\log b|/b$ is a localization radius and $T^{\RR}(b;y)$
is some corrector to be defined in \eqref{eq:def-TRR} below. Define
\begin{equation}
\td T^{\RR}(b;y)\coloneqq J_{1}(y)\int_{1}^{y}g(b;y')J_{2}(y')y'dy'-J_{2}(y)\int_{0}^{y}g(b;y')J_{1}(y')y'dy',\label{eq:def-T-tilde-RR}
\end{equation}
where $J_{1}$ and $J_{2}$ are as in \eqref{eq:def-J1J2}, and 
\begin{align}
g & \coloneqq-\Lmb_{0}\Lmb Q+c_{b}\Lmb Q\chi_{B_{0}/4},\label{eq:def-g}\\
c_{b} & \coloneqq\frac{\langle\Lmb_{0}\Lmb Q,\Lmb Q\rangle}{\langle\chi_{\frac{B_{0}}{4}}\Lmb Q,\Lmb Q\rangle}=\frac{1}{2|\log b|}+O(\frac{1}{|\log b|^{2}}),\label{eq:def-cb}\\
B_{0} & \coloneqq1/(b\sqrt{3\tint{}{}y\chi(y)dy}),\label{eq:def-B0}
\end{align}
so that $H\td T^{\RR}=g$. We note that $\langle g,J_{1}\rangle=0$
so the $\int_{0}^{y}$-integral in \eqref{eq:def-T-tilde-RR} can
be replaced by the $-\int_{y}^{\infty}$-integral if necessary. Next,
we define $T^{\RR}=T^{\RR}(b;y)$ by 
\begin{equation}
T^{\RR}\coloneqq\td T^{\RR}-\frac{\langle\td T^{\RR},\chi_{M}\Lmb Q\rangle}{\langle\Lmb Q,\chi_{M}\Lmb Q\rangle}\Lmb Q,\label{eq:def-TRR}
\end{equation}
where $M>1$ is some fixed large constant (whose role is explained
in Remark \ref{rem:Remark-M} below) so that $T^{\RR}$ additionally
satisfies the orthogonality condition $\langle T^{\RR},\chi_{M}\Lmb Q\rangle=0$.
Using our notation \eqref{eq:def-HQ-inv}, we may express $T^{\RR}$
in an alternative form: 
\begin{equation}
T^{\RR}=H^{-1}g-c_{M,b}\Lmb Q,\quad\text{where}\quad c_{M,b}\coloneqq\frac{\langle H^{-1}g,\chi_{M}\Lmb Q\rangle}{\langle\Lmb Q,\chi_{M}\Lmb Q\rangle}.\label{eq:TRR-alt-def}
\end{equation}
With this $T^{\RR}$, the modified profile $P^{\RR}$ is defined by
\eqref{eq:def-PRR}. We are now ready to state refined properties
of the Raphaël--Rodnianski blow-up solutions of Theorem~\ref{thm:RR-blow-up-intro}.
\begin{prop}[Raphaël--Rodnianski stable blow-up \cite{RaphaelRodnianski2012Publ.Math.}]
\label{prop:RR-blow-up-original}There exists an open (initial data)
set $\bm{\calO}$ in $\bm{\calH}_{Q}^{2}$ (see \eqref{eq:def-H2Q})
with the following properties. For any $\bm{u}_{0}=(u_{0},\dot{u}_{0})\in\bm{\calO}$,
its forward-in-time solution $\bm{u}=(u,\dot{u})$ to \eqref{eq:WM}
blows up in finite time $T=T(\bm{u})\in(0,+\infty)$ with the following
descriptions:
\begin{itemize}
\item (Decomposition of the solution) There exist $0<b^{\ast}\ll1$ and
$C^{1}$ functions $(\wh{\lmb}(t),\wh b(t)):[0,T)\to(0,\infty)\times(0,b^{\ast})$
such that the solution $\bm{u}$ admits the decomposition 
\begin{equation}
u(t,r)=P^{\RR}(\wh b(t);\frac{r}{\wh{\lmb}(t)})+w^{\RR}(t,r)\label{eq:RR-decomp}
\end{equation}
with
\begin{align}
\langle w^{\RR},(\chi_{M}\Lmb Q)_{\wh{\lmb}}\rangle & =0,\label{eq:RR-orthog}\\
\wh{\lmb}_{t}+\wh b & =0.\label{eq:RR-b-def}
\end{align}
\item (Estimates for the modulation parameters $\wh{\lmb}$ and $\wh b$)
We have 
\begin{equation}
\wh{\lmb}|\wh b_{t}|\aleq\frac{\wh b^{2}}{|\log\wh b|}.\label{eq:b_t-RR-est}
\end{equation}
Moreover, as $t\to T^{-}$, we have 
\begin{gather}
\Big|\frac{\wh{\lmb}(t)}{T-t}-\wh b(t)\Big|\aleq\frac{\wh b(t)}{|\log\wh b(t)|},\label{eq:lmb-b-rel}\\
\wh{\lmb}(t)=(T-t)e^{-\sqrt{|\log(T-t)|}+O(1)}.\label{eq:lmb-RR-asymp}
\end{gather}
\item (Smallness of $w^{\RR}$) Let $W^{\RR}(t)\coloneqq A_{\lmb(t)}w^{\RR}(t)$.
We have 
\begin{align}
\|w^{\RR}\|_{\dot{H}_{1}^{1}}+\|\rd_{t}w^{\RR}\|_{L^{2}} & =o_{b^{\ast}\to0}(1),\label{eq:wRR-est-1}\\
\wh{\lmb}^{2}\int\Big(|\rd_{t}W^{\RR}|^{2}+|\rd_{r}W^{\RR}|^{2}+\frac{\td V_{\wh{\lmb}}}{r^{2}}|W^{\RR}|^{2}\Big) & \aleq\wh b^{4},\label{eq:wRR-est-2}\\
\wh{\lmb}^{2}\int\chf_{(0,4\wh{\lmb}/\wh b]}\Big(|\rd_{t}W^{\RR}|^{2}+|\rd_{r}W^{\RR}|^{2}+\frac{\td V_{\wh{\lmb}}}{r^{2}}|W^{\RR}|^{2}\Big) & \aleq\frac{\wh b^{4}}{|\log\wh b|^{2}},\label{eq:wRR-est-3}
\end{align}
\end{itemize}
\end{prop}

\begin{rem}
\label{rem:Remark-M}We are implicitly assuming that a large constant
$M$ is fixed, and then there exists small $b^{\ast}=b^{\ast}(M)>0$
satisfying the above statements. The large constant $M$ is fixed
in order to have coercivity estimates (see \eqref{eq:coercivity-1}
and \eqref{eq:coercivity-2} below) and to gain a smallness factor
to close the bootstrap argument in \cite{RaphaelRodnianski2012Publ.Math.}.
\end{rem}

\begin{proof}
All the statements are essentially proved in \cite{RaphaelRodnianski2012Publ.Math.}.
For the decomposition part \eqref{eq:RR-decomp}-\eqref{eq:RR-b-def},
see Section~5.2 of \cite{RaphaelRodnianski2012Publ.Math.}. For \eqref{eq:b_t-RR-est}
and \eqref{eq:wRR-est-1}-\eqref{eq:wRR-est-3}, see (5.34)-(5.36)
and Lemma~6.1 of \cite{RaphaelRodnianski2012Publ.Math.}. Next, \eqref{eq:lmb-RR-asymp}
is the blow-up rate proved in \cite{RaphaelRodnianski2012Publ.Math.}.
Finally, the estimate \eqref{eq:lmb-b-rel} is not mentioned in \cite{RaphaelRodnianski2012Publ.Math.},
but it can be proved in a very similar manner; see Appendix \ref{sec:Proof-lmb-b-rel}
of this paper.
\end{proof}
For our blow-up analysis, (i) we rewrite the above proposition in
the first-order formulation and in terms of the self-similar variables,
and (ii) derive Hardy-type controls on the error part from the bounds
\eqref{eq:wRR-est-2} and \eqref{eq:wRR-est-3}.

For the former purpose, we also define the profile $\dot{P}^{\RR}(b;y)$
by 
\begin{align}
P^{\RR}(b;y) & =(1-\chi_{B_{1}}(y))\pi+\chi_{B_{1}}(y)(Q(y)+b^{2}T^{\RR}(b;y)),\tag{\ref{eq:def-PRR}}\nonumber \\
\dot{P}^{\RR}(b;y) & \coloneqq b\Lmb P^{RR}(b;y),\label{eq:def-Pdot-RR}
\end{align}
and set $\bm{P}^{\RR}$ as a vector $(P^{\RR},\dot{P}^{\RR})$.

For the latter purpose, we recall the adapted function spaces $\dot{\calH}_{1}^{2}$
and $\dot{H}_{1}^{1}$ introduced in Section~\ref{subsec:Main-results}:
\begin{align*}
\|f\|_{\dot{\calH}_{1}^{2}}^{2} & \coloneqq\|\rd_{yy}f\|_{L^{2}}^{2}+\Big\|\frac{1}{y\langle\log y\rangle}|f|_{-1}\Big\|_{L^{2}}^{2},\\
\|g\|_{\dot{H}_{1}^{1}}^{2} & \coloneqq\|\rd_{y}g\|_{L^{2}}^{2}+\Big\|\frac{1}{y}g\Big\|_{L^{2}}^{2}.
\end{align*}
These function spaces were introduced to satisfy the following coercivity
estimates (see \cite[Appendix B]{RaphaelRodnianski2012Publ.Math.}
and also Appendix \ref{sec:Localized-Hardy-inequalities} of this
paper): 
\begin{align}
\|A^{\ast}Af\|_{L^{2}} & \sim_{M}\|f\|_{\dot{\calH}_{1}^{2}},\qquad\forall f\in\dot{\calH}_{1}^{2}\cap\{\chi_{M}\Lmb Q\}^{\perp},\label{eq:coercivity-1}\\
\|Ag\|_{L^{2}} & \sim_{M}\|g\|_{\dot{H}_{1}^{1}},\qquad\forall g\in\dot{H}_{1}^{1}\cap\{\chi_{M}\Lmb Q\}^{\perp},\label{eq:coercivity-2}
\end{align}
provided that $M\gg1$.

Our blow-up analysis will be done in the backward lightcone $\rho\leq1$.
Thus we also define localized versions of the aforementioned norms
\begin{align*}
\|f\|_{(\dot{\calH}_{1}^{2})_{R}}^{2} & \coloneqq\|\chf_{(0,R]}\rd_{yy}f\|_{L^{2}}^{2}+\Big\|\chf_{(0,R]}\frac{1}{y\langle\log y\rangle}|f|_{-1}\Big\|_{L^{2}}^{2},\\
\|g\|_{(\dot{H}_{1}^{1})_{R}}^{2} & \coloneqq\|\chf_{(0,R]}\rd_{y}g\|_{L^{2}}^{2}+\Big\|\chf_{(0,R]}\frac{1}{y}g\Big\|_{L^{2}}^{2},
\end{align*}
as well as the scaled version of the local $\dot{\calH}_{1}^{2}$-norm
(restricted to the region $\rho\leq1$)
\begin{align*}
\|f\|_{(\dot{\calH}_{1}^{2})_{1}^{\nu}}^{2} & \coloneqq\|\chf_{(0,1]}\rd_{\rho\rho}f\|_{L^{2}}^{2}+\Big\|\chf_{(0,1]}\frac{1}{\rho\langle\log(\frac{\rho}{\nu})\rangle}|f|_{-1}\Big\|_{L^{2}}^{2}.
\end{align*}
With these localized norms we still have analogous coercivity estimates
for $R\gg M$: (see Appendix \ref{sec:Localized-Hardy-inequalities}
for the proof) 
\begin{align}
\|\chf_{(0,R]}A^{\ast}Af\|_{L^{2}} & \sim_{M}\|f\|_{(\dot{\calH}_{1}^{2})_{R}},\qquad\forall f\in\dot{\calH}_{1}^{2}\cap\{\chi_{M}\Lmb Q\}^{\perp},\label{eq:loc-coercivity-1}\\
\|\chf_{(0,R]}Ag\|_{L^{2}} & \sim_{M}\|g\|_{(\dot{H}_{1}^{1})_{R}},\qquad\forall g\in\dot{H}_{1}^{1}\cap\{\chi_{M}\Lmb Q\}^{\perp},\label{eq:loc-coercivity-2}
\end{align}
where the implicit constants can be chosen \emph{uniformly in $R$}.
The above coercivity estimates allow us to transfer the controls \eqref{eq:wRR-est-1}-\eqref{eq:wRR-est-3}
to the $\bm{\calH}^{2}$-norm of the remainder. As a result, we have
the following corollary.
\begin{cor}[Raphaël--Rodnianski blow-up solutions in self-similar variables]
\label{cor:RR-blow-up-self-similar}Let $\bm{u}=(u,\dot{u})$ be
a finite-time blow-up solution as in Proposition~\ref{prop:RR-blow-up-original}
and $T\in(0,\infty)$ be its blow-up time. Let $(\tau,\rho)$ be the
self-similar coordinates \eqref{eq:def-self-similar-variables} with
this $T$. Define $\bm{v}=(v,\dot{v})$ by 
\[
u(t,r)=v(\tau,\rho),\qquad\dot{u}(t,r)=\frac{1}{T-t}\dot{v}(\tau,\rho),
\]
and the modulation parameters $\wh{\nu}$ and $\wh b$ by\footnote{Abuse of notation for $\wh b$}
\[
\wh{\nu}(\tau)\coloneqq\frac{\wh{\lmb}(t)}{T-t}\quad\text{and}\quad\wh b(\tau)\coloneqq\wh b(t).
\]
Then, the following hold for all large $\tau$:
\begin{itemize}
\item (Decomposition) $\bm{v}$ admits the decomposition 
\[
\bm{v}(\tau,\rho)=\big[\bm{P}^{\RR}(\wh b(\tau);\cdot)+\bm{\eps}^{\RR}(\tau,\cdot)\big]_{\wh{\nu}(\tau)}(\rho).
\]
\item (Estimates for the modulation parameters $\wh{\nu}$ and $\wh b$)
As $\tau\to\infty$, we have 
\begin{align}
\Big|\frac{\wh{\nu}_{\tau}}{\wh{\nu}}\Big|+\Big|\frac{\wh b_{\tau}}{\wh b}\Big|+\Big|\frac{\wh b}{\wh{\nu}}-1\Big| & \aleq\frac{1}{|\log\wh{\nu}|},\label{eq:nu-b-RR-bound}\\
\wh{\nu}(\tau)\sim\wh b(\tau) & \sim e^{-\sqrt{\tau}}.\label{eq:nu-RR-asymp}
\end{align}
\item (Smallness of $\bm{\eps}^{\RR}$) We have 
\begin{align}
\|\eps^{\RR}\|_{\dot{H}_{1}^{1}}+\|\dot{\eps}^{\RR}\|_{L^{2}} & =o_{b^{\ast}\to0}(1),\label{eq:eRR-est1}\\
\|\eps^{\RR}\|_{\dot{\calH}_{1}^{2}}+\|\dot{\eps}^{\RR}\|_{\dot{H}_{1}^{1}} & \aleq\wh b^{2},\label{eq:eRR-est2}\\
\|\eps^{\RR}\|_{(\dot{\calH}_{1}^{2})_{2/\wh b}}+\|\dot{\eps}^{\RR}\|_{(\dot{H}_{1}^{1})_{2/\wh b}} & \aleq\frac{\wh b^{2}}{|\log\wh b|}.\label{eq:eRR-est3}
\end{align}
\end{itemize}
\end{cor}

The proof is given in Appendix \ref{sec:Proof-of-Corollary}.

\subsection{\label{subsec:Decomposition-of-solutions}Decomposition of solutions}

From now on, we pick a solution $\bm{v}$ as  in Corollary~\ref{cor:RR-blow-up-self-similar}
and work with that solution. Moreover, by considering only large $\tau$,
we may assume $\wh b(\tau)\ll\min\{b^{\ast},\nu^{\ast}\}$ as well.

The goal of this subsection (Proposition~\ref{prop:decomposition})
is to write $\bm{v}$ in a slightly different way: 
\[
\bm{v}(\tau,\rho)=\bm{P}(\nu(\tau),b(\tau);\rho)+\bm{\eps}(\tau,\rho),
\]
where $\bm{P}=\bm{P}(\nu,b;\rho)$ is a new modified profile that
satisfies $\bm{P}(\nu,0;\cdot)=\bm{Q}_{\nu}$ and is described in
terms of $\bm{Q}_{\nu}$ and $\bm{\varphi}_{j}$. Using such a profile
will enable more refined modulation estimates for $\nu$ and $b$
in view of the computations in Lemma~\ref{lem:conv-comp-for-mod-eqn}.
At the same time, we also want to use a priori information on $\bm{v}$
in Corollary~\ref{cor:RR-blow-up-self-similar}. Thus our modified
profile $\bm{P}$ will be chosen sufficiently close to $\bm{P}^{\RR}$
(Lemma~\ref{lem:ProfDiff}) so that the bounds of $\bm{\eps}^{\RR}$
can be transferred to those of $\bm{\eps}$ of our new decomposition.
Our $\bm{\eps}$ does not satisfy any orthogonality conditions, but
its smallness (c.f. \eqref{eq:eRR-est3}) will be sufficient.

From the above discussion, we define the \emph{modified profile} $\bm{P}=\bm{P}(\nu,b;\rho)$
by 
\begin{equation}
\bm{P}=\bm{Q}_{\nu}+\frac{b}{2}\Big(\bm{\varphi}_{0}-\bm{\varphi}_{1}+\begin{bmatrix}0\\
\frac{1}{\nu}\Lmb Q_{\nu}
\end{bmatrix}\Big).\label{eq:def-mod-prof}
\end{equation}
Denoting $\bm{P}=(P,\dot{P})$, this is equivalent to saying that
\begin{align*}
P & \coloneqq Q_{\nu}+\frac{b}{2}(\varphi_{0}-\varphi_{1}),\\
\dot{P} & \coloneqq\frac{b}{2}\Big((\lmb_{0}+\Lmb)\varphi_{0}-(\lmb_{1}+\Lmb)\varphi_{1}+\frac{1}{\nu}\Lmb Q_{\nu}\Big).
\end{align*}
The main proposition of this subsection is the following.
\begin{prop}[Decomposition]
\label{prop:decomposition}There exist $C^{1}$ functions $\nu(\tau)$
and $b(\tau)$ defined for all large $\tau$ such that 
\begin{equation}
\bm{v}(\tau,\rho)=\bm{P}(\nu(\tau),b(\tau);\rho)+\bm{\eps}(\tau,\rho)\label{eq:decomposition}
\end{equation}
with the following properties:
\begin{itemize}
\item (Estimates for the modulation parameters) We have 
\begin{align}
\Big|\frac{\nu_{\tau}}{\nu}\Big|+\Big|\frac{b_{\tau}}{b}\Big|+\Big|\frac{b}{\nu}-1\Big| & \aleq\frac{1}{|\log\nu|},\label{eq:nu-b-bound}\\
\nu(\tau)\sim b(\tau) & \sim e^{-\sqrt{\tau}}.\label{eq:rough-asymp-nu-b}
\end{align}
\item (Bounds for the remainder inside the light cone) We have 
\begin{align}
\|\eps\|_{(\dot{\calH}_{1}^{2})_{1}^{\nu}}+\|\dot{\eps}\|_{(\dot{H}_{1}^{1})_{1}} & \aleq\frac{\nu}{|\log\nu|},\label{eq:eps-bound-1}\\
\chf_{(0,1]}\Big\{\frac{1}{\langle\log(\frac{\rho}{\nu})\rangle^{\frac{1}{2}}}|\eps|_{-1}+|\dot{\eps}|\Big\} & \aleq\frac{\nu}{|\log\nu|},\label{eq:eps-bound-2}
\end{align}
and 
\begin{align}
|\ell_{j}(\bm{\eps})| & \aleq\frac{\nu}{|\log\nu|^{\frac{1}{2}}}.\label{eq:eps-bound-3}\\
|[\nu\rd_{\nu}\ell_{j}](\bm{\eps})| & \aleq\frac{\nu}{|\log\nu|^{\frac{5}{2}}}.\label{eq:eps-bound-4}
\end{align}
\end{itemize}
\end{prop}

First, let us show that the profiles $\bm{P}^{\RR}=(P^{\RR},\dot{P}^{\RR})$
and $\bm{P}$ are close enough in the self-similar region $\rho\leq1$,
after changing the parameters $\nu$ and $b$ slightly.
\begin{lem}[Difference estimate for profiles]
\label{lem:ProfDiff}For $\wh b$ and $\wh{\nu}$ satisfying 
\[
0<\wh b,\wh{\nu}\ll\min\{b^{\ast},\nu^{\ast}\}\quad\text{and}\quad\Big|\frac{\wh b}{\wh{\nu}}-1\Big|\aleq\frac{1}{|\log\wh{\nu}|},
\]
define $\nu$ and $b$ by the relations (see \eqref{eq:TRR-alt-def}
for the definition of $c_{M,\wh b}$)
\begin{equation}
\log\Big(\frac{\nu}{\wh{\nu}}\Big)=-c_{M,\wh b}\wh b^{2}\quad\text{and}\quad\wh b=b\cdot\frac{1}{2}(\lmb_{0}(\nu)-\lmb_{1}(\nu)+1).\label{eq:Rel-nu-and-nuhat}
\end{equation}
Then, the following hold.
\begin{itemize}
\item (Parameter changes are small) We have 
\begin{align}
\Big|\frac{\nu}{\wh{\nu}}-1\Big|+\Big|\frac{b}{\wh{\nu}}-1\Big|+\Big|\frac{\wh b}{\wh{\nu}}-1\Big| & \aleq\frac{\nu}{|\log\nu|},\label{eq:ProfDiff-Param1}\\
\Big|\frac{\nu}{\wh{\nu}}-1\Big| & \aleq|c_{M,\wh b}|\nu^{2}\aleq\nu^{2-}.\label{eq:ProfDiff-Param2}
\end{align}
\item (Derivative estimates) We have 
\begin{equation}
\begin{pmatrix}\rd_{\wh{\nu}}\nu & \rd_{\wh b}\nu\\
\rd_{\wh{\nu}}b & \rd_{\wh b}b
\end{pmatrix}=\begin{pmatrix}1 & 0\\
0 & 1
\end{pmatrix}+O\Big(\frac{1}{|\log\nu|}\Big).\label{eq:ProfDiff-Param3}
\end{equation}
\item (Main difference estimate) We have 
\begin{equation}
\Big\| P^{\RR}(\wh b;\frac{\rho}{\wh{\nu}})-P(\nu,b;\rho)\Big\|_{(\dot{\calH}_{1}^{2})_{1}^{\nu}}+\Big\|\frac{1}{\wh{\nu}}P^{\RR}(\wh b;\frac{\rho}{\wh{\nu}})-\dot{P}(\nu,b;\rho)\Big\|_{(\dot{H}_{1}^{1})_{1}}\aleq\frac{\nu}{|\log\nu|}.\label{eq:ProfDiff}
\end{equation}
\end{itemize}
\end{lem}

\begin{proof}
\textbf{Step 1.} Bounds of $H^{-1}g$ and $\wh b\rd_{\wh b}(H^{-1}g)$
and structure of $H^{-1}g$.

Recall the definitions \eqref{eq:def-g}-\eqref{eq:def-B0}, where
$b$ is replaced by $\wh b$. In other words, we will use $g=g(\wh b;y)$,
$c_{\wh b}$, and $B_{0}=B_{0}(\wh b)$. In this step, we claim the
following pointwise bounds on $H^{-1}g$ and $\wh b\rd_{\wh b}(H^{-1}g)$:
\begin{align}
|H^{-1}g|_{2} & \aleq\chf_{(0,1]}y^{3}+\chf_{[1,+\infty)}\frac{y\langle\log_{-}(\wh by)\rangle}{|\log\wh b|},\label{eq:ProfDiffStep1-1}\\
|\wh b\rd_{\wh b}(H^{-1}g)|_{2} & \aleq\chf_{(0,1]}\frac{y^{3}}{|\log\wh b|^{2}}+\chf_{[1,B_{0}/4]}\frac{y}{|\log\wh b|}+\chf_{[B_{0}/4,+\infty)}\frac{1}{\wh b^{2}|\log\wh b|y},\label{eq:ProfDiffStep1-2}
\end{align}
and the structure of $H^{-1}g$: 
\begin{equation}
|H^{-1}g-(S_{1}-c_{\wh b}U_{1})|_{2}\aleq\chf_{[B_{0}/4,+\infty)}\frac{y\langle\log(\wh by)\rangle}{|\log\wh b|}.\label{eq:ProfDiffStep1-3}
\end{equation}

Let us first record the following pointwise bounds on $g$ and $\wh b\rd_{\wh b}g$,
which are immediate consequences of \eqref{eq:def-g}, \eqref{eq:def-cb},
and $|\wh b\rd_{\wh b}c_{\wh b}|\aleq\frac{1}{|\log\wh b|^{2}}$:
\begin{align*}
|g| & \aleq\chf_{(0,1]}y+\chf_{[1,+\infty)}\frac{1}{y^{3}}+\chf_{[1,B_{0}/2]}\frac{1}{|\log\wh b|y},\\
|\wh b\rd_{\wh b}g| & \aleq\chf_{(0,1]}\frac{y}{|\log\wh b|^{2}}+\chf_{[1,B_{0}/4]}\frac{1}{|\log\wh b|^{2}y}+\chf_{[B_{0}/4,B_{0}/2]}\frac{1}{|\log\wh b|y}.
\end{align*}

Now we show \eqref{eq:ProfDiffStep1-1} and \eqref{eq:ProfDiffStep1-2}.
Note that $\wh b\rd_{\wh b}(H^{-1}g)=H^{-1}(\wh b\rd_{\wh b}g)$.
In the region $y\leq1$, taking $H^{-1}$ is essentially the multiplication
by $y^{2}$ and hence \eqref{eq:ProfDiffStep1-1} and \eqref{eq:ProfDiffStep1-2}
follow. Henceforth, let us focus on the region $y\geq1$. Since $g$
and $\wh b\rd_{\wh b}g$ satisfy the orthogonality conditions $\langle g,J_{1}\rangle=\langle\wh b\rd_{\wh b}g,J_{1}\rangle=0$
(recall that $J_{1}=\Lmb Q$), we have 
\[
H^{-1}F=J_{1}\int_{0}^{y}FJ_{2}y'dy'+J_{2}\int_{y}^{+\infty}FJ_{1}y'dy',\qquad\text{for }F\in\{g,\wh b\rd_{\wh b}g\}.
\]
The first integral is estimated by (for $y\geq1$)
\begin{align*}
\Big|J_{1}\int_{0}^{y}gJ_{2}y'dy'\Big|_{2} & \aleq\frac{\langle\log y\rangle}{y}+\frac{y}{|\log\wh b|},\\
\Big|J_{1}\int_{0}^{y}(\wh b\rd_{\wh b}g)J_{2}y'dy'\Big|_{2} & \aleq\frac{y}{|\log\wh b|^{2}}+\chf_{[B_{0}/4,+\infty)}\frac{1}{\wh b^{2}|\log\wh b|y}.
\end{align*}
The second integral is estimated by (for $y\geq1$) 
\begin{align*}
\Big|J_{2}\int_{y}^{+\infty}gJ_{1}y'dy'\Big|_{2} & \aleq\frac{1}{y}+\chf_{[1,B_{0}/2]}\frac{y\langle\log(\wh by)\rangle}{|\log\wh b|},\\
\Big|J_{2}\int_{y}^{+\infty}(\wh b\rd_{\wh b}g)J_{1}y'dy'\Big|_{2} & \aleq\chf_{[1,B_{0}/2]}\frac{y}{|\log\wh b|}.
\end{align*}
Summing up the above estimates yields \eqref{eq:ProfDiffStep1-1}
and \eqref{eq:ProfDiffStep1-2} for $y\geq1$. The proof of \eqref{eq:ProfDiffStep1-1}
and \eqref{eq:ProfDiffStep1-2} is complete.

We turn to show \eqref{eq:ProfDiffStep1-3}. By the definition \eqref{eq:def-T1S1U1},
we have 
\[
H^{-1}g-(S_{1}-c_{\wh b}U_{1})=H^{-1}\big(g+(\Lmb_{0}\Lmb Q-c_{\wh b}\Lmb Q)\big)=-c_{\wh b}H^{-1}(\chi_{\ageq B_{0}/4}\Lmb Q).
\]
Applying \eqref{eq:def-cb} and \eqref{eq:MappingHQinv}, we have
\[
|H^{-1}g-(S_{1}-c_{\wh b}U_{1})|_{2}\aleq\chf_{[B_{0}/4,+\infty)}\frac{y\langle\log(\wh by)\rangle}{|\log\wh b|},
\]
completing the proof of \eqref{eq:ProfDiffStep1-3}.

\textbf{Step 2.} Estimates for $c_{M,\wh b}$ and $\wh b\rd_{\wh b}c_{M,\wh b}$.

Recall the definition \eqref{eq:TRR-alt-def} of $c_{M,\wh b}$. By
\eqref{eq:ProfDiffStep1-1} and \eqref{eq:ProfDiffStep1-2}, we in
particular have 
\begin{align*}
\chf_{(0,2M]}|H^{-1}g| & \aleq\chf_{(0,1]}y^{3}+\chf_{[1,2M]}y\\
\chf_{(0,2M]}|\wh b\rd_{\wh b}(H^{-1}g)| & \aleq\chf_{(0,1]}\frac{y^{3}}{|\log\wh b|^{2}}+\chf_{[1,2M]}\frac{y}{|\log\wh b|}.
\end{align*}
Therefore, we have the following estimates for $c_{M,\wh b}$ and
$\wh b\rd_{\wh b}c_{M,\wh b}$.
\begin{align}
|c_{M,\wh b}| & \aleq\frac{M^{2}}{\log M}=O_{M}(1),\label{eq:ProfDiffStep2-1}\\
|\wh b\rd_{\wh b}c_{M,\wh b}| & \aleq\frac{1}{|\log\wh b|}\cdot\frac{M^{2}}{\log M}=O_{M}\Big(\frac{1}{|\log\wh b|}\Big).\label{eq:ProfDiffStep2-2}
\end{align}

\textbf{Step 3.} Proof of \eqref{eq:ProfDiff-Param1}-\eqref{eq:ProfDiff-Param3}.

We note that \eqref{eq:ProfDiff-Param1} and \eqref{eq:ProfDiff-Param2}
are immediate from \eqref{eq:ProfDiffStep2-1} and $\lmb_{0}-\lmb_{1}+1=2+O(\frac{1}{|\log\nu|})$.
Henceforth, we show \eqref{eq:ProfDiff-Param3}. First, we differentiate
the first equation of \eqref{eq:Rel-nu-and-nuhat}, and use \eqref{eq:ProfDiff-Param1}
and \eqref{eq:ProfDiffStep2-1}-\eqref{eq:ProfDiffStep2-2} to have
\begin{align*}
\rd_{\wh{\nu}}\nu & =\frac{\nu}{\wh{\nu}}=1+\Big(\frac{1}{|\log\nu|}\Big),\\
\rd_{\wh b}\nu & =-\nu((\rd_{\wh b}c_{M,\wh b})\wh b^{2}+2c_{M,\wh b}\wh b)=O_{M}(\nu^{2})=O(\nu^{2-}).
\end{align*}
Next, we differentiate the second equation of \eqref{eq:Rel-nu-and-nuhat},
and use the above display, \eqref{eq:ProfDiff-Param1}, and \eqref{eq:nudnu-eig-val}
to have 
\begin{align*}
\rd_{\wh{\nu}}b & =\frac{-\frac{1}{2}b(\rd_{\wh{\nu}}\nu)\rd_{\nu}(\lmb_{0}-\lmb_{1})}{\frac{1}{2}(\lmb_{0}(\nu)-\lmb_{1}(\nu)+1)}=O\Big(\frac{1}{|\log\nu|^{2}}\Big),\\
\rd_{\wh b}b & =\frac{1-\frac{1}{2}b(\rd_{\wh b}\nu)\rd_{\nu}(\lmb_{0}-\lmb_{1})}{\frac{1}{2}(\lmb_{0}(\nu)-\lmb_{1}(\nu)+1)}=1+O\Big(\frac{1}{|\log\nu|}\Big).
\end{align*}
This completes the proof of \eqref{eq:ProfDiff-Param3}.

\textbf{Step 4.} Difference of $P^{\RR}$ and $P$.

In this step, we show \eqref{eq:ProfDiff} for $P$. We are only interested
in the estimates inside the light cone, i.e., $\rho\leq1$. Thus the
cutoff $\chi_{B_{1}}$ in the definition of $P^{\RR}$ does not play
any role. For $\rho\leq1$, we note by the definitions \eqref{eq:def-PRR},
\eqref{eq:TRR-alt-def}, and \eqref{eq:def-mod-prof} that 
\begin{align}
 & P^{\RR}(\wh b;\frac{\rho}{\wh{\nu}})-P(\nu,b;\rho)\label{eq:ProfDiffStep4-1}\\
 & =\Big(Q_{\wh{\nu}}+\wh b^{2}(H^{-1}g-c_{M,\wh b}\Lmb Q)_{\wh{\nu}}\Big)-\Big(Q_{\nu}+\frac{b}{2}[\varphi_{0}-\varphi_{1}](\nu;\rho)\Big).\nonumber 
\end{align}
For the first term of RHS\eqref{eq:ProfDiffStep4-1}, we use \eqref{eq:ProfDiffStep1-3}
in the form
\[
\|H^{-1}g-(S_{1}-c_{\wh b}U_{1})\|_{(\dot{\calH}_{1}^{2})_{1/\nu}}\aleq\Big\|\chf_{[B_{0}/4,1/\nu]}\frac{1}{|\log\wh b|y}\Big\|_{L^{2}}\lesssim\frac{1}{|\log\nu|}
\]
to have 
\begin{equation}
\wh b^{2}[H^{-1}g]_{\wh{\nu}}=\wh b^{2}\Big(S_{1}-c_{\wh b}U_{1}\Big)_{\wh{\nu}}+O_{(\dot{\calH}_{1}^{2})_{1}^{\nu}}\Big(\frac{\nu}{|\log\nu|}\Big).\label{eq:ProfDiffStep4-6}
\end{equation}
For the second term of RHS\eqref{eq:ProfDiffStep4-1}, we use \eqref{eq:EigenFunc-01}
to have 
\begin{equation}
\frac{1}{2}(\varphi_{0}-\varphi_{1})=\nu[\frks S_{1}+\frku U_{1}]_{\nu}+O_{(\dot{\calH}_{1}^{2})_{1}^{\nu}}\Big(\frac{1}{|\log\nu|}\Big).\label{eq:ProfDiffStep4-7}
\end{equation}
where 
\begin{align*}
\frks(\nu) & \coloneqq\lmb_{0}-\lmb_{1}=1+O\Big(\frac{1}{|\log\nu|}\Big),\\
\frku(\nu) & \coloneqq\frac{1}{2}\big(\lmb_{0}(\lmb_{0}-1)-\lmb_{1}(\lmb_{1}-1)\big)=-\frac{1}{2|\log\nu|}+O\Big(\frac{1}{|\log\nu|^{2}}\Big).
\end{align*}
Substituting \eqref{eq:ProfDiffStep4-6} and \eqref{eq:ProfDiffStep4-7}
into \eqref{eq:ProfDiffStep4-1}, we have 
\begin{align}
 & P^{\RR}(\wh b;\frac{\rho}{\wh{\nu}})-P(\nu,b;\rho)\label{eq:ProfDiffStep4-2}\\
 & =\big(Q_{\wh{\nu}}-Q_{\nu}-c_{M,\wh b}\wh b^{2}\Lmb Q_{\wh{\nu}}\big)+\big(\wh b^{2}[S_{1}]_{\wh{\nu}}-\frks b\nu[S_{1}]_{\nu}\big)\nonumber \\
 & \peq-\Big(\wh b^{2}c_{\wh b}[U_{1}]_{\wh{\nu}}+\frku b\nu[U_{1}]_{\nu}\Big)+O_{(\dot{\calH}_{1}^{2})_{1}^{\nu}}\Big(\frac{\nu}{|\log\nu|}\Big).\nonumber 
\end{align}
Therefore, it suffices to show that each term of RHS\eqref{eq:ProfDiffStep4-2}
is of $O_{(\dot{\calH}_{1}^{2})_{1}^{\nu}}(\frac{\nu}{|\log\nu|})$.

For the first term of RHS\eqref{eq:ProfDiffStep4-2}, we note that
\[
Q_{\wh{\nu}}-Q_{\nu}-c_{M,\wh b}\wh b^{2}\Lmb Q_{\wh{\nu}}=(Q-Q_{\nu/\wh{\nu}}-c_{M,\wh b}\wh b^{2}\Lmb Q)_{\wh{\nu}}.
\]
Thanks to our relation \eqref{eq:Rel-nu-and-nuhat}, we can write
\[
Q-Q_{\nu/\wh{\nu}}-c_{M,\wh b}\wh b^{2}\Lmb Q=\int_{\nu/\wh{\nu}}^{1}(\Lmb Q_{a}-\Lmb Q)\frac{da}{a}=\int_{\nu/\wh{\nu}}^{1}\int_{a}^{1}\Lmb\Lmb Q_{a'}\frac{da'}{a'}\frac{da}{a}.
\]
Hence by \eqref{eq:ProfDiff-Param2} we have 
\begin{equation}
\begin{aligned} & \big\|(Q-Q_{\nu/\wh{\nu}}-c_{M,\wh b}\wh b^{2}\Lmb Q)_{\wh{\nu}}\big\|_{(\dot{\calH}_{1}^{2})_{1}^{\nu}}\\
 & \quad=\frac{1}{\nu}\Big\|(Q-Q_{\nu/\wh{\nu}}-c_{M,\wh b}\wh b^{2}\Lmb Q)_{\wh{\nu}/\nu}\Big\|_{(\dot{\calH}_{1}^{2})_{1/\nu}}\aleq\frac{1}{\nu}\Big|\log\Big(\frac{\nu}{\wh{\nu}}\Big)\Big|^{2}\aleq\nu^{3-}.
\end{aligned}
\label{eq:ProfDiffStep4-3}
\end{equation}

For the second term of RHS\eqref{eq:ProfDiffStep4-2}, we write 
\[
\wh b^{2}[S_{1}]_{\wh{\nu}}-\frks b\nu[S_{1}]_{\nu}=\frac{\wh b^{2}-\frks b\wh{\nu}}{\wh{\nu}}[\wh{\nu}S_{1}]_{\wh{\nu}}+\frks b([\wh{\nu}S_{1}]_{\wh{\nu}}-[\nu S_{1}]_{\nu}).
\]
We estimate each term using \eqref{eq:ProfDiff-Param1}-\eqref{eq:ProfDiff-Param2}
and \eqref{eq:T1S1U1-asymptotics} (in particular improved spatial
decay of $\rd_{yy}S_{1}$ and $\Lmb_{2}S_{1}$): 
\begin{align*}
\Big\|\frac{\wh b^{2}-\frks b\wh{\nu}}{\wh{\nu}}[\wh{\nu}S_{1}]_{\wh{\nu}}\Big\|_{(\dot{\calH}_{1}^{2})_{1}^{\nu}} & =\Big|\frac{\wh b^{2}-\frks b\wh{\nu}}{\wh{\nu}}\Big|\|[S_{1}]_{\wh{\nu}/\nu}\|_{(\dot{\calH}_{1}^{2})_{1/\nu}}\aleq\frac{\nu}{|\log\nu|},\\
\|\frks b([\wh{\nu}S_{1}]_{\wh{\nu}}-[\nu S_{1}]_{\nu})\Big\|_{(\dot{\calH}_{1}^{2})_{1}^{\nu}} & =|\frks b|\Big\|\int_{1}^{\wh{\nu}/\nu}[a\Lmb_{2}S_{1}]_{a}\frac{da}{a}\Big\|_{(\dot{\calH}_{1}^{2})_{1/\nu}}\aleq\nu\Big|\log\Big(\frac{\nu}{\wh{\nu}}\Big)\Big|\aleq\nu^{3-}.
\end{align*}
Thus we have proved that
\begin{equation}
\Big\|\wh b^{2}[S_{1}]_{\wh{\nu}}-\frks b\nu[S_{1}]_{\nu}\Big\|_{(\dot{\calH}_{1}^{2})_{1}^{\nu}}\aleq\frac{\nu}{|\log\nu|}.\label{eq:ProfDiffStep4-4}
\end{equation}

Finally, for the third term of RHS\eqref{eq:ProfDiffStep4-2}, we
write 
\begin{align*}
 & \wh b^{2}c_{\wh b}[U_{1}]_{\wh{\nu}}+\frku b\nu[U_{1}]_{\nu}\\
 & \quad=\Big(\frac{\wh b^{2}}{\wh{\nu}}c_{\wh b}+\frku b\Big)[\wh{\nu}U_{1}]_{\wh{\nu}}+\frku b([\wh{\nu}U_{1}]_{\wh{\nu}}-[\nu U_{1}]_{\nu}).
\end{align*}
As in the proof of \eqref{eq:ProfDiffStep4-4}, we use \eqref{eq:ProfDiff-Param1}-\eqref{eq:ProfDiff-Param2}
and the logarithmically improved spatial decay for $y\geq1$
\[
|\rd_{yy}U_{1}|+\frac{1}{y^{2}}|\Lmb_{2}U_{1}|_{2}\aleq\frac{1}{y}
\]
to have (note that $\|\chf_{[1,\nu^{-1}]}\frac{1}{y}\|_{L^{2}}\sim|\log\nu|^{\frac{1}{2}}$)
\begin{align*}
\Big\|\Big(\frac{\wh b^{2}}{\wh{\nu}}c_{\wh b}+\frku b\Big)[\wh{\nu}U_{1}]_{\wh{\nu}}\Big\|_{(\dot{\calH}_{1}^{2})_{1}^{\nu}} & \aleq\frac{\nu}{|\log\nu|^{2}}\cdot|\log\nu|^{\frac{1}{2}}\aleq\frac{\nu}{|\log\nu|^{\frac{3}{2}}},\\
\Big\|\frku b([\wh{\nu}U_{1}]_{\wh{\nu}}-[\nu U_{1}]_{\nu})\Big\|_{(\dot{\calH}_{1}^{2})_{1}^{\nu}} & \aleq\frac{\nu}{|\log\nu|}\cdot\Big|\log\Big(\frac{\nu}{\wh{\nu}}\Big)\Big||\log\nu|^{\frac{1}{2}}\aleq\nu^{3-}.
\end{align*}
Thus we have proved that 
\begin{equation}
\Big\|\wh b^{2}c_{\wh b}[U_{1}]_{\wh{\nu}}+\frku b[U_{1}]_{\nu}\Big\|_{(\dot{\calH}_{1}^{2})_{1}^{\nu}}\aleq\frac{\nu}{|\log\nu|^{\frac{3}{2}}}.\label{eq:ProfDiffStep4-5}
\end{equation}
Substituting \eqref{eq:ProfDiffStep4-3}-\eqref{eq:ProfDiffStep4-5}
into \eqref{eq:ProfDiffStep4-2} completes the proof of \eqref{eq:ProfDiff}
for $P$.

\textbf{Step 5.} Difference of $\dot{P}^{RR}$ and $\dot{P}$.

In this step, we show \eqref{eq:ProfDiff} for $\dot{P}$. As in the
previous step, we will only consider the region $\rho\leq1$ and the
cutoff $\chi_{B_{1}}$ in the definition of $P$ plays no role. For
$\rho\leq1$, we note by the definitions \eqref{eq:def-Pdot-RR} and
\eqref{eq:def-mod-prof} that 
\begin{equation}
\begin{aligned}\frac{1}{\wh{\nu}}\dot{P}^{\RR}(\wh b;\frac{\rho}{\wh{\nu}})-\dot{P}(\nu,b;\rho) & =\frac{\wh b}{\wh{\nu}}\Big(\Lmb Q_{\wh{\nu}}+\wh b^{2}\Lmb(H^{-1}g-c_{M,\wh b}\Lmb Q)_{\wh{\nu}}\Big)\\
 & \peq-\frac{b}{2}\Big((\lmb_{0}+\Lmb)\varphi_{0}-(\lmb_{1}+\Lmb)\varphi_{1}+\frac{1}{\nu}\Lmb Q_{\nu}\Big),
\end{aligned}
\label{eq:ProfDiffStep5-1}
\end{equation}
where we abbreviated $\lmb_{j}=\lmb_{j}(\nu)$. By \eqref{eq:ProfDiffStep1-1}
and \eqref{eq:ProfDiffStep2-1}, we have 
\begin{align*}
\Big\|\frac{\wh b^{3}}{\wh{\nu}}\Lmb(H^{-1}g)_{\wh{\nu}}\Big\|_{(\dot{H}_{1}^{1})_{1}} & =\Big|\frac{\wh b^{3}}{\wh{\nu}}\Big|\|\Lmb H^{-1}g\|_{(\dot{H}_{1}^{1})_{1/\wh{\nu}}}\aleq\frac{\nu}{|\log\nu|},\\
\Big\|\frac{\wh b^{3}}{\wh{\nu}}c_{M,\wh b}\Lmb Q_{\wh{\nu}}\Big\|_{(\dot{H}_{1}^{1})_{1}} & =\Big|\frac{\wh b^{3}}{\wh{\nu}}c_{M,\wh b}\Big|\|\Lmb Q\|_{(\dot{H}_{1}^{1})_{1/\wh{\nu}}}\aleq\nu^{4-}.
\end{align*}
By \eqref{eq:EigenFuncEst-2}, we have 
\[
\Big\|\varphi_{j}-\frac{1}{\nu}\Lmb Q_{\nu}\Big\|_{(\dot{H}_{1}^{1})_{1}}\aleq\frac{1}{|\log\nu|}.
\]
Substituting the above two displays into \eqref{eq:ProfDiffStep5-1}
yields 
\begin{equation}
\begin{aligned} & \frac{1}{\wh{\nu}}\dot{P}^{\RR}(\wh b;\frac{\rho}{\wh{\nu}})-\dot{P}(\nu,b;\rho)\\
 & \quad=\Big(\frac{\wh b}{\wh{\nu}}\Lmb Q_{\wh{\nu}}-\frac{b}{2}(\lmb_{0}-\lmb_{1}+1)\frac{1}{\nu}\Lmb Q_{\nu}\Big)+O_{(\dot{H}_{1}^{1})_{1}}\Big(\frac{\nu}{|\log\nu|}\Big).
\end{aligned}
\label{eq:ProfDiffStep5-2}
\end{equation}
Next, we write 
\begin{align*}
 & \frac{\wh b}{\wh{\nu}}\Lmb Q_{\wh{\nu}}-\frac{b}{2}(\lmb_{0}-\lmb_{1}+1)\frac{1}{\nu}\Lmb Q_{\nu}\\
 & \quad=\wh b\Big(\frac{1}{\wh{\nu}}\Lmb Q_{\wh{\nu}}-\frac{1}{\nu}\Lmb Q_{\nu}\Big)+\Big(\wh b-\frac{b}{2}(\lmb_{0}-\lmb_{1}+1)\Big)\frac{1}{\nu}\Lmb Q_{\nu}.
\end{align*}
Notice that the second term \emph{vanishes} thanks to the relation
\eqref{eq:Rel-nu-and-nuhat}. The first term is small due to \eqref{eq:ProfDiff-Param2}:
\[
\Big\|\wh b\Big(\frac{1}{\wh{\nu}}\Lmb Q_{\wh{\nu}}-\frac{1}{\nu}\Lmb Q_{\nu}\Big)\Big\|_{(\dot{H}_{1}^{1})_{1}}\aleq\frac{\wh b}{\nu}\int_{1}^{\wh{\nu}/\nu}\Big\|\frac{1}{a}\Lmb_{0}\Lmb Q_{a}\Big\|_{(\dot{H}_{1}^{1})_{1/\nu}}\frac{da}{a}\aleq\Big|\log\Big(\frac{\wh{\nu}}{\td{\nu}}\Big)\Big|\aleq\nu^{2-}.
\]
Substituting this into \eqref{eq:ProfDiffStep5-2} completes the proof
of \eqref{eq:ProfDiff} for $\dot{P}$.
\end{proof}
To show the $\ell_{j}(\bm{\eps})$ bounds \eqref{eq:eps-bound-3}-\eqref{eq:eps-bound-4}
in Proposition~\ref{prop:decomposition}, we also need the following
mapping properties of $\ell_{j}$.
\begin{lem}[Mapping properties of $\ell_{j}$]
\label{lem:ell_j-mapping}We have 
\begin{align}
|\ell_{j}(\bm{f})| & \aleq|\log\nu|^{\frac{1}{2}}\|f\|_{(\dot{\calH}_{1}^{2})_{1}^{\nu}}+\|\dot{f}\|_{(\dot{H}_{1}^{1})_{1}},\label{eq:ell_j-mapping}\\
|[\nu\rd_{\nu}\ell_{j}](\bm{f})| & \aleq\frac{1}{|\log\nu|^{2}}(|\log\nu|^{\frac{1}{2}}\|f\|_{(\dot{\calH}_{1}^{2})_{1}^{\nu}}+\|\dot{f}\|_{(\dot{H}_{1}^{1})_{1}}).\label{eq:rd_nu-ell_j-mapping}
\end{align}
\end{lem}

\begin{proof}
In the proof, we use the weighted $L^{\infty}$-bounds in Lemma~\ref{lem:Local-Linfty-est},
which by scaling reads 
\begin{align}
\chf_{(0,1]}|f|_{1} & \aleq\|f\|_{(\dot{\calH}_{1}^{2})_{1}^{\nu}}\cdot\rho\langle\log(\frac{\rho}{\nu})\rangle^{\frac{1}{2}},\label{eq:mapping-01}\\
\chf_{(0,1]}|\dot{f}| & \aleq\|\dot{f}\|_{(\dot{H}_{1}^{1})_{1}}.\label{eq:mapping-02}
\end{align}
One important point here is that we only have power $\frac{1}{2}$
in the log factor of RHS\eqref{eq:mapping-01}. (Compare this with
the factor $\langle\log y\rangle^{-1}$ in the $\dot{\calH}_{1}^{2}$-norm.)

We first show \eqref{eq:ell_j-mapping}. Recall that 
\begin{align*}
\ell_{j}(\bm{f}) & =\langle(\lmb_{j}+\Lmb)f+\dot{f},g_{j}\varphi_{j}\rangle,\tag{\ref{eq:def-ell_j}}\\
|g_{j}\varphi_{j}| & \aleq\chf_{(0,\frac{1}{2}]}\rho^{-1}+\chf_{[\frac{1}{2},1]}\cdot(1-\rho)^{-\frac{3}{4}}.\tag{\ref{eq:ExactComp-6-2}}
\end{align*}
Thus we have (by \eqref{eq:mapping-01})
\begin{align*}
\big|\langle(\lmb_{j}+\Lmb)f,g_{j}\varphi_{j}\rangle\big| & \aleq\big\||f|_{1}\cdot|g_{j}\varphi_{j}|\big\|_{L^{1}}\\
 & \aleq\Big(\int_{0}^{\frac{1}{2}}\rho\langle\log(\frac{\rho}{\nu})\rangle^{\frac{1}{2}}d\rho+\int_{\frac{1}{2}}^{1}|\log\nu|^{\frac{1}{2}}(1-\rho)^{-\frac{3}{4}}d\rho\Big)\|f\|_{(\dot{\calH}_{1}^{2})_{1}^{\nu}}\\
 & \aleq|\log\nu|^{\frac{1}{2}}\|f\|_{(\dot{\calH}_{1}^{2})_{1}^{\nu}}
\end{align*}
and (by \eqref{eq:mapping-02})
\[
\big|\langle\dot{f},g_{j}\varphi_{j}\rangle\big|\aleq\Big(\int_{0}^{\frac{1}{2}}d\rho+\int_{\frac{1}{2}}^{1}(1-\rho)^{-\frac{3}{4}}d\rho\Big)\|\dot{f}\|_{(\dot{H}_{1}^{1})_{1}}\aleq\|\dot{f}\|_{(\dot{H}_{1}^{1})_{1}}.
\]
This completes the proof of \eqref{eq:ell_j-mapping}.

We turn to show \eqref{eq:rd_nu-ell_j-mapping}. Note that 
\begin{align}
[\nu\rd_{\nu}\ell_{j}](\bm{f}) & =(\nu\rd_{\nu}\lmb_{j})\langle f,g_{j}\varphi_{j}\rangle+\langle(\lmb_{j}+\Lmb_{0})f+\dot{f},(\nu\rd_{\nu}g_{j})\varphi_{j}\rangle\label{eq:mapping-03}\\
 & \peq+\langle(\lmb_{j}+\Lmb_{0})f+\dot{f},g_{j}(\nu\rd_{\nu}\varphi_{j})\rangle.\nonumber 
\end{align}
We also recall that 
\[
|(\nu\rd_{\nu}\lmb_{j})g_{j}|+|\nu\rd_{\nu}g_{j}|\aleq\frac{1}{|\log\nu|^{2}}\chf_{(0,1]}\cdot(1-\rho)^{-\frac{3}{4}},\tag{\ref{eq:ExactComp-9}}
\]
which is just $\frac{1}{|\log\nu|^{2}}$ times \eqref{eq:ExactComp-6-2}.
Thus the first line of RHS\eqref{eq:mapping-03} is estimated by 
\[
\frac{1}{|\log\nu|^{2}}(|\log\nu|^{\frac{1}{2}}\|f\|_{(\dot{\calH}_{1}^{2})_{1}^{\nu}}+\|\dot{f}\|_{(\dot{H}_{1}^{1})_{1}})
\]
by the proof of \eqref{eq:ell_j-mapping}. Using \eqref{eq:EigenFuncEst-3},
the second line of RHS\eqref{eq:mapping-03} can be estimated by 
\begin{align*}
 & \big|\langle(\lmb_{j}+\Lmb_{0})f,g_{j}(\nu\rd_{\nu}\varphi_{j})\rangle\big|\\
 & \aleq\bigg\{\int_{0}^{\nu}\frac{\rho^{3}\langle\log(\frac{\rho}{\nu})\rangle^{\frac{1}{2}}}{\nu^{2}}d\rho+\int_{\nu}^{\frac{1}{2}}\Big(\frac{\nu^{2}}{\rho}+\frac{\rho^{3}\langle\log\rho\rangle}{|\log\nu|^{2}}\Big)\langle\log(\frac{\rho}{\nu})\rangle^{\frac{1}{2}}d\rho\\
 & \peq\qquad\qquad\qquad\qquad+\int_{\frac{1}{2}}^{1}\frac{(1-\rho)^{-\frac{3}{4}}}{|\log\nu|^{\frac{3}{2}}}d\rho\bigg\}\|f\|_{(\dot{\calH}_{1}^{2})_{1}^{\nu}}\aleq\frac{1}{|\log\nu|^{\frac{3}{2}}}\|f\|_{(\dot{\calH}_{1}^{2})_{1}^{\nu}}
\end{align*}
and 
\begin{align*}
 & \big|\langle\dot{f},g_{j}(\nu\rd_{\nu}\varphi_{j})\rangle\big|\\
 & \aleq\bigg\{\int_{0}^{\nu}\frac{\rho^{2}}{\nu^{2}}d\rho+\int_{\nu}^{\frac{1}{2}}\Big(\frac{\nu^{2}}{\rho^{2}}+\frac{\rho^{2}\langle\log\rho\rangle}{|\log\nu|^{2}}\Big)d\rho+\int_{\frac{1}{2}}^{1}\frac{(1-\rho)^{-\frac{3}{4}}}{|\log\nu|^{2}}d\rho\bigg\}\|\dot{f}\|_{(\dot{H}_{1}^{1})_{1}}\\
 & \aleq\frac{1}{|\log\nu|^{2}}\|\dot{f}\|_{(\dot{H}_{1}^{1})_{1}}.
\end{align*}
This completes the proof of \eqref{eq:rd_nu-ell_j-mapping}.
\end{proof}
\begin{proof}[Proof of Proposition~\ref{prop:decomposition}]
For all large $\tau$, define $\nu(\tau)$ and $b(\tau)$ using
Lemma~\ref{lem:ProfDiff}. The estimates \eqref{eq:nu-b-bound} and
\eqref{eq:rough-asymp-nu-b} follow from combining the corresponding
statements \eqref{eq:nu-b-RR-bound} and \eqref{eq:nu-RR-asymp} for
$\wh{\nu}(\tau)$ and $\wh b(\tau)$ with \eqref{eq:ProfDiff-Param1}-\eqref{eq:ProfDiff-Param3}.
The bound \eqref{eq:eps-bound-1} follows from \eqref{eq:eRR-est3}
and \eqref{eq:ProfDiff}. Now \eqref{eq:eps-bound-1} implies \eqref{eq:eps-bound-2},
\eqref{eq:eps-bound-3}, and \eqref{eq:eps-bound-4} thanks to \eqref{eq:mapping-01}-\eqref{eq:mapping-02}
and Lemma~\ref{lem:ell_j-mapping}. This completes the proof of Proposition~\ref{prop:decomposition}.
\end{proof}

\subsection{\label{subsec:Modulation-estimates}Modulation estimates}

So far, we started with a blow-up solution $\bm{u}=(u,\dot{u})$ described
in Proposition~\ref{prop:RR-blow-up-original}, applied the self-similar
transform to $\bm{u}$ to get $\bm{v}=(v,\dot{v})$ as in Corollary~\ref{cor:RR-blow-up-self-similar},
and finally decomposed $\bm{v}$ according to Proposition~\ref{prop:decomposition}.
The main goal (Proposition~\ref{prop:ModulationEstimates}) of this
subsection is to obtain refined modulation equations for the parameters
$\nu$ and $b$ by testing the evolution equation of $\bm{\eps}$
against $\ell_{j}$, where we recall 
\begin{align*}
\ell_{j}(\bm{\eps}) & =\langle(\lmb_{j}+\Lmb_{0})\eps+\dot{\eps},g_{j}\varphi_{j}\rangle,\\
g_{j}(\nu;\rho) & =\chf_{(0,1]}(\rho)\cdot(1-\rho^{2})^{\lmb_{j}-\frac{1}{2}}.
\end{align*}

It is instructive to see how testing against $\ell_{j}$ improves
the modulation estimates in \cite{RaphaelRodnianski2012Publ.Math.}.
The modulation estimates in \cite[Section 7]{RaphaelRodnianski2012Publ.Math.}
are obtained by testing the equation for $\eps^{\RR}$ against $\Lmb P_{B_{0}}^{\RR}$,
where $P_{B_{0}}^{\RR}$ is defined similarly with \eqref{eq:def-PRR}
but the cutoff $\chi_{B_{1}}$ is replaced by $\chi_{B_{0}}$. This
motivates some function $G(b)\approx4b|\log b|$ with the property
\begin{equation}
\lmb\{G(b)+O(b)\}_{t}=-2b^{2}+O\Big(\frac{b^{2}}{|\log b|}\Big).\label{eq:RR-ref-mod-eqn}
\end{equation}
Note that \eqref{eq:RR-ref-mod-eqn} is not sufficient to deduce the
sharp blow-up rate \eqref{eq:sharp-blow-up-rate} because \eqref{eq:sharp-blow-up-rate}
requires more refined information on both the $O(b)$-term of LHS\eqref{eq:RR-ref-mod-eqn}
and the $O(\frac{b^{2}}{|\log b|})$-order term of RHS\eqref{eq:RR-ref-mod-eqn}.

The problem is that there are several error terms in \eqref{eq:RR-ref-mod-eqn}
which we do not know whether they can be improved or not. One type
of such terms arises from the cutoff errors; even in the definition
of $G(b)$ there is a cutoff error of size $O(b)$. Here the spatial
decays of the involved profiles are critical in the sense that changing
the cutoff radius does not improve the error. Another type of such
error terms in RHS\eqref{eq:RR-ref-mod-eqn} would be the linear error
term
\begin{equation}
\langle H_{B_{1}}\eps^{\RR},\Lmb P_{B_{0}}^{\RR}\rangle,\label{eq:RR-ref-mod-eqn-2}
\end{equation}
where $H_{B_{1}}\eps=-\rd_{yy}\eps-\frac{1}{y}\rd_{y}\eps+b^{2}\Lmb_{0}\Lmb\eps+\frac{1}{y^{2}}\cos(P_{B_{0}}^{\RR})\eps$.
From the construction of $\Lmb P_{B_{0}}^{\RR}$, \eqref{eq:RR-ref-mod-eqn-2}
essentially becomes $\langle\eps^{\RR},\Lmb\Psi_{B_{0}}+2\Psi_{B_{0}}\rangle$,
where $\Psi_{B_{0}}$ is the equation error for the modified profile
$\Lmb P_{B_{0}}^{\RR}$. This error is of critical size $O(\frac{b^{2}}{|\log b|})$,
and improving this error requires either the improvements on the size
of $\eps^{\RR}$ or the modified profile ansatz, which seem to be
difficult.

Our key observation is that, without changing the modified profiles
or improving the bounds for $\eps$, one can still derive refined
modulation equations \emph{if one changes the test function}. We use
$\ell_{j}$ in the self-similar coordinates, instead of $\langle\cdot,\Lmb P_{B_{0}}^{\RR}\rangle$
in the inner coordinates for $\eps^{\RR}$. The linear functional
$\ell_{j}$ is defined explicitly \eqref{eq:def-ell_j} and the sharp
cutoff $\chf_{(0,1]}$ in its definition does not cause any problem
in the computations, thanks to the choice of the weight $g_{j}$ (c.f.
the proof of \eqref{eq:Invariance}). Moreover, due to the invariance
\eqref{eq:Invariance}, the linear error term (the analogue of \eqref{eq:RR-ref-mod-eqn-2}
in this case) has additional structure, which is simply 
\[
\ell_{j}(\mathbf{M}_{\nu}\bm{\eps})=\lmb_{j}\ell_{j}(\bm{\eps}).
\]
In particular, this error can be improved by a logarithmic factor
when $j=1$, thanks to $|\lmb_{1}|\lesssim\frac{1}{|\log\nu|}$.

We now derive the evolution equation for $\bm{\eps}$. Recall \eqref{eq:v-vec-eqn}:
\[
\rd_{\tau}\bm{v}=\rd_{\tau}\begin{bmatrix}v\\
\dot{v}
\end{bmatrix}=\begin{bmatrix}-\Lmb v+\dot{v}\\
\rd_{\rho\rho}v+\frac{1}{\rho}\rd_{\rho}v-\frac{\sin(2v)}{2\rho^{2}}-\Lmb_{0}\dot{v}
\end{bmatrix}.
\]
Substituting the decomposition \eqref{eq:decomposition} of $\bm{v}$
into the above, we have 
\begin{equation}
\rd_{\tau}(\bm{P}+\bm{\eps})+\bm{\Lmb Q}_{\nu}=\mathbf{M}_{\nu}\td{\bm{v}}-\begin{bmatrix}0\\
R_{\mathrm{NL}}(\td v)
\end{bmatrix},\label{eq:evol-eqn-for-mod-est}
\end{equation}
where 
\begin{align}
\td{\bm{v}} & \coloneqq\bm{v}-\bm{Q}_{\nu}=(\bm{P}-\bm{Q}_{\nu})+\bm{\eps},\label{eq:def-v-tilde}\\
R_{\mathrm{NL}}(\td v) & \coloneqq\frac{1}{2\rho^{2}}\{\sin(2v)-\sin(2Q_{\nu})-2\cos(2Q_{\nu})\td v\}.\label{eq:def-R-NL}
\end{align}
We remark that the nonlinear error term $R_{\mathrm{NL}}(\td v)$
will not affect the modulation equations.
\begin{prop}[Modulation estimates]
\label{prop:ModulationEstimates}We have a rough estimate 
\begin{equation}
|\rd_{\tau}[\ell_{j}(\bm{\eps})]|\aleq\nu,\qquad\forall j\in\{0,1\}.\label{eq:eps-bound-5}
\end{equation}
Moreover, we have the following refined modulation estimates 
\begin{equation}
\bigg|\frac{\nu_{\tau}}{\nu}+\Big(\frac{b}{\nu}-1\Big)+\frac{\nu_{1}}{|\log\nu|}+\sum_{j=0}^{1}L_{j}^{(\nu)}(\rd_{\tau}-\lmb_{j})[\ell_{j}(\bm{\eps})]\bigg|\aleq\frac{1}{|\log\nu|^{2}},\label{eq:RefMod-nu}
\end{equation}
and 
\begin{equation}
\begin{aligned}\bigg|\frac{b_{\tau}}{b}+\Big(\frac{b}{\nu}-1\Big)\frac{b_{-1}}{|\log\nu|}+\frac{b_{1}}{|\log\nu|}+\frac{b_{2}}{|\log\nu|^{2}}\qquad\qquad\\
+\sum_{j=0}^{1}L_{j}^{(b)}(\rd_{\tau}-\lmb_{j})[\ell_{j}(\bm{\eps})]\bigg| & \aleq\frac{1}{|\log\nu|^{3}},
\end{aligned}
\label{eq:RefMod-b}
\end{equation}
where 
\begin{equation}
\nu_{1}=\frac{1}{3},\quad b_{-1}=\frac{1}{2},\quad b_{1}=\frac{1}{2},\quad b_{2}=\frac{5}{12}-\frac{\log2}{2},\label{eq:values-nu1-b-12}
\end{equation}
and $L_{j}^{(\nu)}$ and $L_{j}^{(b)}$ are functions satisfying 
\begin{align}
|L_{j}^{(\nu)}|+|L_{j}^{(b)}| & \aleq\frac{1}{\nu|\log\nu|},\label{eq:Lj-1}\\
|\rd_{\tau}L_{j}^{(\nu)}|+|\rd_{\tau}L_{j}^{(b)}| & \aleq\frac{1}{\nu|\log\nu|^{2}},\label{eq:Lj-2}\\
|L_{0}^{(b)}|+|L_{1}^{(\nu)}-L_{1}^{(b)}| & \aleq\frac{1}{\nu|\log\nu|^{2}}.\label{eq:Lj-3}
\end{align}
\end{prop}

\begin{rem}
\label{rem:mod-est}In order to prove the universality and to determine
the precise constant of the blow-up rate (see the proof of Theorem~\ref{thm:MainThm}
in Section~\ref{subsec:Proof-of-Theorem}), it is necessary to know
precise $\frac{1}{|\log\nu|^{2}}$-order terms (i.e., the values of
$b_{-1}$ and $b_{2}$) of the modulation equation \eqref{eq:RefMod-b}
for $b_{\tau}$. In contrast, the $\nu_{\tau}$-equation \eqref{eq:RefMod-nu}
does not need to be as precise as the $b_{\tau}$-equation. This fact
allows us to simplify some of the computations of $\lmb_{j},\frkb_{j},\frkc_{j}$
(see Lemma~\ref{lem:conv-comp-for-mod-eqn}); only $\lmb_{1}$ and
$\frkb_{1}$ require precise $\frac{1}{|\log\nu|}$-order terms (see
\eqref{eq:RefModStep5-5}).

In principle, it is possible to compute precise $\frac{1}{|\log\nu|^{2}}$-order
terms of \eqref{eq:RefMod-nu}. However, one needs to use the delicate
information \eqref{eq:SharpEigftDecomp} on the eigenfunctions, and
the computations of definite integrals are much more involved.
\end{rem}

\begin{rem}
On the other hand, because of the necessary accuracies for the modulation
equations, we cannot simply regard $(\rd_{\tau}-\lmb_{j})\ell_{j}(\bm{\eps})$-type
terms in \eqref{eq:RefMod-nu}-\eqref{eq:RefMod-b} as perturbative
terms (mostly due to the rough estimate \eqref{eq:eps-bound-5}).
When integrating the modulation equations in Section~\ref{subsec:Proof-of-Theorem},
we will always need to incorporate these $\bm{\eps}$-dependent terms
as \emph{corrections} to our modulation parameters $\nu$ and $b$,
and also utilize the structure \eqref{eq:Lj-1}-\eqref{eq:Lj-3} of
$L_{j}$'s.
\end{rem}

\begin{proof}[Proof of Proposition~\ref{prop:ModulationEstimates}]
Recall $\frkb_{j}$ and $\frkc_{j}$ defined from Lemma~\ref{lem:conv-comp-for-mod-eqn}.
These quantities will be important to compute the universal constants
shown in \eqref{eq:values-nu1-b-12}.

\textbf{Step 1.} Algebraic computations.

In this step, we claim the following algebraic identity: 
\begin{equation}
\begin{aligned}\nu_{\tau}\ell_{j}(\rd_{\nu}\bm{P})+b_{\tau}\ell_{j}(\rd_{b}\bm{P})+(\rd_{\tau}-\lmb_{j})[\ell_{j}(\bm{\eps})] & +\nu\frkb_{j}-b\lmb_{j}\frkc_{j}\\
 & =\nu_{\tau}[\rd_{\nu}\ell_{j}](\bm{\eps})-\langle R_{\mathrm{NL}}(\td v),g_{j}\varphi_{j}\rangle.
\end{aligned}
\label{eq:mod-eqn-step1-claim}
\end{equation}
To see this, we start from taking $\ell_{j}$ to the equation \eqref{eq:evol-eqn-for-mod-est}:
\begin{equation}
\ell_{j}(\rd_{\tau}\bm{P})+\ell_{j}(\rd_{\tau}\bm{\eps})+\ell_{j}(\bm{\Lmb Q}_{\nu})=\ell_{j}({\bf M}_{\nu}\td{\bm{v}})-\langle R_{\mathrm{NL}}(\td v),g_{j}\varphi_{j}\rangle.\label{eq:mod-eqn-step1-2}
\end{equation}
For the first term of LHS\eqref{eq:mod-eqn-step1-2}, we write 
\[
\ell_{j}(\rd_{\tau}\bm{P})=\nu_{\tau}\ell_{j}(\rd_{\nu}\bm{P})+b_{\tau}\ell_{j}(\rd_{b}\bm{P}).
\]
For the second term of LHS\eqref{eq:mod-eqn-step1-2}, we write 
\[
\ell_{j}(\rd_{\tau}\bm{\eps})=\rd_{\tau}[\ell_{j}(\bm{\eps})]-\nu_{\tau}[\rd_{\nu}\ell_{j}](\bm{\eps}).
\]
For the third term of LHS\eqref{eq:mod-eqn-step1-2}, by \eqref{eq:def-frkb}
we have 
\[
\ell_{j}(\bm{\Lmb Q}_{\nu})=\nu\frkb_{j}
\]
For the first term of RHS\eqref{eq:mod-eqn-step1-2}, we use the invariance
\eqref{eq:Invariance} of ${\bf M}_{\nu}$ and the definitions of
$\td{\bm{v}}$, $\bm{P}$, and $\frkc_{j}$ (see \eqref{eq:def-v-tilde},
\eqref{eq:def-mod-prof}, and \eqref{eq:def-frkc}) to write 
\[
\ell_{j}({\bf M}_{\nu}\td{\bm{v}})=\lmb_{j}\ell_{j}(\td{\bm{v}})=\lmb_{j}\ell_{j}(\bm{P}-\bm{Q}_{\nu})+\lmb_{j}\ell_{j}(\bm{\eps})=b\lmb_{j}\frkc_{j}+\lmb_{j}\ell_{j}(\bm{\eps}).
\]
Substituting the previous displays into \eqref{eq:mod-eqn-step1-2}
completes the proof of \eqref{eq:mod-eqn-step1-claim}.

\textbf{Step 2. }Computation of $\ell_{j}(\rd_{\nu}\bm{P})$ and $\ell_{j}(\rd_{b}\bm{P})$.

In this step, we claim 
\begin{align}
\ell_{j}(\rd_{\nu}\bm{P}) & =-(\frkb_{j}+2)+O\Big(\frac{1}{|\log\nu|}\Big),\label{eq:ell-rd_nu-P}\\
\ell_{j}(\rd_{b}\bm{P}) & =\frkc_{j}.\label{eq:ell-rd_b-P}
\end{align}
Assuming these claims and using \eqref{eq:nu-b-bound}, \eqref{eq:mod-eqn-step1-claim}
becomes 
\begin{equation}
\begin{aligned}-\nu_{\tau}(\frkb_{j}+2)+b_{\tau}\frkc_{j} & +(\rd_{\tau}-\lmb_{j})[\ell_{j}(\bm{\eps})]+\nu\frkb_{j}-b\lmb_{j}\frkc_{j}\\
 & =\nu_{\tau}[\rd_{\nu}\ell_{j}](\bm{\eps})-\langle R_{\mathrm{NL}}(\td v),g_{j}\varphi_{j}\rangle+O\Big(\frac{\nu}{|\log\nu|^{2}}\Big).
\end{aligned}
\label{eq:mod-eqn-step2-1}
\end{equation}
From now on, we show \eqref{eq:ell-rd_nu-P} and \eqref{eq:ell-rd_b-P}.

\emph{Proof of \eqref{eq:ell-rd_nu-P}.} Observe that 
\[
\ell_{j}(\rd_{\nu}\bm{P})=-\ell_{j}\Big(\frac{1}{\nu}\bm{\Lmb Q}_{\nu}\Big)+\frac{b}{2\nu}\Big\{\ell_{j}(\nu\rd_{\nu}(\bm{\varphi}_{0}-\bm{\varphi}_{1}))-\Big\langle\frac{1}{\nu}\Lmb_{0}\Lmb Q_{\nu},g_{j}\varphi_{j}\Big\rangle\Big\}.
\]
Applying \eqref{eq:def-frkb}, \eqref{eq:OneMoreComp}, and $|\frac{b}{\nu}-1|\aleq\frac{1}{|\log\nu|}$
to the above, we get \eqref{eq:ell-rd_nu-P}.

\emph{Proof of \eqref{eq:ell-rd_b-P}.} This directly follows from
the definition of $\bm{P}$ and $\frkc_{j}$ (see \eqref{eq:def-mod-prof}
and \eqref{eq:def-frkc}).

\textbf{Step 3.} Treatment of the error terms.

In this step, we claim that RHS\eqref{eq:mod-eqn-step2-1} is treated
as an error: 
\begin{align}
\nu_{\tau}|[\rd_{\nu}\ell_{j}](\bm{\eps})| & \aleq\frac{\nu}{|\log\nu|^{\frac{7}{2}}},\label{eq:RefModStep3-1}\\
|\langle R_{\mathrm{NL}}(\td v),g_{j}\varphi_{j}\rangle| & \aleq\nu^{3}|\log\nu|.\label{eq:RefModStep3-2}
\end{align}
Assuming these claims, \eqref{eq:mod-eqn-step2-1} becomes 
\begin{equation}
\big|-\nu_{\tau}(\frkb_{j}+2)+b_{\tau}\frkc_{j}+(\rd_{\tau}-\lmb_{j})[\ell_{j}(\bm{\eps})]+\nu\frkb_{j}-b\lmb_{j}\frkc_{j}\big|\aleq\frac{\nu}{|\log\nu|^{2}}.\label{eq:RefModStep3-3}
\end{equation}
Dividing the both hand sides by $-\nu$, we arrive at 
\begin{equation}
\bigg|\frac{\nu_{\tau}}{\nu}(\frkb_{j}+2)-\frac{b_{\tau}}{b}\Big(\frac{b}{\nu}\frkc_{j}\Big)+\frac{b}{\nu}\lmb_{j}\frkc_{j}-\frkb_{j}-\frac{1}{\nu}(\rd_{\tau}-\lmb_{j})[\ell_{j}(\bm{\eps})]\bigg|\aleq\frac{1}{|\log\nu|^{2}}.\label{eq:RefModStep3-4}
\end{equation}
From now on, we show \eqref{eq:RefModStep3-1} and \eqref{eq:RefModStep3-2}.

\emph{Proof of \eqref{eq:RefModStep3-1}.} This is indeed immediate
from \eqref{eq:nu-b-bound} and \eqref{eq:eps-bound-4}.

\emph{Proof of \eqref{eq:RefModStep3-2}.} Recall the definition \eqref{eq:def-R-NL}
of $R_{\mathrm{NL}}(\td v)$. We apply the trigonometric identity
to have 
\begin{align*}
R_{\mathrm{NL}}(\td v) & =\frac{1}{2\rho^{2}}\{\sin(2Q_{\nu}+2\td v)-\sin(2Q_{\nu})-2\cos(2Q_{\nu})\td v\}\\
 & =\frac{1}{2\rho^{2}}\{\sin(2Q_{\nu})(\cos(2\td v)-1)+\cos(2Q_{\nu})(\sin(2\td v)-2\td v)\}.
\end{align*}
Applying $\sin(2Q_{\nu})=2\cos(Q_{\nu})\Lmb Q_{\nu}$, we have 
\begin{align*}
|R_{\mathrm{NL}}(\td v)| & \aleq\frac{1}{\rho^{2}}(|\Lmb Q_{\nu}||\cos(2\td v)-1|+|\sin(2\td v)-2\td v|)\\
 & \aleq\frac{1}{\rho^{2}}(|\Lmb Q_{\nu}||\td v|^{2}+|\td v|^{3}).
\end{align*}
Note that $|\td v|\aleq b|\varphi_{0}-\varphi_{1}|+|\eps|$ by the
definition \eqref{eq:def-v-tilde} of $\td v$. Using $b\sim\nu$,
\eqref{eq:EigenFuncEst-2}, and \eqref{eq:eps-bound-2}, we have 
\begin{align*}
\chf_{(0,1]}|\td v| & \aleq\chf_{(0,1]}\{b|\varphi_{0}-\varphi_{1}|+|\eps|\}\\
 & \aleq\nu\cdot\Big(\chf_{(0,\nu]}\rho\Big(\frac{\rho}{\nu}\Big)^{2}+\chf_{[\nu,1]}\frac{\rho\langle\log\rho\rangle}{|\log\nu|}+\chf_{(0,1]}\frac{\rho\langle\log(\frac{\rho}{\nu})\rangle^{\frac{1}{2}}}{|\log\nu|}\Big)\aleq\chf_{(0,1]}\nu\rho.
\end{align*}
Using $|\Lmb Q_{\nu}|\aleq\frac{\nu}{\rho}$, we have a rough pointwise
bound of $R_{\mathrm{NL}}(\td v)$: 
\[
\chf_{(0,1]}|R_{\mathrm{NL}}(\td v)|\aleq\chf_{(0,1]}\frac{\nu^{3}}{\rho}.
\]
Combining this with the pointwise bound 
\[
|g_{j}\varphi_{j}|\aleq g_{j}\cdot\frac{1}{\nu}\Lmb Q_{\nu}\aleq\chf_{(0,\nu]}\frac{\rho}{\nu^{2}}+\chf_{[\nu,1]}\frac{1}{\rho}(1-\rho)^{-\frac{3}{4}}
\]
gives 
\[
|\langle R_{\mathrm{NL}}(\td v),g_{j}\varphi_{j}\rangle|\aleq\int_{0}^{1}\Big(\chf_{(0,\nu]}\nu+\chf_{[\nu,1]}\frac{\nu^{3}}{\rho^{2}}(1-\rho)^{-\frac{3}{4}}\Big)\rho d\rho\aleq\nu^{3}|\log\nu|,
\]
completing the proof of \eqref{eq:RefModStep3-2}.

\textbf{Step 4.} Proof of the rough estimate \eqref{eq:eps-bound-5}.

The rough estimate \eqref{eq:eps-bound-5} directly follows from estimating
each term of \eqref{eq:RefModStep3-3} except $\rd_{\tau}[\ell_{j}(\bm{\eps})]$
using \eqref{eq:nu-b-bound}, \eqref{eq:frkb-asymp}-\eqref{eq:frkc-asymp},
and \eqref{eq:eps-bound-3}. Note that although each of $\nu\frkb_{0}$
and $b\lmb_{0}\frkc_{0}$ is of size $O(\nu|\log\nu|)$, $\nu\frkb_{0}-b\lmb_{0}\frkc_{0}$
is of size $O(\nu)$.

\textbf{Step 5.} Proofs of the refined modulation estimates \eqref{eq:RefMod-nu}
and \eqref{eq:RefMod-b}.

In this step, we show \eqref{eq:RefMod-nu} and \eqref{eq:RefMod-b}.
Our starting point is \eqref{eq:RefModStep3-4}.

We first show \eqref{eq:RefMod-nu}. We subtract \eqref{eq:RefModStep3-4}
for $j=1$ from \eqref{eq:RefModStep3-4} for $j=0$ and then divide
it by $\frkb_{0}-\frkb_{1}$ to have 
\begin{equation}
\begin{aligned}\bigg|\frac{\nu_{\tau}}{\nu}-\frac{b_{\tau}}{b}\Big(\frac{\frac{b}{\nu}\cdot(\frkc_{0}-\frkc_{1})}{\frkb_{0}-\frkb_{1}}\Big)+\frac{b}{\nu}\frac{\lmb_{0}\frkc_{0}-\lmb_{1}\frkc_{1}}{\frkb_{0}-\frkb_{1}}-1\qquad\qquad\\
+\sum_{j=0}^{1}L_{j}^{(\nu)}(\rd_{\tau}-\lmb_{j})[\ell_{j}(\bm{\eps})]\bigg| & \aleq\frac{1}{|\log\nu|^{3}},
\end{aligned}
\label{eq:RefModStep5-3}
\end{equation}
where 
\[
L_{0}^{(\nu)}\coloneqq-\frac{1}{\nu(\frkb_{0}-\frkb_{1})}\qquad\text{and}\qquad L_{1}^{(\nu)}\coloneqq\frac{1}{\nu(\frkb_{0}-\frkb_{1})}.
\]
We look at each term of \eqref{eq:RefModStep5-3}. First, we use \eqref{eq:nu-b-bound}
and \eqref{eq:frkb-asymp}-\eqref{eq:frkc-asymp} to have 
\[
\Big|-\frac{b_{\tau}}{b}\Big(\frac{\frac{b}{\nu}\cdot(\frkc_{0}-\frkc_{1})}{\frkb_{0}-\frkb_{1}}\Big)\Big|\aleq\frac{1}{|\log\nu|}\cdot\frac{1}{|\log\nu|}\aleq\frac{1}{|\log\nu|^{2}}.
\]
Next, we use \eqref{eq:nu-b-bound}, \eqref{eq:lmb-hat-asymp}, and
\eqref{eq:frkb-asymp}-\eqref{eq:frkc-asymp} to have 
\begin{align*}
\frac{b}{\nu}\frac{\lmb_{0}\frkc_{0}-\lmb_{1}\frkc_{1}}{\frkb_{0}-\frkb_{1}}-1 & =\frac{b}{\nu}\Big\{1+\frac{\frac{1}{3}}{|\log\nu|}+O\Big(\frac{1}{|\log\nu|^{2}}\Big)\Big\}-1\\
 & =\Big(\frac{b}{\nu}-1\Big)+\frac{\frac{1}{3}}{|\log\nu|}+O\Big(\frac{1}{|\log\nu|^{2}}\Big).
\end{align*}
Substituting the previous two displays into \eqref{eq:RefModStep5-3},
we have proved 
\begin{equation}
\bigg|\frac{\nu_{\tau}}{\nu}+\Big(\frac{b}{\nu}-1\Big)+\frac{\frac{1}{3}}{|\log\nu|}+\sum_{j=0}^{1}L_{j}^{(\nu)}(\rd_{\tau}-\lmb_{j})[\ell_{j}(\bm{\eps})]\bigg|\aleq\frac{1}{|\log\nu|^{2}},\label{eq:RefModStep5-4}
\end{equation}
which is \eqref{eq:RefMod-nu}.

We turn to the proof of \eqref{eq:RefMod-b}. We divide \eqref{eq:RefModStep3-4}
for $j=1$ by $-\frac{b}{\nu}\frkc_{1}$ to obtain 
\begin{equation}
\bigg|\frac{b_{\tau}}{b}-\frac{\nu_{\tau}}{\nu}\Big(\frac{\nu}{b}\frac{\frkb_{1}+2}{\frkc_{1}}\Big)+\frac{\nu}{b}\frac{\frkb_{1}}{\frkc_{1}}-\lmb_{1}+\wh L_{1}^{(b)}(\rd_{\tau}-\lmb_{1})[\ell_{1}(\bm{\eps})]\bigg|\aleq\frac{1}{|\log\nu|^{3}},\label{eq:RefModStep5-1}
\end{equation}
where 
\[
\wh L_{0}^{(b)}\coloneqq0,\qquad\wh L_{1}^{(b)}\coloneqq\frac{1}{b\frkc_{1}}.
\]
Let us look at each term of \eqref{eq:RefModStep5-1}. First, we use
\eqref{eq:nu-b-bound} and \eqref{eq:frkb-asymp}-\eqref{eq:frkc-asymp}
to have 
\[
-\frac{\nu_{\tau}}{\nu}\Big(\frac{\nu}{b}\frac{\frkb_{1}+2}{\frkc_{1}}\Big)=\frac{\nu_{\tau}}{\nu}\Big(\frac{-\frac{1}{6}}{|\log\nu|}+O\Big(\frac{1}{|\log\nu|^{2}}\Big)\Big)=\frac{\nu_{\tau}}{\nu}\frac{-\frac{1}{6}}{|\log\nu|}+O\Big(\frac{1}{|\log\nu|^{3}}\Big).
\]
Next, we use $\frac{\nu}{b}=1-(\frac{b}{\nu}-1)+O(|\frac{b}{\nu}-1|^{2})=1-(\frac{b}{\nu}-1)+O(\frac{1}{|\log\nu|^{2}})$,
\eqref{eq:frkb-asymp}-\eqref{eq:frkc-asymp}, and \eqref{eq:lmb-hat-asymp}
to have
\begin{equation}
\begin{aligned}\frac{\nu}{b}\frac{\frkb_{1}}{\frkc_{1}}-\lmb_{1} & =\Big(\frac{b}{\nu}-1\Big)\frac{-\frkb_{1}}{\frkc_{1}}+\Big(\frac{\frkb_{1}}{\frkc_{1}}-\lmb_{1}\Big)+O\Big(\frac{1}{|\log\nu|^{3}}\Big)\\
 & =\Big(\frac{b}{\nu}-1\Big)\frac{\frac{1}{3}}{|\log\nu|}+\Big(\frac{\frac{1}{2}}{|\log\nu|}+\frac{\frac{13}{36}-\frac{\log2}{2}}{|\log\nu|^{2}}\Big)+O\Big(\frac{1}{|\log\nu|^{3}}\Big).
\end{aligned}
\label{eq:RefModStep5-5}
\end{equation}
Substituting the previous two displays into \eqref{eq:RefModStep5-1},
we have 
\begin{equation}
\begin{aligned}\bigg|\frac{b_{\tau}}{b}+\frac{\nu_{\tau}}{\nu}\frac{-\frac{1}{6}}{|\log\nu|}+\Big(\frac{b}{\nu}-1\Big)\frac{\frac{1}{3}}{|\log\nu|}+\Big(\frac{\frac{1}{2}}{|\log\nu|}+\frac{\frac{13}{36}-\frac{\log2}{2}}{|\log\nu|^{2}}\Big)\\
+\wh L_{1}^{(b)}(\rd_{\tau}-\lmb_{1})[\ell_{1}(\bm{\eps})]\bigg|\aleq & \frac{1}{|\log\nu|^{3}}.
\end{aligned}
\label{eq:RefModStep5-2}
\end{equation}
Finally, substituting \eqref{eq:RefModStep5-4} into \eqref{eq:RefModStep5-2}
(in other words, we add \eqref{eq:RefModStep5-2} and $\frac{1}{6|\log\nu|}$
times \eqref{eq:RefModStep5-4}), we obtain 
\[
\bigg|\frac{b_{\tau}}{b}+\Big(\frac{b}{\nu}-1\Big)\frac{\frac{1}{2}}{|\log\nu|}+\frac{\frac{1}{2}}{|\log\nu|}+\frac{\frac{5}{12}-\frac{\log2}{2}}{|\log\nu|^{2}}+\sum_{j=0}^{1}L_{j}^{(b)}(\rd_{\tau}-\lmb_{j})[\ell_{j}(\bm{\eps})]\bigg|\aleq\frac{1}{|\log\nu|^{3}},
\]
where 
\[
L_{j}^{(b)}\coloneqq\wh L_{j}^{(b)}+\frac{L_{j}^{(\nu)}}{6|\log\nu|}.
\]
This completes the proof of \eqref{eq:RefMod-b}.

\textbf{Step 6.} Proofs of \eqref{eq:Lj-1}-\eqref{eq:Lj-3}.

We finish the proof by showing \eqref{eq:Lj-1}-\eqref{eq:Lj-3}.
Recall that 
\begin{align*}
L_{0}^{(\nu)} & =-\frac{1}{\nu(\frkb_{0}-\frkb_{1})}, & L_{1}^{(\nu)} & =\frac{1}{\nu(\frkb_{0}-\frkb_{1})},\\
L_{0}^{(b)} & =\frac{L_{0}^{(\nu)}}{6|\log\nu|}, & L_{1}^{(b)} & =\frac{1}{b\frkc_{1}}+\frac{L_{1}^{(\nu)}}{6|\log\nu|}.
\end{align*}
Note that \eqref{eq:Lj-1} is immediate from \eqref{eq:frkb-asymp}-\eqref{eq:frkc-asymp}.
The proof of \eqref{eq:Lj-3} follows from \eqref{eq:Lj-1}, \eqref{eq:frkb-asymp}-\eqref{eq:frkc-asymp},
and $|\frac{b}{\nu}-1|\aleq\frac{1}{|\log\nu|}$:
\begin{align*}
|L_{0}^{(b)}| & \aleq\frac{|L_{0}^{(\nu)}|}{|\log\nu|}\aleq\frac{1}{\nu|\log\nu|^{2}},\\
|L_{1}^{(\nu)}-L_{1}^{(b)}| & \leq\Big|\frac{1}{\nu(\frkb_{0}-\frkb_{1})}-\frac{1}{b\frkc_{1}}\Big|+\frac{|L_{1}^{(\nu)}|}{6|\log\nu|}\aleq\frac{1}{\nu|\log\nu|^{2}}.
\end{align*}
It remains to show \eqref{eq:Lj-2}. From the formulas of $L_{j}^{(\nu)}$
and $L_{j}^{(b)}$ and \eqref{eq:nu-b-bound}, we have 
\begin{align*}
|\rd_{\tau}L_{j}^{(\nu)}|+|\rd_{\tau}L_{j}^{(b)}| & \aleq\Big(\Big|\frac{\nu_{\tau}}{\nu}\Big|+\Big|\frac{b_{\tau}}{b}\Big|\Big)\frac{1}{\nu|\log\nu|}+\frac{|\rd_{\tau}(\frkb_{0}-\frkb_{1})|+|\rd_{\tau}\frkc_{1}|}{\nu|\log\nu|^{2}}\\
 & \aleq\frac{1}{\nu|\log\nu|^{2}}+\frac{|\nu\rd_{\nu}(\frkb_{0}-\frkb_{1})|+|\nu\rd_{\nu}\frkc_{1}|}{\nu|\log\nu|^{3}}.
\end{align*}
Applying \eqref{eq:nudnu-frk} to the above yields \eqref{eq:Lj-2}.
This completes the proof.
\end{proof}

\subsection{\label{subsec:Proof-of-Theorem}Proof of Theorem~\ref{thm:MainThm}}

In this final subsection, we finish the proof of Theorem~\ref{thm:MainThm}.
In the following lemma, we show that \eqref{eq:RefMod-nu} and \eqref{eq:RefMod-b}
improve the rough estimate \eqref{eq:nu-b-bound} to a more refined
relation between $b$ and $\nu$.
\begin{lem}[Compatibility of $\nu$ and $b$ for blow-up]
\label{lem:compatibility}We have 
\begin{equation}
\frac{b}{\nu}=1+\frac{b_{1}-\nu_{1}}{|\log\nu|}+L_{0}^{(\nu)}\ell_{0}(\bm{\eps})+O\Big(\frac{1}{|\log\nu|^{2}}\Big),\label{eq:Compatibility}
\end{equation}
\end{lem}

\begin{proof}
Let 
\begin{equation}
\beta\coloneqq\frac{b}{\nu}-1-\frac{\beta_{1}}{|\log\nu|}-L_{0}^{(\nu)}\ell_{0}(\bm{\eps}),\label{eq:Compat-1}
\end{equation}
where $\beta_{1}\coloneqq b_{1}-\nu_{1}$. It suffices to show that
$|\beta|\aleq\frac{1}{|\log\nu|^{2}}$.

\textbf{Step 1.} Computation of $\beta_{\tau}$.

Thanks to \eqref{eq:nu-b-bound}, we have 
\[
\beta_{\tau}=\frac{b}{\nu}\Big(\frac{b_{\tau}}{b}-\frac{\nu_{\tau}}{\nu}\Big)-\Big(L_{0}^{(\nu)}\rd_{\tau}[\ell_{0}(\bm{\eps})]+(\rd_{\tau}L_{0}^{(\nu)})\ell_{0}(\bm{\eps})\Big)+O\Big(\frac{1}{|\log\nu|^{3}}\Big).
\]
Thanks to \eqref{eq:nu-b-bound}, \eqref{eq:Lj-2}, and \eqref{eq:eps-bound-3},
we have 
\[
\beta_{\tau}=\Big(\frac{b_{\tau}}{b}-\frac{\nu_{\tau}}{\nu}\Big)-L_{0}^{(\nu)}\rd_{\tau}[\ell_{0}(\bm{\eps})]+O\Big(\frac{1}{|\log\nu|^{2}}\Big).
\]
Next, we apply \eqref{eq:RefMod-nu}, \eqref{eq:RefMod-b}, and \eqref{eq:nu-b-bound}
to have 
\begin{align*}
\beta_{\tau} & =\Big(\frac{b}{\nu}-1\Big)+\frac{\nu_{1}-b_{1}}{|\log\nu|}-L_{0}^{(\nu)}\lmb_{0}\ell_{0}(\bm{\eps})\\
 & \peq-L_{0}^{(b)}(\rd_{\tau}-\lmb_{0})[\ell_{0}(\bm{\eps})]+(L_{1}^{(\nu)}-L_{1}^{(b)})(\rd_{\tau}-\lmb_{1})[\ell_{1}(\bm{\eps})]+O\Big(\frac{1}{|\log\nu|^{2}}\Big).
\end{align*}
Using \eqref{eq:Lj-3}, \eqref{eq:eps-bound-5}, and \eqref{eq:eps-bound-3},
the terms in the second line of the above display are of size $O(\frac{1}{|\log\nu|^{2}})$.
Thus we have 
\[
\beta_{\tau}=\Big(\frac{b}{\nu}-1\Big)+\frac{\nu_{1}-b_{1}}{|\log\nu|}-L_{0}^{(\nu)}\lmb_{0}\ell_{0}(\bm{\eps})+O\Big(\frac{1}{|\log\nu|^{2}}\Big).
\]
In view of \eqref{eq:Compat-1}, we replace $\frac{b}{\nu}-1$ by
$\frac{\beta_{1}}{|\log\nu|}+L_{0}^{(\nu)}\ell_{0}(\bm{\eps})+\beta$
to arrive at 
\begin{equation}
\beta_{\tau}=\frac{\beta_{1}+\nu_{1}-b_{1}}{|\log\nu|}+(1-\lmb_{0})L_{0}^{(\nu)}\ell_{0}(\bm{\eps})+\beta+O\Big(\frac{1}{|\log\nu|^{2}}\Big).\label{eq:Compat-1-1}
\end{equation}

Let us look at each term of \eqref{eq:Compat-1-1}. The first term
\emph{vanishes} thanks to the choice $\beta_{1}=b_{1}-\nu_{1}$. For
the second term, we use \eqref{eq:Lj-1} and \eqref{eq:eps-bound-3}
to have 
\[
\big|(1-\lmb_{0})L_{0}^{(\nu)}\ell_{0}(\bm{\eps})\big|\aleq\frac{1}{|\log\nu|}\cdot\frac{1}{\nu|\log\nu|}\cdot\frac{\nu}{|\log\nu|^{\frac{1}{2}}}\aleq\frac{1}{|\log\nu|^{\frac{5}{2}}}.
\]
Finally by \eqref{eq:rough-asymp-nu-b} we have $\frac{1}{|\log\nu|^{2}}\sim\tau^{-1}$.
As a result, \eqref{eq:Compat-1-1} becomes 
\begin{equation}
\beta_{\tau}=\beta+O(\tau^{-1}).\label{eq:Compat-2}
\end{equation}

\textbf{Step 2.} Backward integration of the $\beta_{\tau}$-equation.

To conclude \eqref{eq:Compatibility}, we integrate the $\beta_{\tau}$-equation
\eqref{eq:Compat-2} backwards in time from $\tau=+\infty$. We rewrite
\eqref{eq:Compat-2} as 
\[
|(e^{-\tau}\beta(\tau))_{\tau}|\aleq e^{-\tau}\tau^{-1}.
\]
Integrating this from $+\infty$ to $\tau$ yields 
\[
|\beta(\tau)|\aleq e^{\tau}\int_{\tau}^{\infty}e^{-\tau'}(\tau')^{-1}d\tau'\aleq\tau^{-1}\sim\frac{1}{|\log\nu|^{2}}.
\]
This completes the proof of \eqref{eq:Compatibility}.
\end{proof}
Substituting the refined relation \eqref{eq:Compatibility} into the
the modulation equation for $b$, we obtain the following sharp modulation
equation only in terms of $b$ and $\bm{\eps}$.
\begin{cor}[Sharp modulation equation for $b$]
\label{cor:Sharp-Mod}We have 
\begin{equation}
\Big|\frac{b_{\tau}}{b}+\frac{\frac{1}{2}}{|\log b|}+\frac{\frac{1}{2}-\frac{\log2}{2}}{|\log b|^{2}}+\sum_{j=0}^{1}L_{j}^{(b)}\rd_{\tau}[\ell_{j}(\bm{\eps})]\Big|\aleq\frac{1}{|\log b|^{\frac{5}{2}}}.\label{eq:SharpMod}
\end{equation}
\end{cor}

\begin{proof}
Substituting \eqref{eq:Compatibility} into \eqref{eq:RefMod-b} yields
\begin{equation}
\begin{aligned}\Big|\frac{b_{\tau}}{b}+\frac{b_{1}}{|\log\nu|}+\frac{b_{2}+(b_{1}-\nu_{1})b_{-1}}{|\log\nu|^{2}}+\frac{b_{-1}L_{0}^{(\nu)}}{|\log\nu|}\ell_{0}(\bm{\eps})\qquad\\
+\sum_{j=0}^{1}L_{j}^{(b)}(\rd_{\tau}-\lmb_{j})[\ell_{j}(\bm{\eps})]\Big| & \aleq\frac{1}{|\log\nu|^{3}}.
\end{aligned}
\label{eq:SharpMod-2}
\end{equation}
In the above display, the following terms are considered as errors.
We use \eqref{eq:Lj-1}, $|\lmb_{1}|\aleq\frac{1}{|\log\nu|}$ and
\eqref{eq:eps-bound-3} to have 
\[
\Big|\frac{b_{-1}L_{0}^{(\nu)}}{|\log\nu|}\ell_{0}(\bm{\eps})\Big|+\Big|-L_{1}^{(b)}\lmb_{1}\ell_{1}(\bm{\eps})\Big|\aleq\frac{1}{|\log\nu|^{\frac{5}{2}}}.
\]
Next, we use \eqref{eq:Lj-3} and \eqref{eq:eps-bound-3} to have
\[
\Big|-L_{0}^{(b)}\lmb_{0}\ell_{0}(\bm{\eps})\Big|\aleq\frac{1}{|\log\nu|^{\frac{5}{2}}}.
\]
Substituting the above error estimates into \eqref{eq:SharpMod-2},
we have 
\[
\Big|\frac{b_{\tau}}{b}+\frac{b_{1}}{|\log\nu|}+\frac{b_{2}+(b_{1}-\nu_{1})b_{-1}}{|\log\nu|^{2}}+\sum_{j=0}^{1}L_{j}^{(b)}\rd_{\tau}[\ell_{j}(\bm{\eps})]\Big|\aleq\frac{1}{|\log\nu|^{\frac{5}{2}}}.
\]
Replacing $\nu_{1},b_{-1},b_{1},b_{2}$ by their values \eqref{eq:values-nu1-b-12}
and applying $\log\nu=\log b+O(\frac{1}{|\log b|})$ (due to \eqref{eq:nu-b-bound})
completes the proof.
\end{proof}
We are now ready to prove Theorem~\ref{thm:MainThm}.
\begin{proof}[Proof of Theorem~\ref{thm:MainThm}]
The proof follows from integrating \eqref{eq:SharpMod}.

\textbf{Step 1.} A correction $\ul b$.

Let us introduce a correction $\ul b$ to $b$: 
\begin{equation}
\ul b\coloneqq b\Big(1+\sum_{j=0}^{1}L_{j}^{(b)}\ell_{j}(\bm{\eps})\Big).\label{eq:SharpMod-3}
\end{equation}
Using \eqref{eq:Lj-1} and \eqref{eq:eps-bound-3}, we see that $\ul b$
is a correction of $b$ in the sense that 
\begin{equation}
\ul b=b\Big(1+O\Big(\frac{1}{|\log\nu|^{\frac{3}{2}}}\Big)\Big).\label{eq:SharpMod-4}
\end{equation}
With this $\ul b$, we have 
\[
\frac{\ul b_{\tau}}{\ul b}=\frac{b_{\tau}}{b}+\frac{\sum_{j=0}^{1}L_{j}^{(b)}\rd_{\tau}[\ell_{j}(\bm{\eps})]}{1+\sum_{j=0}^{1}L_{j}^{(b)}\ell_{j}(\bm{\eps})}+\frac{\sum_{j=0}^{1}(\rd_{\tau}L_{j}^{(b)})\ell_{j}(\bm{\eps})}{1+\sum_{j=0}^{1}L_{j}^{(b)}\ell_{j}(\bm{\eps})}.
\]
Using \eqref{eq:Lj-1}, \eqref{eq:Lj-2}, \eqref{eq:eps-bound-3},
and \eqref{eq:eps-bound-5}, we have 
\begin{align*}
\frac{\ul b_{\tau}}{\ul b} & =\frac{b_{\tau}}{b}+\Big(\sum_{j=0}^{1}L_{j}^{(b)}\rd_{\tau}[\ell_{j}(\bm{\eps})]\Big)\Big(1+O\Big(\frac{1}{|\log\nu|^{\frac{3}{2}}}\Big)\Big)+O\Big(\frac{1}{|\log\nu|^{\frac{5}{2}}}\Big).\\
 & =\frac{b_{\tau}}{b}+\sum_{j=0}^{1}L_{j}^{(b)}\rd_{\tau}[\ell_{j}(\bm{\eps})]+O\Big(\frac{1}{|\log\nu|^{\frac{5}{2}}}\Big).
\end{align*}
Substituting \eqref{eq:SharpMod} into the above, we get 
\[
\Big|\frac{\ul b_{\tau}}{\ul b}+\frac{\frac{1}{2}}{|\log b|}+\frac{\frac{1}{2}-\frac{\log2}{2}}{|\log b|^{2}}\Big|\aleq\frac{1}{|\log b|^{\frac{5}{2}}}.
\]
Finally using \eqref{eq:SharpMod-4}, we have 
\begin{equation}
\Big|\frac{\ul b_{\tau}}{\ul b}+\frac{\frac{1}{2}}{|\log\ul b|}+\frac{\frac{1}{2}-\frac{\log2}{2}}{|\log\ul b|^{2}}\Big|\aleq\frac{1}{|\log\ul b|^{\frac{5}{2}}}.\label{eq:SharpMod-5}
\end{equation}

\textbf{Step 2.} Integration of \eqref{eq:SharpMod-5}.

It is more convenient to introduce the variable 
\[
\mu\coloneqq|\log\ul b|=-\log\ul b\sim\sqrt{\tau}
\]
so that \eqref{eq:SharpMod-5} becomes 
\[
\bigg|\mu_{\tau}-\frac{\frac{1}{2}}{\mu}-\frac{\frac{1}{2}-\frac{\log2}{2}}{\mu^{2}}\bigg|\aleq\frac{1}{\tau^{5/4}}.
\]
Multiplying the above by $2\mu$ and using $\mu=\sqrt{\tau}+O(1)$
from \eqref{eq:rough-asymp-nu-b}, we have 
\[
\Big|(\mu^{2})_{\tau}-1-\frac{1-\log2}{\sqrt{\tau}}\Big|\aleq\frac{1}{\tau^{3/4}}.
\]
Integrating this, we have 
\[
\mu^{2}=\tau+(1-\log2)\cdot2\sqrt{\tau}+O(\tau^{1/4})
\]
and hence 
\[
\mu=\sqrt{\tau}+(1-\log2)+O(\tau^{-1/4}).
\]
Going back to the variable $\ul b$, we have proved that 
\begin{equation}
\ul b=2e^{-1}e^{-\sqrt{\tau}}(1+o_{\tau\to\infty}(1)).\label{eq:SharpMod-6}
\end{equation}

\textbf{Step 3.} Conclusion.

By \eqref{eq:nu-b-bound} and \eqref{eq:SharpMod-4}, we have 
\[
\nu=\ul b(1+o_{\tau\to\infty}(1)).
\]
Thus by \eqref{eq:SharpMod-6}, we have 
\[
\nu(\tau)=2e^{-1}e^{-\sqrt{\tau}}(1+o_{\tau\to\infty}(1)).
\]
In terms of the original spacetime coordinates $(t,r)$, this reads
\[
\lmb(t)=2e^{-1}(T-t)e^{-\sqrt{|\log(T-t)|}}(1+o_{t\to T}(1)).
\]
This completes the proof of Theorem~\ref{thm:MainThm}.
\end{proof}

\appendix

\section{\label{sec:Proof-lmb-b-rel}Proof of \eqref{eq:lmb-b-rel}}
\begin{proof}[Proof of \eqref{eq:lmb-b-rel}]
In this proof, we denote $\wh{\lmb}$ and $\wh b$ by $\lmb$ and
$b$, respectively, and abbreviate the integral $\int_{t}^{T}dt'$
as $\int_{t}^{T}$. We will also rely on the following facts:
\begin{itemize}
\item \cite[p.108]{RaphaelRodnianski2012Publ.Math.} Almost monotonicity
of $b$: for $t'\in[t,T)$,
\begin{equation}
\frac{b^{2}(t')}{|\log b(t')|}\leq2\frac{b^{2}(t)}{|\log b(t)|}.\label{eq:apdx-a-1}
\end{equation}
\item \cite[(5.34)]{RaphaelRodnianski2012Publ.Math.} Control of $b_{t}$:
\begin{equation}
\lmb|b_{t}|\aleq\frac{b^{2}}{|\log b|}.\label{eq:apdx-a-2}
\end{equation}
\end{itemize}
The overall scheme of the proof is to follow the proof in \cite[pp.107--108 Step 2]{RaphaelRodnianski2012Publ.Math.},
but we also need to keep $\frac{1}{|\log b|}$-smallness. We study
the quantity 
\[
\int_{t}^{T}b^{2}(t')dt'
\]
in two ways.

On one hand, using $\lmb_{t}+b=0$ and \eqref{eq:apdx-a-2}, we observe
\[
\int_{t}^{T}b^{2}=\int_{t}^{T}-b\lmb_{t}=b(t)\lmb(t)+\int_{t}^{T}\lmb b_{t}=b(t)\lmb(t)+\int_{t}^{T}O\Big(\frac{b^{2}}{|\log b|}\Big).
\]
Thus we have 
\begin{equation}
b(t)\lmb(t)=\int_{t}^{T}\Big(1+O\Big(\frac{1}{|\log b|}\Big)\Big)b^{2}=\Big(1+O\Big(\frac{1}{|\log b(t)|}\Big)\Big)\int_{t}^{T}b^{2},\label{eq:Appendix-tmp1}
\end{equation}
where in the last equality we used the almost monotonicity \eqref{eq:apdx-a-1}
of $b$.

On the other hand, we integrate by parts and use \eqref{eq:apdx-a-1}
to have 
\begin{align*}
\bigg|\frac{1}{(T-t)b^{2}(t)}\int_{t}^{T}b^{2}-1\bigg| & =\frac{2}{(T-t)b^{2}(t)}\bigg|\int_{t}^{T}bb_{t}\cdot(T-t')\bigg|\\
 & \aleq\frac{1}{(T-t)|\log b(t)|}\int_{t}^{T}\frac{b}{\lmb}\cdot(T-t').
\end{align*}
Using $\log\lmb(t)=\log(T-t)-\sqrt{|\log(T-t)|}+O(1)$ of \eqref{eq:lmb-RR-asymp},
we have 
\[
\int_{t}^{T}\frac{b}{\lmb}(T-t')=\int_{t}^{T}\frac{-\lmb_{t}}{\lmb}(T-t')=(T-t)\log\lmb(t)-\int_{t}^{T}\log\lmb=O(T-t).
\]
Combining the above two displays yields 
\begin{equation}
\frac{1}{(T-t)b^{2}(t)}\int_{t}^{T}b^{2}=1+O\Big(\frac{1}{|\log b(t)|}\Big).\label{eq:Appendix-tmp2}
\end{equation}

The estimates \eqref{eq:Appendix-tmp1} and \eqref{eq:Appendix-tmp2}
yield 
\[
\frac{\lmb(t)}{(T-t)b(t)}=\frac{1}{(T-t)b^{2}(t)}\cdot b(t)\lmb(t)=1+O\Big(\frac{1}{|\log b(t)|}\Big),
\]
completing the proof of \eqref{eq:lmb-b-rel}.
\end{proof}

\section{\label{sec:Localized-Hardy-inequalities}Localized Hardy inequalities}

In this appendix, we show the localized coercivity estimates \eqref{eq:loc-coercivity-1}-\eqref{eq:loc-coercivity-2}
as well as the weighted $L^{\infty}$-estimates (Lemma~\ref{lem:Local-Linfty-est})
used in Sections~\ref{subsec:RR-blow-up-sol}-\ref{subsec:Decomposition-of-solutions}.
\begin{lem}[Local subcoercivity estimates and kernel characterization]
\label{lem:loc-subcoer}For $R\gg1$, we have 
\begin{align}
\|\chf_{(0,R]}A^{\ast}Af\|_{L^{2}}+\|\chf_{y\sim1}f\|_{L^{2}} & \sim\|f\|_{(\dot{\calH}_{1}^{2})_{R}},\qquad\forall f\in\dot{\calH}_{1}^{2},\label{eq:loc-subcoer-1}\\
\|\chf_{(0,R]}Ag\|_{L^{2}}+\|\chf_{y\sim1}g\|_{L^{2}} & \sim\|g\|_{(\dot{H}_{1}^{1})_{R}},\qquad\forall g\in\dot{H}_{1}^{1},\label{eq:loc-subcoer-2}
\end{align}
where the implicit constants are uniform in $R$. Moreover, the kernels
of $A^{\ast}A:(\dot{\calH}_{1}^{2})_{R}\to L^{2}(r<R)$ and $A:(\dot{H}_{1}^{1})_{R}\to L^{2}(r<R)$
are equal to the span of $\Lmb Q$.
\end{lem}

\begin{proof}
Without localization $\chf_{(0,R]}$, these subcoercivity estimates
are already proved in \cite[Lemma B.2]{RaphaelRodnianski2012Publ.Math.}.
The key point here is to keep track of the localization $\chf_{(0,R]}$
while we reproduce the proof.

\textbf{Step 1.} Proof of \eqref{eq:loc-subcoer-2}.

Let us first show \eqref{eq:loc-subcoer-2}. As the $\aleq$-inequality
is clear, we focus on the proof of the $\ageq$-inequality. An integration
by parts yields the identity 
\[
\int_{0}^{R}|Ag|^{2}ydy=\int_{0}^{R}\Big(|\rd_{y}g|^{2}+\frac{V}{y^{2}}|g|^{2}\Big)ydy+\frac{-1+R^{2}}{1+R^{2}}|g|^{2}(R).
\]
For $R\gg1$, the last term is \emph{nonnegative} (which is a good
sign). Thus we have 
\begin{align*}
\int_{0}^{R}|Ag|^{2}ydy & \geq\int_{0}^{R}\Big(|\rd_{y}g|^{2}+\frac{1}{y^{2}}|g|^{2}+\frac{V-1}{y^{2}}|g|^{2}\Big)ydy\\
 & =\|g\|_{(\dot{H}_{1}^{1})_{R}}^{2}-O\Big(\int_{0}^{R}\frac{1}{(1+y^{2})^{2}}|g|^{2}ydy\Big).
\end{align*}
Since 
\[
\int_{0}^{R}\frac{1}{(1+y^{2})^{2}}|g|^{2}ydy\aleq\|\chf_{[r_{0}^{-1},r_{0}]}g\|_{L^{2}}^{2}+r_{0}^{-2}\|g\|_{(\dot{H}_{1}^{1})_{R}}^{2},
\]
taking $r_{0}>1$ large yields 
\[
\int_{0}^{R}|Ag|^{2}ydy+\|\chf_{[r_{0}^{-1},r_{0}]}g\|_{L^{2}}^{2}\ageq\|g\|_{(\dot{H}_{1}^{1})_{R}}^{2}.
\]
This completes the proof of \eqref{eq:loc-subcoer-2}.

\textbf{Step 2.} Localized Hardy controls from $\Delta_{1}$ and $\rd_{yy}$.

From now on, we turn to the more delicate proof of \eqref{eq:loc-subcoer-1}.
We follow the proof of \cite[Lemma A.7]{KimKwonOh2020arXiv}, but
we also need to keep track of the localization $\chf_{(0,R]}$. In
this step, we claim the following Hardy controls from $-\Delta_{1}$
and $\rd_{yy}$:
\begin{equation}
\|\chf_{(0,R]}\Delta_{1}f\|_{L^{2}}\sim\Big\|\chf_{(0,R]}\Big|\Big(\rd_{y}-\frac{1}{y}\Big)f\Big|_{-1}\Big\|_{L^{2}}\sim\|\chf_{(0,R]}\rd_{yy}f\|_{L^{2}}.\label{eq:appendix-a4}
\end{equation}
To see this, let us note the operator identity 
\begin{align*}
\Delta_{1} & =(\rd_{y}+\frac{2}{y})(\rd_{y}-\frac{1}{y}),\\
\rd_{yy} & =(\rd_{y}+\frac{1}{y})(\rd_{y}-\frac{1}{y}).
\end{align*}
Now we recall Hardy's inequality for $\rd_{y}+\frac{k}{y}$: For $k>0$
and $0<r_{1}<r_{2}<\infty$, we have 
\[
\int_{r_{1}}^{r_{2}}\Big|\frac{g}{y}\Big|^{2}ydy\aleq\int_{r_{1}}^{r_{2}}\Big|\Big(\rd_{y}+\frac{k}{y}\Big)g\Big|^{2}ydy+|g|^{2}(r_{1}).
\]
We let $k\in\{1,2\}$, $g=(\rd_{y}-\frac{1}{y})f$, $r_{2}=R$, and
take $r_{1}\to0$. Here, $|g|^{2}(r_{1})\to0$ thanks to the density
argument and the degeneracy $|g(y)|\leq C(f)\cdot y^{2}$ for the
radial part $f$ of a smooth $1$-equivariant function (See for example
\cite[Proposition 2.3]{RaphaelRodnianski2012Publ.Math.}). Thus we
have 
\[
\int_{0}^{R}\Big|\frac{1}{y}\Big(\rd_{y}-\frac{1}{y}\Big)f\Big|^{2}ydy\aleq\int_{0}^{R}\Big|\Big(\rd_{y}+\frac{k}{y}\Big)\Big(\rd_{y}-\frac{1}{y}\Big)f\Big|^{2}ydy
\]
for $k\in\{1,2\}$. This implies \eqref{eq:appendix-a4}.

\textbf{Step 3.} Proof of \eqref{eq:loc-subcoer-1}.

In this step, we use the Hardy controls \eqref{eq:appendix-a4} to
prove \eqref{eq:loc-subcoer-1}. We use 
\[
A^{\ast}A=-\rd_{yy}-\frac{1}{y}\rd_{y}+\frac{V}{y^{2}}=-\Delta_{1}+O\Big(\frac{1}{(1+y^{2})^{2}}\Big),
\]
to have 
\begin{equation}
\big|\|\chf_{(0,R]}A^{\ast}Af\|_{L^{2}}-\|\chf_{(0,R]}\Delta_{1}f\|_{L^{2}}\big|\aleq\|\chf_{[r_{0}^{-1},r_{0}]}f\|_{L^{2}}+r_{0}^{-2+}\|f\|_{(\dot{\calH}_{1}^{2})_{R}}.\label{eq:appendix-a3}
\end{equation}
Now the $\aleq$-inequality of \eqref{eq:loc-subcoer-1} easily follows
from \eqref{eq:appendix-a3} and \eqref{eq:appendix-a4}: 
\begin{align*}
\|\chf_{(0,R]}A^{\ast}Af\|_{L^{2}} & \aleq\|\chf_{(0,R]}\Delta_{1}f\|_{L^{2}}+\|f\|_{(\dot{\calH}_{1}^{2})_{R}}\\
 & \aleq\|\chf_{(0,R]}\rd_{yy}f\|_{L^{2}}+\|f\|_{(\dot{\calH}_{1}^{2})_{R}}\aleq\|f\|_{(\dot{\calH}_{1}^{2})_{R}}.
\end{align*}
We turn to the proof of the $\ageq$-inequality of \eqref{eq:loc-subcoer-1}.
First, by \eqref{eq:appendix-a4} and \eqref{eq:appendix-a3}, we
have 
\begin{equation}
\begin{aligned} & \Big\|\chf_{(0,R]}\Big|\Big(\rd_{y}-\frac{1}{y}\Big)f\Big|_{-1}\Big\|_{L^{2}}\\
 & \quad\aleq\|\chf_{(0,R]}\Delta_{1}f\|_{L^{2}}\aleq\|\chf_{(0,R]}A^{\ast}Af\|_{L^{2}}+\|\chf_{[r_{0}^{-1},r_{0}]}f\|_{L^{2}}+r_{0}^{-2+}\|f\|_{(\dot{\calH}_{1}^{2})_{R}}.
\end{aligned}
\label{eq:appendix-a5}
\end{equation}
Next, we use 
\[
\Big\|\chf_{(0,R]}\frac{1}{y\langle\log y\rangle}|f|_{-1}\Big\|_{L^{2}}\aleq\Big\|\chf_{(0,R]}\frac{1}{y\langle\log y\rangle}\Big(\rd_{y}-\frac{1}{y}\Big)f\Big\|_{L^{2}}+\Big\|\chf_{(0,R]}\frac{1}{y^{2}\langle\log y\rangle}f\Big\|_{L^{2}},
\]
the following logarithmic Hardy's inequality
\[
\Big\|\chf_{(0,R]}\frac{1}{y^{2}\langle\log y\rangle}f\Big\|_{L^{2}}\aleq\Big\|\chf_{(0,R]}\frac{1}{y}\Big(\rd_{y}-\frac{1}{y}\Big)f\Big\|_{L^{2}}+\|\chf_{[1,2]}f\|_{L^{2}},
\]
and \eqref{eq:appendix-a5} to have 
\begin{equation}
\begin{aligned}\Big\|\chf_{(0,R]}\frac{1}{y\langle\log y\rangle} & |f|_{-1}\Big\|_{L^{2}}\aleq\Big\|\chf_{(0,R]}\Big|\Big(\rd_{y}-\frac{1}{y}\Big)f\Big|_{-1}\Big\|_{L^{2}}+\|\chf_{[1,2]}f\|_{L^{2}}\\
 & \aleq\|\chf_{(0,R]}A^{\ast}Af\|_{L^{2}}+\|\chf_{[r_{0}^{-1},r_{0}]}f\|_{L^{2}}+r_{0}^{-2+}\|f\|_{(\dot{\calH}_{1}^{2})_{R}}.
\end{aligned}
\label{eq:appendix-a6}
\end{equation}
Combining \eqref{eq:appendix-a5} and \eqref{eq:appendix-a6}, we
have shown that 
\begin{align*}
\|f\|_{(\dot{\calH}_{1}^{2})_{R}} & \aleq\Big\|\chf_{(0,R]}\Big|\Big(\rd_{y}-\frac{1}{y}\Big)f\Big|_{-1}\Big\|_{L^{2}}+\Big\|\chf_{(0,R]}\frac{1}{y\langle\log y\rangle}|f|_{-1}\Big\|_{L^{2}}\\
 & \aleq\|\chf_{(0,R]}A^{\ast}Af\|_{L^{2}}+\|\chf_{[r_{0}^{-1},r_{0}]}f\|_{L^{2}}+r_{0}^{-2+}\|f\|_{(\dot{\calH}_{1}^{2})_{R}}.
\end{align*}
Taking $r_{0}>1$ large, the $\ageq$-inequality of \eqref{eq:loc-subcoer-1}
follows.

For the kernel characterization, by the standard ODE theory, $A^{\ast}Af=0$
implies that $f$ is a linear combination of $J_{1}$ and $J_{2}$.
However, $J_{2}$ is singular at the origin ($J_{2}\sim y^{-1}$ as
$y\to0$) so $J_{2}\notin(\dot{\calH}_{1}^{2})_{R}$. Thus $f$ must
be parallel to $J_{1}$; the kernel of $A^{\ast}A$ is the span of
$J_{1}$. Note that the assertion for $A$ is obvious. This completes
the proof.
\end{proof}
By the standard argument (see e.g., \cite[Lemma B.2]{RaphaelRodnianski2012Publ.Math.}),
we can derive the desired coercivity estimates from the above subcoercivity
estimates.
\begin{cor}[Local coercivity estimates]
For $R\gg M$, we have 
\begin{align*}
\|\chf_{(0,R]}A^{\ast}Af\|_{L^{2}} & \sim_{M}\|f\|_{(\dot{\calH}_{1}^{2})_{R}},\qquad\forall f\in\dot{\calH}_{1}^{2}\cap\{\chi_{M}\Lmb Q\}^{\perp},\tag{\ref{eq:loc-coercivity-1}}\\
\|\chf_{(0,R]}Ag\|_{L^{2}} & \sim_{M}\|g\|_{(\dot{H}_{1}^{1})_{R}},\qquad\forall g\in\dot{H}_{1}^{1}\cap\{\chi_{M}\Lmb Q\}^{\perp},\tag{\ref{eq:loc-coercivity-2}}
\end{align*}
where the implicit constants are uniform in $R$.
\end{cor}

\begin{proof}
We claim that, once the estimates \eqref{eq:loc-coercivity-1}-\eqref{eq:loc-coercivity-2}
hold for some $R_{0}\gg M$, then \eqref{eq:loc-coercivity-1}-\eqref{eq:loc-coercivity-2}
hold uniformly for all $R\geq R_{0}$. Indeed, suppose that we have
an estimate 
\begin{equation}
\|\chf_{(0,R_{0}]}A^{\ast}Af\|_{L^{2}}\geq c(M)\|f\|_{(\dot{\calH}_{1}^{2})_{R_{0}}}\label{eq:loc-coer-tmp1}
\end{equation}
for some $R_{0}\gg M$. Then for any $R\geq R_{0}$, by \eqref{eq:loc-subcoer-1}
and \eqref{eq:loc-coer-tmp1}, we have 
\begin{align*}
\|f\|_{(\dot{\calH}_{1}^{2})_{R}} & \aleq\|\chf_{(0,R]}A^{\ast}Af\|_{L^{2}}+\|\chf_{y\sim1}f\|_{L^{2}}\\
 & \aleq\|\chf_{(0,R]}A^{\ast}Af\|_{L^{2}}+(c(M))^{-1}\|\chf_{(0,R_{0}]}A^{\ast}Af\|_{L^{2}}\\
 & \aleq_{M}\|\chf_{(0,R]}A^{\ast}Af\|_{L^{2}}.
\end{align*}
The argument for $\|\chf_{(0,R]}Ag\|_{L^{2}}$ is very similar and
we omit its proof.

By the previous claim, it suffices to prove \eqref{eq:loc-coercivity-1}-\eqref{eq:loc-coercivity-2}
for some $R_{0}\gg M$. With the help of Lemma~\ref{lem:loc-subcoer},
this can be proved in a standard manner; we refer to \cite[Lemma B.2]{RaphaelRodnianski2012Publ.Math.}
for a closely related proof.
\end{proof}
Finally, let us record some localized weighted $L^{\infty}$-estimates.
\begin{lem}[Local weighted $L^{\infty}$-estimates]
\label{lem:Local-Linfty-est}For any $R>10$, we have 
\begin{align}
\Big\|\chf_{(0,R]}\frac{1}{y\langle\log y\rangle^{1/2}}|f|_{1}\Big\|_{L^{\infty}} & \aleq\|f\|_{(\dot{\calH}_{1}^{2})_{R}},\qquad\forall f\in\dot{\calH}_{1}^{2},\label{eq:local-infty-est1}\\
\|\chf_{(0,R]}g\|_{L^{\infty}} & \aleq\|g\|_{(\dot{H}_{1}^{1})_{R}},\qquad\forall g\in\dot{H}_{1}^{1}.\label{eq:local-infty-est2}
\end{align}
\end{lem}

\begin{proof}
For the proof of \eqref{eq:local-infty-est1}, the fact that $(\rd_{y}-\frac{1}{y})f$
enjoys better Hardy control than $\rd_{y}f$ will play a crucial role.
More precisely, we will use the following control (as a consequence
of \eqref{eq:appendix-a4}): 
\[
\Big\|\chf_{(0,R]}\frac{1}{y}\Big(\rd_{y}-\frac{1}{y}\Big)f\Big\|_{L^{2}}\aleq\|f\|_{(\dot{\calH}_{1}^{2})_{R}}.
\]
Noting that 
\[
\rd_{y}\Big(\frac{|f|^{2}}{y^{2}\langle\log y\rangle}\Big)=\frac{1}{\langle\log y\rangle}\rd_{y}\Big(\frac{|f|^{2}}{y^{2}}\Big)-\frac{\log y}{y^{3}\langle\log y\rangle^{3}}|f|^{2},
\]
we have 
\begin{equation}
\Big|\rd_{y}\Big(\frac{|f|^{2}}{y^{2}\langle\log y\rangle}\Big)\Big|\aleq\frac{1}{y}\Big|\Big(\rd_{y}-\frac{1}{y}\Big)f\Big|\cdot\frac{|f|}{y\langle\log y\rangle}+\frac{|f|^{2}}{y^{3}\langle\log y\rangle^{2}}.\label{eq:appendix-a7}
\end{equation}
Similarly, we have 
\begin{equation}
\Big|\rd_{y}\Big(\frac{|\rd_{y}f|^{2}}{\langle\log y\rangle}\Big)\Big|\aleq|\rd_{yy}f|\cdot\frac{|\rd_{y}f|}{\langle\log y\rangle}+\frac{|\rd_{y}f|^{2}}{y\langle\log y\rangle^{2}}.\label{eq:appendix-a8}
\end{equation}
Integrating \eqref{eq:appendix-a7} and \eqref{eq:appendix-a8} over
$(0,y]$, exploiting $|f|_{-1}\leq C(f)$ as $y\to0$ if $f\in\dot{\calH}_{1}^{2}$
is smooth at the origin (due to the density argument), and using $|\rd_{yy}f|\aleq|(\rd_{y}-\frac{1}{y})f|_{-1}$,
we have 
\[
\chf_{(0,R]}\frac{|f|^{2}}{y^{2}\langle\log y\rangle}\aleq\chf_{(0,R]}\int_{0}^{y}\bigg\{\Big|\Big(\rd_{y}-\frac{1}{y'}\Big)f\Big|_{-1}^{2}+\Big|\frac{|f|_{-1}}{y'\langle\log y'\rangle}\Big|^{2}\bigg\} y'dy'\aleq\|f\|_{(\dot{\calH}_{1}^{2})_{R}}^{2}.
\]
This completes the proof of \eqref{eq:local-infty-est1}. The proof
of \eqref{eq:local-infty-est2} is much easier and we omit the proof.
\end{proof}

\section{\label{sec:Proof-of-Corollary}Proof of Corollary~\ref{cor:RR-blow-up-self-similar}}
\begin{proof}[Proof of Corollary~\ref{cor:RR-blow-up-self-similar}]
We only show the estimates \eqref{eq:eRR-est1}-\eqref{eq:eRR-est3}
as all the other statements are immediate from Proposition~\ref{prop:RR-blow-up-original}.
For the sake of simplicity we write $\wh{\lmb}=\lmb$, $\wh b=b$,
and $R=4/b$ in the proof. We keep $M$-dependences in the following
estimates, but one can simply ignore them because $M$ is already
fixed in Proposition~\ref{prop:RR-blow-up-original} (see also Remark~\ref{rem:Remark-M}).
Notice that we stated \eqref{eq:eRR-est1}-\eqref{eq:eRR-est3} without
$M$-dependence.

\textbf{Step 1.} Proof of the estimates \eqref{eq:eRR-est1}-\eqref{eq:eRR-est3}
for $\eps^{\RR}$.

By definition, we have 
\begin{equation}
\left\{ \begin{aligned}\eps^{\RR}(\tau,y) & =w^{\RR}(t,\lmb(t)y),\\
A\eps^{\RR}(\tau,y) & =\lmb(t)W^{\RR}(t,\lmb(t)y),
\end{aligned}
\right.\label{eq:eRR-wRR}
\end{equation}
and this $\eps^{\RR}$ satisfies the orthogonality 
\begin{equation}
\langle\eps^{\RR},\chi_{M}\Lmb Q\rangle=0.\label{eq:eRR-orthog}
\end{equation}
Note that \eqref{eq:eRR-est1} for $\eps^{\RR}$ easily follows from
\eqref{eq:wRR-est-1} and scaling. To show the estimates \eqref{eq:eRR-est2}-\eqref{eq:eRR-est3}
for $\eps^{\RR}$, we first apply the coercivity estimates \eqref{eq:coercivity-1}
and \eqref{eq:loc-coercivity-1} to have 
\begin{align*}
\|\eps^{\RR}\|_{\dot{\calH}_{1}^{2}}^{2} & \sim_{M}\|A^{\ast}A\eps^{\RR}\|_{L^{2}}^{2},\\
\|\eps^{\RR}\|_{(\dot{\calH}_{1}^{2})_{R}}^{2} & \sim_{M}\|\chf_{(0,R]}A^{\ast}A\eps^{\RR}\|_{L^{2}}^{2}.
\end{align*}
An integration by parts shows that the RHS of the above display can
be written as 
\begin{align*}
\|A^{\ast}A\eps^{\RR}\|_{L^{2}}^{2} & =\int\Big(|\rd_{y}A\eps^{\RR}|^{2}+\frac{\td V}{y^{2}}|A\eps^{\RR}|^{2}\Big),\\
\|\chf_{(0,R]}A^{\ast}A\eps^{\RR}\|_{L^{2}}^{2} & =\int\chf_{(0,R]}\Big(|\rd_{y}A\eps^{\RR}|^{2}+\frac{\td V}{y^{2}}|A\eps^{\RR}|^{2}\Big)+\frac{2}{1+R^{2}}|A\eps^{\RR}|^{2}(R).
\end{align*}
By \eqref{eq:eRR-wRR} and \eqref{eq:wRR-est-2}-\eqref{eq:wRR-est-3},
we have 
\begin{align*}
\int\Big(|\rd_{y}(A\eps^{\RR})|^{2}+\frac{\td V}{y^{2}}|A\eps^{\RR}|^{2}\Big) & \aleq b^{4},\\
\int\chf_{(0,R]}\Big(|\rd_{y}(A\eps^{\RR})|^{2}+\frac{\td V}{y^{2}}|A\eps^{\RR}|^{2}\Big) & \aleq\frac{b^{4}}{|\log b|}.
\end{align*}
By the localized $L^{\infty}$-control \eqref{eq:local-infty-est1},
the boundary term is controlled by 
\[
\frac{2}{1+R^{2}}|A\eps^{\RR}|^{2}(R)\aleq\frac{\log R}{R^{2}}\cdot\frac{|\eps^{\RR}|_{-1}^{2}(R)}{\log R}\aleq o_{b\to0}(1)\cdot\|\eps^{\RR}\|_{(\dot{\calH}_{1}^{2})_{R}}^{2}.
\]
Therefore, we have shown that 
\begin{align*}
\|\eps^{\RR}\|_{\dot{\calH}_{1}^{2}}^{2} & \aleq_{M}b^{4},\\
\|\eps^{\RR}\|_{(\dot{\calH}_{1}^{2})_{R}}^{2} & \aleq_{M}\frac{b^{4}}{|\log b|}+o_{b\to0}(1)\cdot\|\eps^{\RR}\|_{(\dot{\calH}_{1}^{2})_{R}}^{2}.
\end{align*}
This completes the proof of \eqref{eq:eRR-est2}-\eqref{eq:eRR-est3}
for $\eps^{\RR}$.

\textbf{Step 2.} Proof of the estimates \eqref{eq:eRR-est1}-\eqref{eq:eRR-est3}
for $\dot{\eps}^{\RR}$.

By definition, we have 
\[
\dot{\eps}^{\RR}(\tau,y)=\lmb(t)[\rd_{t}w^{\RR}](t,\lmb(t)y)+\lmb(t)b_{t}(t)[\rd_{b}P^{\RR}](b(t);y).
\]

First, we show that \eqref{eq:eRR-est1}-\eqref{eq:eRR-est3} holds
for $\lmb b_{t}[\rd_{b}P^{\RR}](b;y)$ in place of $\dot{\eps}^{\RR}$.
Note that we know $|\lmb b_{t}|\aleq\frac{b^{2}}{|\log b|}$ by \eqref{eq:b_t-RR-est}.
For $\rd_{b}P^{\RR}$, one has the following pointwise estimate \cite[(3.61)]{RaphaelRodnianski2012Publ.Math.}:
\begin{align*}
|\rd_{b}P^{\RR}|_{2} & \aleq\chf_{(0,\frac{B_{0}}{2}]}\frac{b}{|\log b|}y\langle\log b(1+y)\rangle+\chf_{[\frac{B_{0}}{2},2B_{1}]}\frac{1}{b|\log b|y}\\
 & \quad+\chf_{[\frac{B_{1}}{2},2B_{1}]}\frac{1}{by}+C(M)\frac{by}{1+y^{2}}.
\end{align*}
Therefore, we have (recall that $B_{0}\sim b^{-1}$ and $B_{1}\sim b^{-1}|\log b|$)
\begin{align*}
\|\rd_{b}P^{\RR}\|_{L^{2}} & \aleq\frac{1}{b}+C(M)b|\log b|\aleq\frac{1}{b},\\
\|\rd_{b}P^{\RR}\|_{\dot{H}_{1}^{1}} & \aleq\frac{1}{|\log b|}+C(M)b\aleq\frac{1}{|\log b|}.
\end{align*}
Multiplying the above estimates by $|\lmb b_{t}|\aleq\frac{b^{2}}{|\log b|}$
completes the proof of \eqref{eq:eRR-est1}-\eqref{eq:eRR-est3} for
$\lmb b_{t}[\rd_{b}P^{\RR}](b;y)$.

It remains to show \eqref{eq:eRR-est1}-\eqref{eq:eRR-est3} for $\lmb(t)[\rd_{t}w^{\RR}](t,\lmb(t)y)\eqqcolon\td{\eps}(\tau,y)$
in place of $\dot{\eps}^{\RR}$. Note that \eqref{eq:eRR-est1} for
$\td{\eps}$ is immediate from \eqref{eq:wRR-est-1}. For the proof
of \eqref{eq:eRR-est2}-\eqref{eq:eRR-est3} for $\td{\eps}$, we
would like to use the coercivity estimates \eqref{eq:coercivity-2}
and \eqref{eq:loc-coercivity-2} as in the case of $\eps^{\RR}$,
but the orthogonality condition \eqref{eq:eRR-orthog} for $\td{\eps}$
does not hold. However, the orthogonality condition is almost satisfied
in the sense that 
\begin{equation}
\begin{aligned}\langle\td{\eps},\chi_{M}\Lmb Q\rangle & =\big\langle\rd_{t}w^{\RR},\frac{1}{\lmb}[\chi_{M}\Lmb Q]_{\lmb}\big\rangle\\
= & -\big\langle w^{\RR},\rd_{t}(\frac{1}{\lmb}[\chi_{M}\Lmb Q]_{\lmb})\big\rangle=-b\langle\eps^{\RR},\Lmb_{0}(\chi_{M}\Lmb Q)\rangle\aleq_{M}\frac{b^{3}}{|\log b|},
\end{aligned}
\label{eq:almost-orthog}
\end{equation}
where in the last inequality we used \eqref{eq:eRR-est3} for $\eps^{\RR}$.
Thus \eqref{eq:coercivity-2} and \eqref{eq:loc-coercivity-2} become
\begin{align*}
\|\td{\eps}\|_{\dot{H}_{1}^{1}} & \aleq_{M}\|A\td{\eps}\|_{L^{2}}+|\langle\td{\eps},\chi_{M}\Lmb Q\rangle|\aleq_{M}\|A\td{\eps}\|_{L^{2}}+\frac{b^{3}}{|\log b|},\\
\|\td{\eps}\|_{(\dot{H}_{1}^{1})_{R}} & \aleq_{M}\|\chf_{(0,R]}A\td{\eps}\|_{L^{2}}+|\langle\td{\eps},\chi_{M}\Lmb Q\rangle|\aleq_{M}\|\chf_{(0,R]}A\td{\eps}\|_{L^{2}}+\frac{b^{3}}{|\log b|}.
\end{align*}
Next, we use the identity 
\[
A_{\lmb}\rd_{t}w=\rd_{t}W+[A_{\lmb},\rd_{t}]w=\rd_{t}W-\frac{b}{\lmb^{2}}\Big[\frac{4y}{(1+y^{2})^{2}}\Big]_{\lmb}w
\]
to have 
\begin{align*}
\|A\td{\eps}\|_{L^{2}} & =\lmb\|A_{\lmb}\rd_{t}w\|_{L^{2}}\aleq\lmb\|\rd_{t}W\|_{L^{2}}+b\Big\|\frac{4y}{(1+y^{2})^{2}}\eps^{\RR}\Big\|_{L^{2}},\\
\|\chf_{(0,R]}A\td{\eps}\|_{L^{2}} & =\lmb\|\chf_{(0,\lmb R]}A_{\lmb}\rd_{t}w\|_{L^{2}}\aleq\lmb\|\chf_{(0,\lmb R]}\rd_{t}W\|_{L^{2}}+b\Big\|\chf_{(0,R]}\frac{4y}{(1+y^{2})^{2}}\eps^{\RR}\Big\|_{L^{2}}.
\end{align*}
Applying \eqref{eq:wRR-est-2}-\eqref{eq:wRR-est-3} for $\rd_{t}W$
and \eqref{eq:eRR-est2}-\eqref{eq:eRR-est3} for $\eps^{\RR}$, we
obtain 
\[
\|A\td{\eps}\|_{L^{2}}\aleq b^{2}\qquad\text{and}\qquad\|\chf_{(0,R]}A\td{\eps}\|_{L^{2}}\aleq\frac{b^{2}}{|\log b|}.
\]
Therefore, we conclude that 
\begin{align*}
\|\td{\eps}\|_{\dot{H}_{1}^{1}} & \aleq_{M}\|A\td{\eps}\|_{L^{2}}+\frac{b^{3}}{|\log b|}\aleq_{M}b^{2},\\
\|\td{\eps}\|_{(\dot{H}_{1}^{1})_{R}} & \aleq_{M}\|\chf_{(0,R]}A\td{\eps}\|_{L^{2}}+\frac{b^{3}}{|\log b|}\aleq_{M}\frac{b^{2}}{|\log b|}.
\end{align*}
This completes the proof of \eqref{eq:eRR-est2}-\eqref{eq:eRR-est3}
for $\td{\eps}$. The proof is complete.
\end{proof}

\bibliographystyle{abbrv}
\bibliography{References}

\end{document}